\documentclass[12pt,a4paper]{amsart}

\usepackage[utf8]{inputenc}
\usepackage[english]{babel}
\usepackage{amsmath,amsfonts,amssymb,amsthm,mathtools,mathrsfs}

\usepackage[backref,ocgcolorlinks,linkcolor=blue]{hyperref}
\usepackage[nameinlink]{cleveref}
\usepackage[dvipsnames]{xcolor}
\usepackage{enumitem}

\usepackage[normalem]{ulem}
\usepackage{cancel}
\usepackage[textsize=tiny]{todonotes}

\newtheorem{main-theorem}{Theorem}
\newtheorem{proposition}{Proposition}[section]
\newtheorem{corollary}[proposition]{Corollary}
\newtheorem{lemma}[proposition]{Lemma}
\newtheorem{theorem}[proposition]{Theorem}
\newtheorem{definition}[proposition]{Definition}
\theoremstyle{remark}
\newtheorem{remark}[proposition]{Remark}
\newtheorem*{acknowledgements}{Acknowledgements}

\def\R{\mathbb R} 
\def\C{\mathbb C} 
\def\N{\mathbb N} 
\def\dd{\mathrm d} 
\def\i{\mathrm i} 
\def\SO{\mathrm{SO}} 
\def\GL{\mathrm{GL}} 
\def\M{\mathrm{M}} 
\def\O{\mathrm{O}} 
\def\U{\mathrm{U}} 
\def\SS{\mathbb{S}} 
\def\id{\mathrm{id}} 
\def\Id{\mathrm{Id}} 
\def\T{\mathsf{T}} 

\DeclareMathOperator*{\vspan}{span} 
\DeclareMathOperator*{\supp}{supp} 
\DeclareMathOperator*{\tr}{tr} 

\author[P. Caro]{Pedro  Caro}
\address[PC]{BCAM - Basque Center of Applied Mathematics, Bilbao, Spain.}
\email{pcaro@bcamath.org}

\author[C. J. Meroño]{Cristóbal J. Meroño}
\address[CJM]{Universidad Politécnica de Madrid, ETSI Caminos, Departmento de Matemática e Informática, Madrid, Spain.}
\email{cj.merono@upm.es}

\author[I. Parissis]{Ioannis Parissis}
\address[IP]{Departamento de Matem\'aticas, Universidad del Pais Vasco, Aptdo. 644, 48080 Bilbao, Spain and Ikerbasque, Basque Foundation for Science, Bilbao, Spain}
\email{ioannis.parissis@ehu.es}

\title[Rotational smoothing]{Rotational smoothing}

\date{\today}

\begin{document}

\begin{abstract}
Rotational smoothing is a phenomenon consisting in a gain of regularity by means of {averaging over} rotations. This phenomenon is present in operators that regularize only in certain directions, in contrast to operators regularizing in all directions. The gain of regularity is the result of rotating the directions where the corresponding operator performs the smoothing effect. In this paper we carry out a systematic study of the rotational smoothing  for a class of operators that includes $k$-vector-space Riesz potentials in $\R^n$ with $k < n$, and the convolution with fundamental solutions of elliptic constant-coefficient differential operators acting on $k$-dimensional linear subspaces. Examples of the {latter} type of operators are the planar Cauchy transform in $\R^n$, or a solution operator for the transport equation in $\R^n$. The analysis of rotational smoothing is motivated by the resolution of some inverse problems under low-regularity assumptions.
\end{abstract}

\maketitle

\section{Introduction} In this paper we introduce the concept of rotational smoothing to refer to a phenomenon consisting in the gain of regularity by means of {averaging over} rotations. In order to illustrate this phenomenon, let us consider an elliptic homogeneous polynomial of order $m \in \N$ with $k \in \N$ variables and complex coefficients
\begin{equation}\label{id:polynomial}
p(\xi) = \sum_{|\alpha| = m} c_\alpha \xi^\alpha.
\end{equation}
Here $c_\alpha \in \C$, $\xi = (\xi_1, \dots, \xi_k)$, $\alpha = (\alpha_1, \dots, \alpha_k) \in \N_0^k$ is a multi-index so that $|\alpha| = \alpha_1 +\dots + \alpha_k$ and $\xi^\alpha = \xi_1^{\alpha_1} \dots \xi_k^{\alpha_k}$. By calling this homogeneous polynomial elliptic, we mean that if $p(\xi) = 0$ for $\xi \in \R^k$ then $\xi = 0$. It is well known that the differential operator 
\begin{equation}
P = {p(-\i\partial)} =\sum_{|\alpha| = m} (-\i)^m c_\alpha \partial^\alpha
\label{id:elliptic_operator}
\end{equation}
associated to \eqref{id:polynomial} ---by analogy $\partial^\alpha = \partial_{x_1}^{\alpha_1} \dots \partial_{x_k}^{\alpha_k}$--- admits a fundamental solution $E$ in $\R^k$ so that $u = E \star f $ is a solution of the equation $Pu = f$ in $\R^k$ and
\begin{equation}
\| u \|_{\dot{H}^m(\R^k)} \lesssim \| f \|_{L^2(\R^k)}
\label{in:regularity_gain}
\end{equation}
for all $f \in L^2(\R^k)$.\footnote{Throughout the paper we write $a \lesssim b$ or equivalently $b \gtrsim a$, when $a$ and $b$ are positive constants and there exists $C > 0$ so that $a \leq C b$. We refer to $C$ as the implicit constant. Additionally, if $a \lesssim b$ and $b \lesssim a$, we write $a \eqsim b$.} Here and throughout the article the semi-norm of the homogeneous Sobolev space $\dot{H}^s (\R^n)$ with $s \in \R$ and $n \in \N$ is 
given by
\[
\| \phi \|_{\dot{H}^s(\R^n)} = \Big( \int_{\R^n} |\xi|^{2 s} |\widehat{\phi}(\xi)|^2 \, \dd \xi \Big)^{1/2},
\]
where 
$\widehat{\phi}(\xi) = (2\pi)^{-n/2} \int_{\R^n} e^{-\i\xi \cdot x} \phi(x) \, \dd x$ 
denotes the Fourier transform of $\phi \in \mathcal{S}(\R^n)$ in the Schwartz 
class. {Note that inequality \eqref{in:regularity_gain} entails} a gain of $m$-derivatives for $u$ with respect to the regularity of $f$. 

Suppose now that $f$ depends on $n$ variables $x = (x_1, \dots, x_n)$ and $P$ still denotes the differential operator \eqref{id:elliptic_operator} of $k$ variables $x^\prime = (x_1, \dots, x_k)$ with $n > k$. If $u$ denotes the solution of $P u = f$ in $\R^n$ given by $u = \tilde{E} \star f$ with $\tilde{E} = E \otimes \delta_0$ and $\delta_0$ denoting the Dirac distribution in $\R^{n - k}$ supported at the origin, then we can only expect that
\[
\Big( \int_{\R^{n - k}} \| u (\centerdot, x^{\prime\prime}) \|_{\dot{H}^m(\R^k)}^2 \dd x^{\prime\prime} \Big)^{1/2} \lesssim \| f \|_{L^2(\R^n)}.
\]
This means that the only possible regularity gain happens on the first $k$ 
variables and definitely not on the latter $n - k$ ones. This is due to the fact 
that the convolution with $\tilde{E}$ in $x = (x^\prime, x^{\prime\prime})$ has 
the component $E$ which regularizes in $x^\prime$ and the component $\delta_0$ 
which acts as the identity on the $x^{\prime\prime}$. 

After this discussion the question of rotational smoothing can be set {up} as 
follows. Consider the family of polynomials of $n$ variables 
$\{\tilde{p}_Q(\xi) : Q \in \SO(n) \}$ given by
\begin{equation}
\tilde{p}_Q(\xi) = p (Q e_1 \cdot \xi, \dots, Q e_k \cdot \xi)
\label{id:rotated_polynomial}
\end{equation}
with $p$ the polynomial function in \eqref{id:polynomial}. Here and throughout the paper $\{e_1,\ldots,e_n\}$ denotes the standard basis of $\R^n$ and
\[
\SO(n) = \{ Q \in \GL(n) : \, Q^\T Q = \id, \, \det Q = 1 \},
\]
where $\GL(n)$ is the general linear group of $n\times n$ real matrices, $Q^\T$ 
is the transpose of $Q$ and $\id$ stands for the identity matrix.
Note that $\tilde{p}_Q(\xi)$ is {not} elliptic any more since it vanishes on 
linear subspaces of codimension $k$.
We continue by writing $\tilde{P}(Q)$ {for} the 
differential operator associated to the polynomial $\tilde{p}_Q(\xi)$, and {let} 
$u(\centerdot,Q)$ be the solution of the equation 
$\tilde{P}(Q) u(\centerdot, Q) = f$ in $\R^n$ given by
\[
u(x, Q) = ([Q^\T]^\ast \tilde{E} \star f)(x),\qquad (x, Q) \in \R^n \times \SO(n),
\] 
where $[Q^\T]^\ast \tilde{E}$ denotes the pull-back of $\tilde{E}$ by the {map} 
$ x\in \R^n \mapsto Q^\T x \in \R^n$: {namely the pull-back acts as 
$[Q^\T]^\ast \phi (x) = \phi (Q^\T x)$ for all $\phi \in \mathcal{S}(\R^n)$ and 
$x\in \R^n$.} Then, the rotational smoothing phenomenon consists of a gain of 
$m$-derivatives in all the variables when performing an average {over 
rotations $Q\in \SO(n)$.} To put in quantitative terms, we ask whether it is 
possible to have an inequality of the following type
\[
\bigg( \int_{\SO(n)} \| u (\centerdot, Q) \|_{\dot{H}^m(\R^n)}^2 \dd \mu(Q) \bigg)^{1/2} \lesssim \| f \|_{L^2(\R^n)}
\]
for all $f \in L^2(\R^n)$, where $\mu$ stands for the normalized Haar measure on the special orthogonal group $\SO(n)$.

\subsection{Main results}
In this article, we carry out a systematic study of the phenomenon of rotational 
smoothing for a general class of operators. 
These operators are defined by considering, for each 
$\alpha\in\R_+=(0,\infty)$, a continuous function $q_\alpha:\R^k\to\C$ 
satisfying the following properties:
\begin{enumerate}[label=(\alph*), ref=\textnormal{\alph*}]
\item \label{it:ellipticity} If $q_\alpha (\eta) = 0$ then $\eta = 0$.
\item \label{it:homogeneity} The equality $ q_\alpha (\lambda \eta) = \lambda^\alpha q_\alpha (\eta)$ 
holds whenever $\lambda \in \R_+$ and $\eta \in \R^k \setminus \{ 0 \}$.
\end{enumerate}
Then, define the symbols
\[ 
p_\alpha (\xi, Q) = q_\alpha(Q e_1 \cdot \xi, \dots, Q e_k \cdot \xi),\qquad  (\xi, Q) \in \R^n \times \SO(n). 
\]
Let $P_\alpha(Q)$ be the operator
defined by
\[
[P_\alpha (Q) f] (x) = \frac{1}{(2 \pi)^{n/2}} \int_{\R^n} e^{\i x \cdot \xi} \, p_\alpha (\xi, Q) \, \widehat{f}(\xi) \dd \xi,\quad  (x, Q) \in \R^n \times \SO(n),
\]
for $f \in \mathcal{S}(\R^n)$. 
{We ask whether we can provide a solution operator $S_\alpha$ so that if
$u(x, Q) = S_\alpha f(x, Q) $, then $P_\alpha (Q) u (\centerdot, Q) = f$ in 
$\R^n$ for all $Q \in \SO(n)$ and 
\begin{equation}
\label{in:rot-smooth_motivation}
\bigg( \int_{\SO(n)} \| S_\alpha f(\centerdot, Q) \|_{\dot{H}^\alpha(\R^n)}^2 \dd \mu(Q) \bigg)^{1/2} \lesssim \| f \|_{L^2(\R^n)}
\end{equation}
for all $f \in L^2(\R^n)$.} Note {again} that these operators are not elliptic since their symbols vanish on linear subspaces of codimension $k$.

In order to clarify how this framework generalizes the previous discussion about 
differential operators, it may be convenient to note that if we choose $q_m = p$ 
with $p$ the polynomial function in \eqref{id:polynomial}, then $p_m(\xi, Q) $ is 
the polynomial $ \tilde{p}_Q(\xi)$ in \eqref{id:rotated_polynomial} and $P_m(Q) $ 
is its corresponding differential operator $\tilde{P}(Q)$. The generalization 
described above allows us to also introduce {non-local operators  in our analysis, 
such} as the fractional Laplacian acting on $k$-dimensional vector subspaces of 
$\R^n$. Note that if $q_\alpha(\eta) = |\eta|^\alpha$ for all
$\eta \in \R^k$, then $P_\alpha (\id) = (-\Delta_{x^\prime})^{\alpha/2}$, where 
$\Delta_{x^\prime}$ denotes $\partial_{x_1}^2 + \dots + \partial_{x_k}^2$.

Before stating the result quantifying the rotational smoothing phenomenon, we 
need to complement the conditions
\eqref{it:ellipticity} and \eqref{it:homogeneity} with an additional cancellation 
assumption, whenever $\alpha - k \in \N_0$:
\begin{enumerate}[label=(\alph*),ref=\alph*]
\setcounter{enumi}{2}
\item \label{it:cancellation} If $\alpha - k \in \N_0$, then $q_\alpha$ satisfies 
that
\[
\int_{\SS^{k - 1}} \frac{\theta^\beta}{q_\alpha(\theta)} \, \dd S(\theta) = 0, \qquad \forall \beta \in \N_0^k, \quad |\beta| = \alpha - k,
\]
\end{enumerate}
where the volume form {$\dd S$ denotes the usual spherical measure on $\SS^{k- 1}$.}  {In the case $k = 1$ we have that $\SS^0 = \{ - 1, 1\}$, and \eqref{it:cancellation} is interpreted as
\[
 \int_{\SS^{k - 1}} \frac{\theta^\beta}{q_\alpha(\theta)} \, \dd S(\theta) = \frac{(-1)^{\alpha - 1}}{q_\alpha (-1)} + \frac{1}{q_\alpha (1)}=0.
\]}

This discussion presents the bare minimum framework and motivation to state our 
main result. {From now on we set}
\[
\lceil \alpha - k \rceil = \min \{ l \in \N_0 : \alpha - k < l \}.
\]

\begin{main-theorem}\label{th:rotational_smoothing}\sl
Assume $n \in \N$ with $n \geq 2$. Let $k \in \{1, \dots, n - 1\}$ and $\alpha > 0$ be given,  and consider $q_\alpha \in C(\R^k)$ satisfying  \eqref{it:ellipticity}, \eqref{it:homogeneity} and \eqref{it:cancellation}. There exists a solution operator $S_\alpha$ and a constant $C > 0$ depending on $n$, $k$, $\alpha$ and $\delta$ such that, for all $f \in \mathcal{S}(\R^n)$ we have 
that
\[ 
\| S_\alpha f \|_{H^{-\delta} (\SO(n); \dot{H}^\alpha(\R^n))} = C \| f\|_{L^2(\R^n)}
\]
with $\delta = \lceil \alpha - k \rceil$ if $\alpha - k/ 2 < \lceil \alpha - k \rceil$, and $\delta > \alpha - k/2$ if $\alpha - k/ 2 \geq \lceil \alpha - k \rceil$.
\end{main-theorem}

Note that in the range $0 < \alpha < k/2$, we can choose $\delta = 0$ so that the
identity in \Cref{th:rotational_smoothing} implies the inequality 
\eqref{in:rot-smooth_motivation}. Unfortunately, whenever $\alpha \geq k/2$
we have to replace the $L^2$-average on the rotations by a weaker quantity 
captured by the $H^{-\delta}$-norm on $\SO(n)$. Nevertheless, this {result} is 
satisfactory in terms of the regularity gain because {the} $-\delta$ derivatives 
of the Sobolev space in $\SO(n)$ {act directly} on the operator and not on the 
functions of its domain. Thus, we still obtain a gain of $\alpha$ derivatives in
$\R^n$.

At this point, the statement of the result is vague in the sense that it is not 
clear if the solution operator $S_\alpha$ can be explicitly defined. 
Additionally, the restrictions on $\delta$ imposed by the 
relations between $\alpha$ and $k$ may seem mysterious. For this reason, we 
present in \Cref{sec:precise_statement} a more precise version of 
\Cref{th:rotational_smoothing} where the definition of $S_\alpha$ and the 
constant $C$ are explicit. The discussion in 
\Cref{sec:precise_statement} also clarifies the restrictions on $\delta$.
For the case $0 < \alpha < k$, the definition of $S_\alpha$
is very simple and can be found in \eqref{id:solution_map_multiplier}.

In order to analyse the need of replacing the 
$L^2$-average on the rotations by the 
$H^{-\delta}$-norm on $\SO(n)$, we choose  $q_\alpha (\eta) = |\eta|^\alpha$ for 
all $\eta \in \R^k$ with $0<\alpha<k$ and define the solution map, denoted in 
this case by $I_\alpha$, as 
\begin{equation*}
I_\alpha  f(x,Q) = \frac{1}{(2 \pi)^{n/2}} \int_{\R^n} e^{\i x \cdot \xi} \, \frac{\widehat{f}(\xi)}{|(Q e_1 \cdot \xi)^2 + \dots + (Q e_k \cdot \xi)^2|^{\alpha/2}} \, \dd \xi, \qquad (x, Q) \in \R^n \times \SO(n),
\end{equation*}
for $f \in \mathcal{S}(\R^n)$.
The solution operator $I_\alpha$ is the $k$-vector-space Riesz potential, and its 
definition is possible because the function 
$$ \xi\in \R^n \setminus \Sigma_Q \longmapsto |(Q e_1 \cdot \xi)^2 + \dots + (Q e_k \cdot \xi)^2|^{-\alpha/2}, $$ 
with
$\Sigma_Q = \{  \xi \in \R^n : Q e_1 \cdot \xi = \dots = Q e_k \cdot \xi = 0 \}$,
can be extended to a locally-integrable function in $\R^n$. However, 
if $k/2 \leq \alpha < k$ and $f \in \mathcal{S}(\R^n)$ satisfies that 
$\widehat{f}(\xi) \neq 0 $ whenever $ \xi \in \Sigma_Q$ then
\[ \| I_\alpha f  (\centerdot, Q) \|_{\dot{H}^\alpha (\R^n)}^2 = \int_{\R^n} \frac{|\xi|^{2\alpha}}{|(Q e_1 \cdot \xi)^2 + \dots + (Q e_k \cdot \xi)^2|^{\alpha}} |\widehat{f}(\xi)|^2 \, \dd \xi = \infty, \qquad  \forall\, Q \in \SO(n). \]
This follows from the fact that the function
$$\xi\in \R^n \setminus \Sigma_Q \mapsto |\xi|^\alpha/|(Q e_1 \cdot \xi)^2 + \dots + (Q e_k \cdot \xi)^2|^{\alpha/2} $$
cannot be extended to a locally-square-integrable function in $\R^n$. 
We remark that the first identity above is a consequence of Plancherel's theorem. Thus, an $L^2$-average version of rotational smoothing in the 
range $k/2 \leq \alpha < k$ cannot be global. A possible way around this is to 
postulate a local version of the rotational smoothing estimate where the 
left-hand side of the identity in \Cref{th:rotational_smoothing} would be 
replaced, in the case of the $k$-vector-space Riesz potential, by
\[
\bigg( \int_{\SO(n)} \| (-\Delta)^{\alpha/2} I_\alpha  f(\centerdot,Q) \|^2_{L^2 (B_R)} \, \dd \mu(Q) \bigg)^{1/2}.
\]
Here $(-\Delta)^{\alpha/2}$ denotes the fractional Laplacian in $\R^n$ and  
$B_R = \{ x \in \R^n : |x| < R \}$. However, there is no hope even for an 
inequality of this type without loss of derivatives. 
In fact, we will see in the following result that, when localizing in balls and 
then averaging in $\SO(n)$, the $k$-vector-space Riesz potential exhibits a loss 
of derivatives in the range $ k/2 \leq \alpha < k $.

\begin{main-theorem}\label{th:loss_of_derivatives}\sl Assume $n \in \N$ with $n \geq 2$. Let $k \in \{1, \dots, n - 1\}$ and $ k/2 \leq \alpha < k $ be given. For every $R > 0$ there exists a sequence of functions $\{ f_N : N \in \N \}$ such that $\| f_N \|_{\dot{H}^s(\R^n)} \eqsim N^s $ with $s \in \R$ and
\[
\int_{\SO(n)} \| (-\Delta)^{\alpha/2} I_\alpha  f_N(\centerdot,Q) \|^2_{L^2 (B_R)} \, \dd \mu(Q) \gtrsim \mathfrak{g} (RN) 
\]
for all $N > \pi/(2 R)$. The growth function $\mathfrak{g}$ is given by $\mathfrak{g} (t) = \log t$ if $\alpha = k/2$ and $\mathfrak{g} (t) = t^{2 \alpha -k}$ if $\alpha > k/2$. The implicit constants are independent of $R$ and $N$.
\end{main-theorem}

This theorem ensures a loss of at least $s$ derivatives with $s = \alpha - k/2 $ 
if $\alpha > k/2$ and $s > 0 $ if $\alpha = k/2$, {compared to the maximum 
expected gain of $\alpha$ derivatives.}  The proof of 
{\Cref{th:loss_of_derivatives}} is given in \Cref{sec:counterexamples}.
Putting together Theorems \ref{th:rotational_smoothing} and 
\ref{th:loss_of_derivatives}
we see that the $s$ uncontrolled derivatives of {\Cref{th:loss_of_derivatives}} 
are regularized by $-\delta$ derivatives in $\SO(n)$ of 
\Cref{th:rotational_smoothing} ---note that $\delta\geq s$ in the range 
$k > \alpha \geq k/2$.

The homogeneity assumption \eqref{it:homogeneity} for the function $q_\alpha$ 
plays a key role in the proof of \Cref{th:rotational_smoothing}. This property is 
immediately transferred to the distribution in $\R^k \setminus \{ 0 \}$ 
determined by the function 
$ \eta \in \R^k \setminus \{ 0 \} \mapsto 1/q_\alpha(\eta) $.
To construct the solution operator $S_\alpha$, we need to extend 
this distribution to belong to
$\mathcal{S}^\prime (\R^k)$. However, in order for the extension to preserve the 
homogeneity in the case $\alpha - k \in \N_0$, we need to impose the 
cancellation condition \eqref{it:cancellation}.
This is analysed in \Cref{sec:extension}. Furthermore, as derived from the 
following result, this cancellation property is needed not only to preserve 
the homogeneity but also to guarantee that the rotational smoothing phenomenon 
yields the maximum gain of $\alpha$ derivatives.

\begin{main-theorem}\label{th:cancellation_counterexample}\sl
Assume $n \in \N$ with $n \geq 2$. Let 
$k \in \{1, \dots, n - 1\}$ and $\alpha > 0$ be such that $\alpha - k \in \N_0$.
Consider $q_\alpha \in C(\R^k)$ satisfying 
\eqref{it:ellipticity}, \eqref{it:homogeneity}, but assume that
\eqref{it:cancellation} does not hold. 
Then, there exists a 
sequence of functions $\{ f_N \in \mathcal S(\R^n) : N \in \N \}$ satisfying that $\| f_N \|_{L^2(\R^n)} =1$  and a number $N_0 \in \N$ such that
\[ \| S_\alpha f_N \|_{H^{-\delta} (\SO(n); \dot{H}^\alpha(\R^n))}  \gtrsim \log N,\]
for all $N \ge N_0$.
\end{main-theorem}

This theorem ensures a logarithmic loss of derivatives with respect to the 
maximum expected gain of $\alpha$ derivatives.  The proof of 
{\Cref{th:cancellation_counterexample}} is given in {\Cref{sec:counterexamples}}.

\subsection{Rotational smoothing and inverse problems}

The study of rotational smoothing is motivated by certain inverse boundary value problems in a non-regular setting. In order to give the relevant examples that drive the research reported in this paper,  we consider the inverse problem consisting in  determining the {advection term} in the {convection equation} from non-invasive measurements. One can consider two different physical situations when formulating this problem. 

In a first approach, we can assume the advection to be described by a static vector field $A$ and data to be obtained from boundary measurements of solutions of the steady-state {convection equation} $\Delta u + A \cdot \nabla u = 0$ in a bounded open subset $D \subset \R^n$ with $n \geq 3$. Usually, boundary data is encoded in the Dirichlet-to-Neumann map defined by
\[
\Lambda_A : f \mapsto \nu \cdot \nabla u |_{\partial D} 
\]
where $\partial D$ denotes the boundary of $D$, $\nu$ is the outward unit normal vector along $\partial D$, and $u$ is the solution of the convection equation in $D$ such that $u|_{\partial D} = f$ ---for more details on the formulation of this inverse boundary value problem see \cite{zbMATH02237598,zbMATH02123614,zbMATH06460406}. The only currently available procedure for determining $A$ from $\Lambda_A$ proceeds by constructing some special solutions for the differential operator  $\Delta + A \cdot \nabla$. These solutions are called \emph{complex geometrical optics} {and they are expressed as follows}
\[
u_\zeta (x) = e^{\zeta \cdot x} (v_\zeta(x) + w_\zeta(x)), 
\] 
where $\zeta \in \C^n$ is such that $\zeta \cdot \zeta = 0$; equivalently $|\Re \zeta| = |\Im \zeta|$ and $\Re \zeta \cdot \Im \zeta = 0$ with $\Re \zeta$ and $\Im \zeta$ denoting the real and imaginary parts of $\zeta$ respectively. The \emph{amplitude} $v_\zeta$ and the \emph{correction term} $w_\zeta$ are constructed as solutions of the equations
\begin{equation}
\label{eq:pre_complex_transport}
2 \zeta \cdot \nabla v_\zeta + \zeta \cdot A v_\zeta = 0,
\end{equation}
and
\begin{equation}
\label{eq:reminder}
(\Delta + 2 \zeta \cdot \nabla + A \cdot \nabla + \zeta \cdot A) w_\zeta = - (\Delta + A \cdot \nabla) v_\zeta.
\end{equation}
In order to solve \eqref{eq:pre_complex_transport} we choose $\zeta = \tau Q (e_1 + i e_2)$ with $Q \in \SO(n)$ and the amplitude in the form $v_\zeta = e^{\psi_\zeta}$, with $\psi_\zeta$ solving
\begin{equation}
\label{eq:complex_transport}
Q (e_1 + i e_2) \cdot \nabla \psi_\zeta = - \frac{1}{2} Q (e_1 + i e_2) \cdot A.
\end{equation}
As we see from \eqref{eq:reminder}, the regularity of $v_\zeta$ is important in 
order to ensure that the right-hand side  $- (\Delta + A \cdot \nabla) v_\zeta$ 
is in an appropriate space so that we can solve for $w_\zeta$. In the case that 
the components of $A$ {are just} in  $L^\infty (D)$, the a priori 
regularity of solutions of the equation \eqref{eq:complex_transport} is not 
{sufficient}. For this reason, one needs to look  for a hidden smoothing 
phenomenon, namely rotational smoothing. Note that {solutions of} equation 
\eqref{eq:complex_transport} fit in the framework {of
\eqref{id:solution_map_multiplier},} with $\alpha = 1$, $k = 2$ and 
{$q_1 (\eta) = i \eta_1 - \eta_2$} for all 
$\eta \in \R^2$, {so that 
\[
p_1(\xi,Q)= q_1(Qe_1\cdot\xi,Qe_2\cdot\xi)=\xi \cdot Q(ie_1 - e_2).
\]
{Indeed the solutions of \eqref{eq:complex_transport} }are given by the Cauchy transform {along planes} in $\R^n$ with $n \geq 3$: for $f \in \mathcal{S}(\R^n)$ and $(x,Q)\in \R^n \times \SO(n)$,}
\[
\mathcal{C}  f(x,Q) =\frac{1}{(2\pi)^{\frac{n}{2}}}\int_{\R^n} e^{ix\cdot \xi} \frac{\hat f(\xi)}{\xi \cdot Q(ie_1 - e_2)} \dd\xi.
\]

The second physical situation where one can consider this inverse problem corresponds to allowing advection terms modelled by time-dependent vector fields. Under these conditions, one has to consider the equation $\partial_t u -\Delta u + A \cdot \nabla u = 0$ in $D$ throughout the time interval $(0, T)$. For a full description of the problem, see \cite{caro2019determination}. For the resolution of this inverse problem, one {constructs} geometric optics
\[
u_\kappa (t, x) = e^{|\kappa|^2 t + \kappa \cdot x} (v_\kappa(t, x) + w_\kappa (t, x)), 
\] 
where $\kappa \in \R^n$. The \emph{amplitude} $v_\kappa$ and the \emph{correction 
term} $w_\kappa$ are constructed as solutions of the equations
\begin{equation}
\label{eq:pre_transport}
2 \kappa \cdot \nabla v_\kappa - \kappa \cdot A v_\kappa = 0
\end{equation}
and
\[
(\partial_t -\Delta - 2 \kappa \cdot \nabla + A \cdot \nabla + \kappa \cdot A) w_\kappa =  (- \partial_t + \Delta - A \cdot \nabla) v_\kappa.
\]
In order to solve \eqref{eq:pre_transport} we choose $\kappa = \tau Q e_1 $ with 
$Q \in \SO(n)$, and the amplitude in the form $v_\kappa = e^{\psi_\kappa}$ with 
$\psi_\kappa$ solving
\begin{equation}\label{eq:transport}
Q e_1 \cdot \nabla \psi_\kappa =  \frac{1}{2} Q e_1 \cdot A.
\end{equation}
In the case that the components of $A$ are {just} in 
$L^\infty ((0, T) \times D)$, one needs again to look for a rotational smoothing 
phenomenon. Note that equation \eqref{eq:transport} fits {the framework of 
Theorem~\ref{th:rotational_smoothing}} with $\alpha = k = 1$ and 
$q_\alpha (\eta) = i \eta $ for all $\eta \in \R$. Note that $\alpha - k = 0$ and 
it also satisfies the cancellation property 
$1/q_\alpha (-1) + 1/q_\alpha (1) = 0$.

We {expect} that the systematic analysis of rotational smoothing carried out in 
this article can be applied to the resolution of the previous and possibly other 
inverse problems in non-regular settings.

\subsection{Outline} The paper contains five more sections and an appendix. In 
\Cref{sec:precise_statement} we provide a more detailed formulation of 
\Cref{th:rotational_smoothing}. Additionally, we give a straightforward proof of 
the rotational smoothing phenomenon for the range $0 < \alpha < k/2$.   
\Cref{sec:extension}  contains the construction of the extension of the 
distribution determined by the function 
$ \eta \in \R^k \setminus \{ 0 \} \mapsto 1/q_\alpha(\eta) $.
In \Cref{sec:rotational_smoothing} we prove 
\Cref{th:rotational_smoothing} assuming some technical points that will be 
postponed. \Cref{sec:counterexamples} is devoted to prove Theorems 
\ref{th:loss_of_derivatives} and \ref{th:cancellation_counterexample}.
In \Cref{sec:abstract_framework} we provide the necessary analytical framework to 
cover the gaps assumed along the proof of \Cref{th:rotational_smoothing}. 
Lastly, in {\Cref{app:orthogonal_group}} we 
collect and prove some key properties about $\SO(n)$.

\begin{acknowledgements}
P.C. is supported by the grants BERC 2018-2021, Severo Ochoa SEV-2017-0718 and 
PGC2018-094528-B-I00, and the Ikerbasque foundation. He wants to thank D. Faraco 
and K. Rogers for some fruitful discussions on the topic some years ago. C.J.M. 
is supported by MTM2017-85934-C3-3-P. I.P. is supported by Ikerbasque, PGC2018-094528-B-I00 {(AEI/FEDER, UE)} and the Basque Government grant T1247-19. 
\end{acknowledgements}

\section{More precise statements of rotational smoothing}\label{sec:precise_statement}
In this section we restate \Cref{th:rotational_smoothing} providing an explicit 
definition of the solution operator $S_\alpha$ and computing the constant $C$. 
The statement is reformulated in two theorems, one corresponding the to the range $0 < \alpha < k$ and the other for $\alpha \geq k$.

Let $q_\alpha$ satisfy \eqref{it:ellipticity} 
and \eqref{it:homogeneity} with $0<\alpha<k$ and write
$p_\alpha (\xi, Q) = q_\alpha(Q e_1 \cdot \xi, \dots, Q e_k \cdot \xi)$ for all 
$(\xi, Q) \in \R^n \times \SO(n)$.
We define the solution map $S_\alpha$ for 
$f \in \mathcal{S}(\R^n)$ by 
\begin{equation}
S_\alpha f  (x, Q) =  \frac{1}{(2 \pi)^{n/2}} \int_{\R^n} e^{\i x \cdot \xi} \, \frac{\widehat{f}(\xi)}{p_\alpha (\xi, Q)} \, \dd \xi, \qquad (x, Q) \in \R^n \times \SO(n).
\label{id:solution_map_multiplier}
\end{equation}
This definition is possible because
the function 
$\xi\in \R^n \setminus \Sigma_Q \mapsto 1/p_\alpha(\xi, Q) $ with $\Sigma_Q = \{  \xi \in \R^n : Q e_1 \cdot \xi = \dots = Q e_k \cdot \xi = 0 \}$ 
can be extended to a locally-integrable function in $\R^n$ with temperate growth.
Let $r_\alpha$ denote the restriction to $\SS^{n - 1}$ of the function
\begin{equation}
\label{map:qalpha-1_function}
(\xi^\prime, \xi^{\prime\prime}) \in \R^k \times \R^{n - k} \longmapsto \left\{
\begin{aligned}
& 1/q_\alpha  (\xi^\prime) & & \textnormal{if } \xi^\prime \neq 0,\\
& 0 & & \textnormal{if } \xi^\prime = 0.
\end{aligned}
\right.
\end{equation}
Thus, the rotational smoothing phenomenon can be formulated as follows.
\begin{theorem} \label{th:0alphak} \sl 
Assume $n \in \N$ with $n \geq 2$. Let $k \in \{1, \dots, n - 1\}$ and 
$\alpha > 0$ be given. Consider $q_\alpha \in C(\R^k)$ satisfying 
\eqref{it:ellipticity} and \eqref{it:homogeneity} with $\alpha < k$.
Then, for all $f \in \mathcal{S}(\R^n)$ we have that
\begin{equation*}
\| S_\alpha f  \|_{H^{-\delta} (\SO(n); \dot{H}^\alpha(\R^n))} = S(\SS^{n - 1})^{-1/2} \| r_\alpha \|_{H^{-\delta}(\SS^{n - 1}) } \| f\|_{L^2(\R^n)}
\end{equation*}
with $\delta = 0$ if $\alpha < k/2$ and $\delta > \alpha - k/2$ if $ k/2 \leq \alpha < k$. Here $S(\SS^{n - 1})$ stands for the measure of the unit sphere.
\end{theorem}

\Cref{th:0alphak} be proved in \Cref{sec:proof-mainTH}.
This result is a more explicit formulation of \Cref{th:rotational_smoothing} for
the case $0 < \alpha < k$. From this formulation we see that the restriction on 
$\delta$ comes from the need of having $r_\alpha \in H^{-\delta} (\SS^{n - 1})$.
Notice also that this identity  shows explicitly that the $-\delta$ derivatives 
affect only $r_\alpha$, and hence just the operator, as we mentioned in the 
introduction.

Before proceeding with the task of providing a more explicit formulation of 
\Cref{th:rotational_smoothing} for the range $\alpha \geq k$, we will include 
here a straightforward proof of \Cref{th:0alphak} in the special case where 
$\alpha < k/2$.
That is, we prove that if $\alpha < k/2$ then
\begin{equation}
\| S_\alpha  f \|_{L^2(\SO(n);\dot{H}^\alpha (\R^n))} = S(\SS^{n - 1})^{-1/2} \| r_\alpha \|_{L^2(\SS^{n - 1})} \|f\|_{L^2(\R^n)}
\label{id:easy_sec}
\end{equation}
for all $f \in \mathcal{S}(\R^n)$.  
Despite the fact that the proof of the case $\alpha < k/2$ 
will be contained in the 
argument presented in the subsequent sections, the reader can use this more 
immediate proof as a roadmap for the general case.

\begin{proof}[Proof of the identity \eqref{id:easy_sec}.]
Using Plancherel's identity, together with the identities \eqref{id:solution_map_multiplier}, and switching the order of integration we have that
\[
\int_{\SO(n)} \| S_\alpha  f(\centerdot,Q) \|^2_{\dot{H}^\alpha (\R^n)} \, \dd \mu(Q) = \int_{\R^n} \|m_\alpha (\xi,\centerdot)\|^2_{L^2(\SO(n))} |\widehat{f}(\xi)|^2 \, \dd \xi,
\]
where
\begin{equation}
m_\alpha (\xi,Q) = \frac{|\xi|^\alpha}{p_\alpha(\xi, Q)}\qquad  (\xi,Q)\in \R^n \times \SO(n).
\label{id:mAlphaK}
\end{equation}
Let $M \in \SO(n)$ be such that $\xi = |\xi| M e_1 $; then
\[
m_\alpha (\xi,Q) = \frac{|\xi|^\alpha}{q_\alpha(|\xi| Q e_1 \cdot M e_1, \dots, |\xi| Q e_k \cdot M e_1)} = r_\alpha(Q^\T M e_1).
\]
In the last identity we have used the homogeneity of $q_\alpha$ given in \eqref{it:homogeneity} and the fact that at this point $r_\alpha$ denotes the restriction of \eqref{map:qalpha-1_function} to $\SS^{n - 1}$.
Therefore,
\[
\begin{split}
\| S_\alpha  f \|_{L^2(\SO(n);\dot{H}^\alpha (\R^n))} = \int_{\R^n} \bigg(\int_{\SO(n)} |r_\alpha (Q^\T M e_1)|^2 \dd \mu(Q) \bigg)  |\widehat{f}(\xi)|^2 \, \dd \xi.
\end{split}
\]
Since $Q^\T M = (M^\T Q)^\T$ and the Haar measure is left-invariant and invariant under transposition, see \Cref{prop:translation_transposed_invariance}, we have again by Plancherel's identity that
\[
\| S_\alpha  f \|_{L^2(\SO(n);\dot{H}^\alpha (\R^n))} = \bigg( \int_{\SO(n)} |r_\alpha (Q e_1)|^2 \dd \mu(Q) \bigg)^{1/2} \|f\|_{L^2(\R^n)}.
\]
Finally, by a very well-known relation between the Haar measure $\mu$ and the hypersurface measure on $\SS^{n - 1}$, see \Cref{prop:SOn_Sn-1}, we have that
\begin{equation} \label{id:int_haar}
 \bigg( \int_{\SO(n)} |r_\alpha (Q e_1)|^2 \dd \mu(Q) \bigg)^{1/2} = S(\SS^{n - 1})^{-1/2} \| r_\alpha \|_{L^2(\SS^{n - 1})}. 
 \end{equation}
This proves the identity \eqref{id:easy_sec} for $ \alpha < k/2$. The fact that 
$r_\alpha \in L^2(\SS^{n-1})$ when $ \alpha < k/2$ is derived from the discussion 
in \Cref{sec:ralph_S}.
\end{proof}

We continue with the task of providing a more explicit formulation of 
\Cref{th:rotational_smoothing} for the range $\alpha \geq k$. Let $q_\alpha$ 
satisfy \eqref{it:ellipticity}, \eqref{it:homogeneity} and 
\eqref{it:cancellation}. We pay attention only to the case $\alpha \geq k$, but 
note that the more general formulation presented in the following lines coincides 
with the previous one for the case that $\alpha < k$. In \Cref{sec:extension} we 
will see that there is an explicitly defined distribution 
$q_\alpha^{-1} \in \mathcal{S}^\prime (\R^k)$ that extends
the distribution
\begin{equation}
\label{map:ditributionRn-0}
 \phi \in \mathcal{D} (\R^k \setminus \{ 0 \}) \longmapsto \int_{\R^k \setminus \{ 0 \}} \frac{\phi(\eta)}{q_\alpha(\eta)} \, \dd \eta.
\end{equation}
With the tempered distribution $q_\alpha^{-1}$ at hand, we first introduce
$q_\alpha^{-1} \otimes \mathbf{1}_{\R^{n - k}}$ in $\R^n$ as the tensor product 
of  $q_\alpha^{-1} \in\mathcal{S}^\prime(\R^k)$ and the distribution 
corresponding to the function $\mathbf{1}_{\R^{n - k}}(\kappa) = 1$ for all 
$\kappa \in \R^{n - k}$. Note that this latter tempered distribution in $\R^n$ 
generalizes \eqref{map:qalpha-1_function} for $\alpha \geq k$.
In the same spirit as for the case 
$\alpha < k$, the distribution 
$q_\alpha^{-1} \otimes \mathbf{1}_{\R^{n - k}}$
will determine another distribution 
$r_\alpha$ in $\SS^{n - 1}$
such that $r_\alpha \in H^{-\delta} (\SS^{n - 1})$
with $\delta = \lceil \alpha - k \rceil$ if 
$\alpha - k/ 2 < \lceil \alpha - k \rceil$ and $\delta > \alpha - k/2$ 
if $\alpha - k/ 2 \geq \lceil \alpha - k \rceil$.
For convenience recall here that 
$$\lceil \alpha - k \rceil = \min \{ l \in \N_0 : \alpha - k < l \}.$$
Additionally, recall that in 
the range $\alpha < k$ the function $r_\alpha$ was chosen precisely as
the restriction of \eqref{map:qalpha-1_function} to $\SS^{n - 1}$.
By analogy, it is convenient to 
visualize $r_\alpha$ as a formal restriction of 
$q_\alpha^{-1} \otimes \mathbf{1}_{\R^{n - k}}$ to $\SS^{n - 1}$. 
Then, we consider the family of tempered distributions 
$\{ p_\alpha^{-1} (\centerdot, Q) \in \mathcal{S}^\prime (\R^n) :\, Q \in \SO(n) \}$ 
defined as
\[
p_\alpha^{-1} (\centerdot, Q) = [Q^\T]^\ast (q_\alpha^{-1} \otimes \mathbf{1}_{\R^{n - k}}), 
\]
and the solution operator defined for $f \in \mathcal{S}(\R^n)$ as
\begin{equation}
S_\alpha f (x, Q) = \frac{1}{(2 \pi)^{n/2}} \langle p_\alpha^{-1} (\centerdot, Q) ,\widehat{f (\centerdot + x)} \rangle, \quad (x, Q) \in \R^n \times \SO(n).
\label{id:solution_operator}
\end{equation}
Here the duality pairing {represents the action of the} tempered distribution $p_\alpha^{-1} (\centerdot, Q)$ on the Fourier transform of the function $y \in \R^n \mapsto f(y + x)$.

\begin{theorem}\label{th:kalpha}\sl
Assume $n \in \N$ with $n \geq 2$. Let $k \in \{1, \dots, n - 1\}$ and 
$\alpha \geq k$ be given, and consider $q_\alpha \in C(\R^k)$ satisfying  
\eqref{it:ellipticity}, \eqref{it:homogeneity} and \eqref{it:cancellation}. 
Define the solution operator $S_\alpha$ as in \eqref{id:solution_operator}. 
{Then, there exists a distribution  $r_\alpha \in H^{-\delta}(\SS^{n-1})$   
determined by $q_\alpha$ such that,} for all $f \in \mathcal{S}(\R^n)$ we have 
that
\[ 
\| S_\alpha f \|_{H^{-\delta} (\SO(n); \dot{H}^\alpha(\R^n))} = S(\SS^{n - 1})^{-1/2} \| r_\alpha \|_{H^{-\delta}(\SS^{n - 1}) } \| f\|_{L^2(\R^n)}
\]
with $\delta = \lceil \alpha - k \rceil$ if $\alpha - k/ 2 < \lceil \alpha - k \rceil$ and $\delta > \alpha - k/2$ if $\alpha - k/ 2 \geq \lceil \alpha - k \rceil$. Here $S(\SS^{n - 1})$ stands for the measure of the unit sphere.
\end{theorem}
We give an explicit formula for $r_\alpha$ in \Cref{prop:r_formula} for a 
specific choice of local coordinates on the sphere $\SS^{n - 1}$. 
\Cref{th:kalpha} be proved in \Cref{sec:proof-mainTH}.

\section{The tempered distribution \texorpdfstring{$q_\alpha^{-1}$}{qalpha-1}}
\label{sec:extension}
Here we discuss the extension of the distributions
\begin{equation}
\phi \in \mathcal{D} (\R^k \setminus \{ 0 \}) \longmapsto \int_{\R^k \setminus \{ 0 \}} \frac{\phi(\eta)}{q_\alpha(\eta)} \, \dd \eta
\label{map:distribution}
\end{equation}
that we use to define the solution operator $S_\alpha$ of 
\Cref{th:kalpha} as explained in \Cref{sec:precise_statement}.
Recall that $\mathcal{D}(\R^k\setminus\{0\})$ denotes the space of smooth 
functions in $\R^k \setminus \{ 0 \}$ with compact support.
As mentioned earlier, the extension of \eqref{map:distribution} will be denoted 
by $q_\alpha^{-1}$, it will 
belong to $\mathcal{S}^\prime (\R^k)$ {and it will have the same homogeneity 
property as \eqref{map:distribution}}.  
More precisely, we will
{construct the distribution 
$q_\alpha^{-1}\in \mathcal{S}^\prime(\R^k)$ satisfying:}
{\begin{enumerate}[label= {\rm (\roman*)} , ref=\roman*]
\item \label{item:ext} For every $\phi \in \mathcal{S}(\R^k)$ such that 
$\supp \phi \subset \R^k \setminus \{ 0 \}$ we have that
\[
\langle q_\alpha^{-1}, \phi \rangle = \int_{\R^k \setminus \{ 0\}} \frac{\phi(\eta)}{q_\alpha(\eta)} \, \dd \eta.
\]
\item \label{item:homogen} For every $\phi \in \mathcal{S}(\R^k)$ and $\lambda> 0$ we have that
\[
\langle q_\alpha^{-1}, \phi \rangle = \lambda^{-\alpha} \langle q_\alpha^{-1}, \phi_\lambda \rangle, 
\]
where $\phi_\lambda(\eta) = \lambda^k \phi(\lambda \eta)$ for $\eta \in \R^k$.
\end{enumerate}}

We establish the extension of the distribution in \eqref{map:distribution} in two 
different lemmas. First when $\alpha - k \notin \N_0$ and then for 
$\alpha - k \in \N_0$. {As we will see below, separating these cases is 
convenient since for $\alpha - k \notin \N_0$  there is a unique extension 
satisfying \eqref{item:homogen}. On the other hand, for $\alpha - k \in \N_0$  
there are no extensions satisfying \eqref{item:homogen} unless  the cancellation 
condition \eqref{it:cancellation} holds. Also, in the latter case the homogeneity 
does not uniquely  determine  the extension.} In order to state the extension 
lemmas we recall once again that
\[
\lceil \alpha - k \rceil = \min \{ l \in \N_0 : \alpha - k < l \}, 
\]
and introduce
\[
\lfloor \alpha - k \rfloor = \max \{ l \in \N_0 : l \leq \alpha - k \}. 
\]

\begin{lemma}\sl \label{lem:non_pos_integers} Let 
$k \in \N$ and $\alpha > 0$ be given, and consider $q_\alpha \in C(\R^k)$ satisfying \eqref{it:ellipticity} and \eqref{it:homogeneity}. If $\alpha - k \notin \N_0$, then the linear functional
\[
 \phi \in \mathcal{S}(\R^k) \longmapsto  
\begin{cases}\displaystyle
 \int_{\R^k} \frac{\phi(\eta)}{q_\alpha(\eta)} \, \dd \eta, & \alpha - k < 0,\vspace{1.2em}  \\ \displaystyle
 \frac{\Gamma(\alpha - k - \lfloor \alpha - k \rfloor)}{\Gamma(\alpha - k + 1)} \int_{\R^k} \frac{(\eta \cdot \nabla)^{\lceil \alpha - k \rceil} \phi(\eta)}{q_\alpha(\eta)} \, \dd \eta,  & \alpha - k > 0,
\end{cases}
\]
denoted by $q_\alpha^{-1}$, belongs to $\mathcal{S}^\prime(\R^k)$  and it is the only tempered distribution 
in $\R^k$ satisfying the properties \eqref{item:ext} and \eqref{item:homogen}.
Here $\Gamma$ denotes the Gamma function.
\end{lemma}

The extension of the distribution in \eqref{map:distribution} in the case $\alpha - k \in \N_0$ is the content of the following lemma.

\begin{lemma}\sl \label{lem:pos_integers} Let $k \in \N$ and $\alpha > 0$ be given  and consider $q_\alpha \in C(\R^k)$ satisfying 
\eqref{it:ellipticity} and \eqref{it:homogeneity}. If $\alpha - k \in \N_0$, then the linear functional
\[
\phi \in \mathcal{S}(\R^k) \longmapsto \frac{-1}{(\alpha - k)!} \int_{\R^k} \frac{\log |\eta|}{q_\alpha(\eta)} (\eta \cdot \nabla)^{\alpha - k + 1} \phi(\eta) \, \dd \eta,
\]
{denoted by $q_\alpha^{-1}$, belongs to $\mathcal{S}^\prime(\R^k)$ and satisfies \eqref{item:ext}. Moreover, $q_\alpha^{-1}$ satisfies the homogeneity property \eqref{item:homogen} if and only if $q_\alpha$ satisfies \eqref{it:cancellation}. }
\end{lemma}

Before going into the  proof of each one of the two lemmas above we make some comments that are relevant for any $k \in \N$ and $\alpha > 0$.

Note that for $\phi \in \mathcal{D} (\R^k \setminus \{ 0 \})$ we can use spherical coordinates in order to write
\begin{equation}
\int_{\R^k \setminus \{ 0 \}} \frac{\phi(\eta)}{q_\alpha(\eta)} \, \dd \eta = 
\int_{\SS^{k - 1}} \frac{1}{q_\alpha(\theta)} \bigg( \int_0^\infty \rho^{k - \alpha - 1} \phi (\rho \theta) \, \dd \rho \bigg) \, \dd S(\theta),
\label{id:2extend}
\end{equation}
where the volume form {$\dd S$ denotes the usual spherical measure on $\SS^{k- 1}$. With this identity in hand we see that if we aim at extending \eqref{map:distribution} it would be convenient to first extend the distribution in $\R\setminus \{ 0 \}$, given by
\begin{equation}
\psi \in \mathcal{D} (\R \setminus \{ 0 \}) \longmapsto \int_0^\infty \rho^{k - \alpha - 1} \psi (\rho) \, \dd \rho.
\label{map:distribution_R-0}
\end{equation}
In order to carry out this extension, consider for $a\in\R$ the function ${\rm x}_+^a : \R \setminus \{ 0 \} \rightarrow [0,\infty)$ given by the formula
\begin{equation}
\mathrm{x}_+^a (x) =  \begin{cases} x^a, & x>0,\vspace{.9em} \\  0, & x < 0. \end{cases}
\label{id:xa+}
\end{equation}
If $-1 < a$ then we can extend the function ${\rm x}_+^a$ to a function defined in the whole set $\R$ by setting for example ${\rm x}_+^a (0) = 0$; this extension is locally integrable and consequently defines a tempered distribution in $\R$, still denoted by ${\rm x}_+^a$:
\begin{equation}
\langle {\rm x}_+^a, \psi \rangle = \int_0^\infty x^a \psi(x)\, \dd x, \qquad \psi \in \mathcal{S}(\R), \,  a > -1.
\label{id:g-1}
\end{equation}
{Our next step is to construct a distribution on $\R$ which extends the function ${\rm x}_+^a$ when $a \leq -1$, defined by \eqref{id:xa+}  in $\R \setminus \{ 0 \}$.} {That is, we want to extend the functional
\begin{equation}
\psi \in \mathcal{D} (\R \setminus \{ 0 \}) \longmapsto \int_0^\infty x^a \psi (x) \, \dd x,
\label{map:distribution_R-0_2}
\end{equation}
 for $a \leq -1$.} We  first deal with the case 
$a = -1$. If $\psi \in \mathcal{D} (\R \setminus \{ 0 \})$ we have that
\[
\int_0^\infty x^{-1} \psi(x)\, \dd x  = - \int_0^\infty \log (x) \, \psi^\prime(x)\, \dd x,
\]
where $\psi^\prime$ denotes the first-order derivative of $\psi$. Note that the right-hand side of the previous inequality is finite for every 
$\psi \in \mathcal{S}(\R)$. Thus, we define the desired extension as
\begin{equation}
\langle {\rm x}_+^{-1}, \psi \rangle = - \int_0^\infty \log (x) \, \psi^\prime(x)\, \dd x, \qquad \forall\, \psi \in \mathcal{S}(\R).
\label{id:e-1}
\end{equation}
In order to perform the extension in the range $a < - 1$, we first note that for $l \in \N$ and $ b \notin \{ 0, \dots, l - 1 \} $ the function in \eqref{id:xa+} satisfies the identity
\[
\frac{d^l}{dx^l} {\rm x}_+^b = b \dots (b - (l - 1)) {\rm x}_+^{b-l}, 
\]
or equivalently, if $a \notin \{ -1, \dots, -l \}$
\[
{\rm x}_+^a = \frac{1}{(a + l) \dots (a +1)}  \frac{d^l}{dx^l} {\rm x}_+^{a+l}. 
\]
This means that, if $l \in \N$, $a \notin \{ -1, \dots, -l \}$ and $\psi \in \mathcal{D} (\R \setminus \{ 0 \})$ then
\[ 
\int_0^\infty x^a \psi(x)\, \dd x  = \frac{(-1)^l}{(a+l)\dots (a+1)} \int_0^\infty x^{a + l} \, \psi^{(l)}(x)\, \dd x. 
\]
Here $\psi^{(l)}$ denotes the $l^{\rm th}$-order derivative of $\psi$. Therefore,  whenever $-l-1 \leq a < - l $  with $l \in \N$, we {can define the tempered distribution ${\rm x}_+^a$ by}
\begin{equation}
\langle {\rm x}_+^a, \psi \rangle = \frac{(-1)^l}{(a+l)\dots (a+1)} \langle {\rm x}_+^{a+l}, \psi^{(l)} \rangle,  \qquad \psi \in \mathcal{S}(\R).
\label{id:l-1}
\end{equation}

The procedure described above defines for all $a \in \R$ a tempered distribution 
${\rm x}_+^a$ on $\R$ that extends the distribution given in 
\eqref{map:distribution_R-0_2} initially defined on $\R \setminus \{ 0 \}$. {We 
now  use this distribution ${\rm x}_+^a$ in order to define a linear functional 
that eventually will represent the tempered distribution $q_\alpha^{-1}$ in the 
whole of $\R^k$.} {This linear functional is}
\begin{equation}
\phi \in \mathcal{S}(\R^k) \longmapsto \int_{\SS^{k - 1}} \frac{1}{q_\alpha(\theta)} \langle {\rm x}_+^{k - \alpha - 1}, \psi_\theta \rangle \, \dd S(\theta)
\label{map:qalpha-1}
\end{equation}
where $\psi_\theta (x) = \phi (x \theta)$ for $x \in \R$. Note that this functional  {is indeed} a tempered distribution on $\R^k$.

{Using \eqref{id:2extend} and the fact that ${\rm x}_+^{k - \alpha - 1}$ extends the functional \eqref{map:distribution_R-0} we conclude that}
\[
\int_{\SS^{k - 1}} \frac{1}{q_\alpha(\theta)} \langle {\rm x}_+^{k - \alpha - 1}, \psi_\theta \rangle \, \dd S(\theta) = \int_{\R^k \setminus \{ 0\}} \frac{\phi(\eta)}{q_\alpha(\eta)} \, \dd \eta
\]
for all $\phi \in \mathcal{S}(\R^k)$ such that 
$\supp \phi \subset \R^k \setminus \{ 0 \}$ with 
$\psi_\theta (x) = \phi (x \theta)$ for $x \in \R$. This shows that the 
functional \eqref{map:qalpha-1} satisfies the property \eqref{item:ext}.

In the proofs of \Cref{lem:non_pos_integers} and \Cref{lem:pos_integers} {{we 
show that the tempered distribution \eqref{map:qalpha-1} coincides with the 
linear functionals defined in the corresponding statements, and satisfies the 
homogeneity property \eqref{item:homogen}.}}
{Moreover, when $\alpha - k \notin \N_0$ we check that this is the unique 
extension of the distribution \eqref{map:distribution}}.

\begin{proof}[Proof of \Cref{lem:non_pos_integers}]
The fact that
\[\int_{\SS^{k - 1}} \frac{1}{q_\alpha(\theta)} \langle {\rm x}_+^{k - \alpha - 1}, \psi_\theta \rangle \, \dd S(\theta) = \int_{\R^k} \frac{\phi(\eta)}{q_\alpha(\eta)} 
\]
{whenever $\alpha-k<0$} follows immediately from the definition \eqref{map:qalpha-1} for $q_\alpha^{-1}$ and the identity \eqref{id:g-1}. One only needs to undo the spherical change of coordinates. The equality
\[
\int_{\SS^{k - 1}} \frac{1}{q_\alpha(\theta)} \langle {\rm x}_+^{k - \alpha - 1}, \psi_\theta \rangle \, \dd S(\theta) = \frac{\Gamma(\alpha - k - \lfloor \alpha - k \rfloor)}{\Gamma(\alpha - k + 1)} \int_{\R^k} \frac{(\eta \cdot \nabla)^{\lceil \alpha - k \rceil} \phi(\eta)}{q_\alpha(\eta)}
\]
{for  $\alpha - k > 0$} follows from the identities \eqref{id:l-1} and 
\eqref{id:g-1}. In this case one needs to use the properties of the Gamma 
function in order to calculate the precise constant, and note that
\begin{equation}
\psi_\theta^{(l)} (x) = (\theta \cdot \nabla)^{l} \phi (x \theta), \quad \forall \,x \in \R,\enspace l \in \N.
\label{id:radial_derivative}
\end{equation}
The homogeneity property \eqref{item:homogen} follows by performing the change of 
variables $\eta = \lambda \zeta$ in the expressions of $q_\alpha^{-1}$ written in 
the statement of the lemma.

Finally, we discuss the issue of uniqueness. {The difference between 
$q_\alpha^{-1}$ and any other  distribution satisfying \eqref{item:ext} and 
\eqref{item:homogen} is  a distribution supported at $\{ 0 \}$.  This means that 
it has to be a finite linear combination of derivatives of $\delta_0$, the Dirac 
distribution in $\R^k$ supported at $\{ 0 \}$. That is, the difference can be 
written as $\sum_{|\beta| \leq l} c_\beta \partial^\beta \delta_0$ with 
$c_\beta \in \C$, $\beta \in \N_0^k$ and $l \in \N_0$.}  However, this difference 
has to satisfy \eqref{item:homogen}, namely it is homogeneous of order $-\alpha$. 
This can only happen if $c_\beta = 0$ for all $\beta \in \N_0^k$ such that 
$|\beta| \leq l$. Indeed,
\[
\langle \partial^\beta \delta_0, \phi_\lambda \rangle = \lambda^{k + |\beta|} \langle \partial^\beta \delta_0, \phi \rangle, 
\]
which means that $\partial^\beta \delta_0$ is homogeneous of order $-k-|\beta|$. 
Since $\alpha - k \notin \N_0$, $\alpha \neq k + |\beta| $ for any 
$\beta \in \N_0^k$. This completes the proof of \Cref{lem:non_pos_integers}.
\end{proof}

\begin{proof}[Proof of \Cref{lem:pos_integers}] The equality
\[
\int_{\SS^{k - 1}} \frac{1}{q_\alpha(\theta)} \langle {\rm x}_+^{k - \alpha - 1}, \psi_\theta \rangle \, \dd S(\theta) = \frac{-1}{(\alpha - k)!} \int_{\R^k} \frac{\log |\eta|}{q_\alpha(\eta)} (\eta \cdot \nabla)^{\alpha - k + 1} \phi(\eta) \, \dd \eta 
\]
for $\alpha - k \in \N_0$ is a consequence of the identity \eqref{id:e-1} when 
$\alpha = k$, and it follows \eqref{id:l-1} and \eqref{id:e-1} when 
$\alpha - k \in \N$. Again one needs to use the relation 
\eqref{id:radial_derivative} {to verify this}.

Let us now address the homogeneity of $q_\alpha^{-1}$. In this case 
$\alpha - k \in \N_0$ and we will have that the homogeneity holds if and only if 
we assume condition \eqref{it:cancellation}.

After the change of variables $\eta = \lambda \zeta$ we have that
\begin{equation}\label{eq:last_L}
\begin{split}
& \int_{\R^k} \frac{\log |\eta|}{q_\alpha(\eta)}  (\eta \cdot \nabla)^{\alpha - k + 1} \phi(\eta) \, \dd \eta 
\\
&\qquad \qquad = \lambda^{- \alpha} \int_{\R^k} \frac{\log |\zeta|}{q_\alpha(\zeta)} (\zeta \cdot \nabla)^{\alpha - k + 1} [ \lambda^k \phi (\lambda \centerdot)](\zeta) \, \dd \zeta 
\\
& \qquad\qquad\quad + \lambda^{- \alpha} \log \lambda \int_{\R^k} \frac{1}{q_\alpha(\zeta)} (\zeta \cdot \nabla)^{\alpha - k + 1} [ \lambda^k \phi (\lambda \centerdot)](\zeta) \, \dd \zeta.
\end{split}
\end{equation}
Hence, we can see that $q_\alpha^{-1}$ satisfies \eqref{item:homogen} if and 
only if the last term on the right-hand side of the previous identity vanishes 
for all $\phi \in \mathcal S(\R^k)$. Let us compute this term explicitly. 
 
Using spherical coordinates we have for any $\phi \in \mathcal{S}(\R^k)$ that
\[
\int_{\R^k} \frac{1}{q_\alpha(\zeta)} (\zeta \cdot \nabla)^{\alpha - k + 1} \phi (\zeta) \, \dd \zeta  = \int_{\SS^{k - 1}} \frac{1}{q_\alpha (\theta)} \int_0^\infty \psi_\theta^{(\alpha - k + 1)} (\rho) \, \dd \rho \, \dd S(\theta)
\]
with $\psi_\theta (\rho) = \phi (\rho \theta)$ for $\rho > 0$, and 
$\psi_\theta^{(\alpha - k + 1)}$ denoting the $(\alpha - k + 1)^{\rm th}$-order 
derivative of $\psi_\theta$. Since by \eqref{id:radial_derivative}
\[
\int_0^\infty \psi_\theta^{(\alpha - k + 1)} (\rho) \, \dd \rho = - (\theta \cdot \nabla)^{\alpha - k} \phi(0),
\]
we have that
\[
\begin{split}
 \int_{\R^k} \frac{1}{q_\alpha(\zeta)} (\zeta \cdot \nabla)^{\alpha - k + 1} \phi (\zeta) \, \dd \zeta & = -  \int_{\SS^{k - 1}} \frac{(\theta \cdot \nabla)^{\alpha - k} \phi(0)}{q_\alpha (\theta)} \, \dd S(\theta)
 \\
& = -  \sum_{|\beta |=\alpha - k} c_\beta \partial^\beta \phi(0) \int_{\SS^{k - 1}} \frac{\theta^\beta }{q_\alpha (\theta)} \, \dd S(\theta)
\end{split}
\]
where the coefficients $c_\beta$ are strictly positive. Therefore, the 
right-hand side vanishes for all $\phi \in \mathcal S(\R^k)$ if and only if 
\eqref{it:cancellation} holds. Consequently, if the cancellation property 
\eqref{it:cancellation}  is verified then 
\[
\int_{\R^k} \frac{1}{q_\alpha(\zeta)} (\zeta \cdot \nabla)^{\alpha - k + 1} \phi (\zeta) \, \dd \zeta  = 0
\]
for all $\phi \in \mathcal S(\R^k)$, and by \eqref{eq:last_L}, property 
\eqref{item:homogen} is satisfied as desired. The proof of the lemma is complete
\end{proof}

\begin{remark}\label{rmrk:unique}{Condition \eqref{it:cancellation} in the 
statement of Lemma~\ref{lem:pos_integers} is necessary: if $q_\alpha $ does not 
satisfy \eqref{it:cancellation} then there is no distribution in 
$\mathcal S'(\R^k)$ such that \eqref{item:ext} and \eqref{item:homogen} are 
satisfied simultaneously.} To see this let $t_\alpha ^{-1}$ be a tempered 
distribution satisfying \eqref{item:ext}. As mentioned in the proof of 
\Cref{lem:non_pos_integers} $t_\alpha ^{-1}$  differs from $q_\alpha^{-1}$ by a 
linear combination of derivatives of  a Dirac delta supported at $\{0\}$. It is 
easy to verify  that a linear combination of derivatives of $\delta_0$ does not 
affect the second summand in the right hand side of \eqref{eq:last_L}. Thus, if 
$t_\alpha ^{-1}$ satisfies \eqref{item:ext} and \eqref{item:homogen} then the 
second summand in the right hand side of \eqref{eq:last_L} must vanish and the 
same proof yields that $t_\alpha ^{-1}$ must satisfy \eqref{it:cancellation}. 
\end{remark}

For future reference it is important to {note} that  $q_\alpha^{-1} \in H^{-\delta}_{\rm loc} (\R^k)$ with $\delta = \lceil \alpha - k \rceil$ if  $\alpha - k/ 2 < \lceil \alpha - k \rceil$ and $\delta > \alpha - k/2$  if $\alpha - k/ 2 \geq \lceil \alpha - k \rceil$. In order to prove this property we show that $\chi q_\alpha^{-1} \in H^{-\delta} (\R^k)$ for any $\chi \in \mathcal{D}(\R^n)$. By definition we have for all  $\phi \in \mathcal{S}(\R^k)$ that
\begin{equation}
|\langle \chi q_\alpha^{-1}, \phi \rangle| \lesssim \Big( \int_{K} \Big( \log \frac{1+|\eta|}{|\eta|} \Big)^p \frac{|\eta|^{p \lceil \alpha - k \rceil}}{|q_\alpha(\eta)|^p} \, \dd \eta \Big)^{1/p} \| \phi \|_{W^{\lceil \alpha - k \rceil, p^\prime}(\R^k)};
\label{in:chiqalpha-1}
\end{equation}
here $K = \supp \chi$. Using spherical coordinates one can check that the  integral in $K$ is finite whenever $p < k/ [ k - (\lceil \alpha - k \rceil -(\alpha - k)) ]$. Since $\alpha - k/ 2 < \lceil \alpha - k \rceil$ if and only if $2 < k/ [ k - (\lceil \alpha - k \rceil -(\alpha - k)) ]$, we can choose $p = 2$ in \eqref{in:chiqalpha-1} and deduce by duality that $\chi q_\alpha^{-1} \in H^{-\lceil \alpha - k \rceil} (\R^k)$ whenever
$\alpha - k/ 2 < \lceil \alpha - k \rceil$. If $\alpha - k/ 2 \geq \lceil \alpha - k \rceil$, by Sobolev's embedding theorem we have that
\[
\begin{split}
\| \phi \|_{W^{\lceil \alpha - k \rceil, p^\prime}(\R^k)} & \eqsim \|(I - \Delta)^{\lceil \alpha - k \rceil/2} \phi \|_{L^{p^\prime}(\R^k)} \\
& \lesssim \|(I - \Delta)^{\lceil \alpha - k \rceil/2} \phi \|_{H^{s}(\R^k)} = \| \phi \|_{H^{s + \lceil \alpha - k \rceil}(\R^k)}
\end{split}
\]
for $s-k/2 = - k/p^\prime$. The upper bound on $p$ imposes the lower bound $s + \lceil \alpha - k \rceil > \alpha - k/2 $. Thus, if $\alpha - k/ 2 \geq \lceil \alpha - k \rceil$, again by duality we can ensure that $\chi q_\alpha^{-1} \in H^{-\delta} (\R^k)$ for all $\delta > \alpha - k/2$.

{We conclude this section by recording some useful remarks. 
Let $\psi \in \mathcal{S}(\R^k)$ be such that 
$\int_{\R^k} \psi(\eta) \, \dd \eta = 1$ and for 
$t \in \R_+$ we set $\psi_{t} (\eta) = t^k \psi (t \eta)$ for $\eta \in \R^k$.
We introduce the smooth function
\begin{equation}
q_{\alpha, t}^{-1} (\eta) = (q_\alpha^{-1} \star \psi_{t})(\eta) = \langle q_\alpha^{-1}, \psi_{t} (\eta - \centerdot) \rangle, \qquad  \eta \in \R^k.
\label{id:qalphaN-1}
\end{equation}
First, one can verify that
\begin{equation}
\lim_{t \to 0} \| q_{\alpha, t}^{-1}    \|_{L^\infty(\R^k)} = 0, 
\label{lim:conv_t_small}
\end{equation} 
This follows from  \eqref{item:homogen} which implies that
\[
q_{\alpha, t}^{-1} (\eta) =  \langle q_\alpha^{-1}, \psi_{t} (\eta - \centerdot) \rangle =  t^\alpha \langle q_\alpha^{-1}, \psi  (t\eta - \centerdot) \rangle. 
\]
Secondly, we have that for every compact subset $K$ of $\R^k$ such that 
$0 \notin K$
\begin{equation}
\lim_{t \to \infty} \sup_{\eta \in K} \Big| \frac{1}{q_\alpha (\eta)} - q_{\alpha, t}^{-1} (\eta) \Big| = 0.
\label{lim:uniform_convergence_compacts}
\end{equation}
Thirdly, for every $\phi \in \mathcal{S}(\R^k)$ we have that
\begin{equation}
\lim_{t \to \infty} \int_{\R^k} q_{\alpha, t}^{-1} (\eta) \phi(\eta) \, \dd \eta = \langle q_\alpha^{-1}, \phi \rangle.
\label{lim:tempered_distributions}
\end{equation}} 
Lastly, {for every $\chi \in \mathcal{D}(\R^n)$ we have that
\begin{equation}
\lim_{t \to \infty} \|\chi (q_\alpha^{-1} - q_{\alpha, t}^{-1}) \|_{H^{- \delta} (\R^k)} = 0,
\label{lim:H-deltaLOC}
\end{equation}
with $\delta = \lceil \alpha - k \rceil$ if $\alpha - k/ 2 < \lceil \alpha - k \rceil$ and $\delta > \alpha - k/2$ if $\alpha - k/ 2 \geq \lceil \alpha - k \rceil$.}

\section{Rotational smoothing in the general case}\label{sec:rotational_smoothing}
This section is divided into four subsections.  
In \Cref{sec:funct_subsec} we provide some preliminary facts about the functional 
spaces and an important property of Sobolev spaces in $\SO(n)$, and in Sections 
\ref{sec:proof-mainTH} and \ref{sec:proof_I_II_III} we prove 
\Cref{th:0alphak,th:kalpha}. Recall from our discussion in the 
\Cref{sec:precise_statement} and the statement of \Cref{th:kalpha},
the need of introducing a distribution $r_\alpha$ standing formally as the 
restriction to $ \SS^{n-1}$ of $q_\alpha^{-1} \otimes \mathbf{1}_{\R^{n - k}}$.  
In \Cref{sec:ralph_S} we give the explicit definition of $r_\alpha$.

\subsection{Some preliminary notions} \label{sec:funct_subsec} We now summarize 
the key properties of Sobolev spaces and $\SO(n)$ that one needs in order to 
prove \Cref{th:0alphak,th:kalpha}. 
We provide the statements without proofs in this section, leaving the more 
rigorous analysis of the functional framework for \Cref{sec:abstract_framework}.

Let $(M, g)$ stand for a compact boundaryless 
Riemannian manifold of dimension $m \in \N$. We denote the canonical Laplace-Beltrami in $M$ by $\Delta_g$. We denote by $\lambda_j$, $j\in \N_0$, the eigenvalues of $-\Delta_g$, which are always non-negative, and denote by $\Pi_j$ the projector onto the corresponding invariant subspace of eigenfunctions associated to $\lambda_j$.  
This eigenvalue expansion yields a natural way to define the fractional Bessel potentials of order  $s \in \R$ for every $u \in C^\infty(M)$
\begin{equation} \label{eq:bessel_pot}
(\Id - \Delta_g)^{s/2} u = \sum_{j \in \N_0} (1 + \lambda_j)^{s/2} \Pi_j u.
\end{equation}
We can now define the fractional Sobolev spaces $H^s(M)$ for $s\in \R$ in an intrinsic way as the Banach completion of {$C^\infty(M)$ with respect to the norm}
\[
\|u  \|_{H^s(M)} = \|(\Id - \Delta_g)^{s/2}  u \|_{L^2(M)} = \Big( \sum_{j \in \N_0} (1 + \lambda_j)^s \| \Pi_j u \|_{L^2(M)}^2 \Big)^{1/2}.
\]

The next proposition yields an important relation between Sobolev spaces in 
$\SO(n)$ and $\SS^{n-1}$ which plays an important role in the proof of 
\Cref{th:0alphak,th:kalpha}.

\begin{proposition} \label{prop:SOn_Sn-1_Sobolev} \sl Fix any $v \in \SS^{n-1}$,  and let $f$ be in $H^s(\SS^{n - 1})$. Consider the function $F_v: \SO(n) \to \SS^{n-1}$ given by
$F_v (Q) = Q^\T v$ for every rotation $Q \in \SO(n)$. 
Then, 
it holds that
\[ 
\| f \circ F_v \|_{H^s(\SO(n))} = S(\SS^{n-1})^{-1/2}\| f \|_{H^s(\SS^{n - 1})}
 \]
for all $s \in \R$. 
\end{proposition}

This result is a generalization for $s\neq 0$ of the well known identity that we 
have already stated in \eqref{id:int_haar}. The proof requires some preliminary 
notions about $\SO(n)$ and is given in \Cref{app:orthogonal_group}. It essentially follows from an identity relating the operators $\Delta_{\SO(n)}$ and $\Delta_{\SS^{n-1}}$ that we prove in \Cref{lem:SOn_Sn-1_LB}.
 
{The following proposition, which we prove in \Cref{sec:abstract_framework}, will allow us to work with Sobolev spaces on local charts.} 
%

\begin{proposition}\label{prop:norm_charts} \sl Let $(M, g)$ be a compact 
boundaryless Riemannian manifold of dimension $m \in \N$. 
Let $\{ (U_1, \varphi_1), \dots, (U_l, \varphi_l) \}$ be a smooth atlas for $M$, 
and consider $\{ \chi_1, \dots, \chi_l \}$ a smooth partition of unity 
subordinated to the atlas. Then, for $s \in \R$ we have that
\[ \| u \|_{H^s(M)} \eqsim \sum_{j = 1}^l \| (\chi_j u)\circ \varphi_j^{-1} \|_{H^s(\R^m)} \]
for all $u \in C^\infty(M)$. The implicit constants only depend on $m$, $s$, the 
atlas and the partition of unity.
\end{proposition}

We now define the spaces $H^s(M; \dot{H}^t(\R^n))$ that appear in the statement 
of \Cref{th:0alphak,th:kalpha}.

\begin{definition}\label{def:A} \sl Let $\mathcal{A}(\R^n \times M)$ denote the set of 
$u \in C^\infty(\R^n \times M)$ such that for every $\alpha, \beta \in \N_0^n$ 
and every $k \in  \N_0$ there exists $C>0$ such that
\[ |x^\beta (\Id - \Delta_g)^k \partial^\alpha u (x, p)| \leq C \qquad \forall\, (x, p) \in \R^n \times M.\]
\end{definition}

A function  
$u \in C^\infty (\R^n \times M) $ belongs to $ \mathcal{A}(\R^n \times M)$ if and only if its Fourier transform, denoted by $\widehat{u}$  and defined by
\begin{equation} \label{eq:Fourier}  
 (\xi, p) \in \R^n \times M \longmapsto \frac{1}{(2\pi)^{n/2}} \int_{\R^n} e^{-\i \xi \cdot x} u(x,p) \, \dd x \in \C, 
 \end{equation}
belongs to $\mathcal{A}(\R^n \times M)$.
In \Cref{sec:abstract_framework} we show that the non-negative functional
\begin{equation} \label{eq:functional}
 u \in \mathcal{A}(\R^n \times M)
\longmapsto  \Big( \int_M \| (\Id - \Delta_g)^{s/2} u(\centerdot, p) \|_{\dot{H}^t(\R^n)}^2 \, \dd \mu_g(p) \Big)^{1/2} 
\end{equation}  
is well defined for all $s \in \R$ and $t \in \R_+$. It defines
a norm on $\mathcal{A}(\R^n \times M)$ that will be denoted from now on by 
$\| \centerdot \|_{H^s(M; \dot{H}^t(\R^n))}$. Then, the space 
$H^s(M; \dot{H}^t(\R^n))$ is the Banach completion of 
$\mathcal{A}(\R^n \times M)$  with respect to the previous norm. {For the 
convenience of the reader let us point out that} $L^2(M; \dot{H}^t(\R^n))$ and 
$ H^0(M; \dot{H}^t(\R^n))$ denote the same space. 
In \Cref{lem:new_norm} we also prove that
\begin{equation}
\label{id:norm_different-look}
\| u \|_{H^s(M; \dot{H}^t(\R^n))} = \Big( \int_{\R^n} |\xi|^{2t} \| \widehat{u}(\xi, \centerdot) \|_{H^s(M)}^2 \, \dd \xi \Big)^{1/2}
\end{equation}
for all $u \in \mathcal{A}(\R^n \times M)$, $s \in \R$ and $t \in \R_+$.
{Finally, we show in \Cref{prop:distributionHsHt} that every} 
$u \in H^s(M; \dot{H}^t(\R^n))$ with $s \in \R$ and $t \in \R_+$ determines a 
unique distribution in $\R^n \times M$.

{With these ingredients we are now ready to prove 
\Cref{th:0alphak,th:kalpha}. We remind the interested reader that detailed 
statements and proofs of the preliminary results above are contained in 
\Cref{sec:abstract_framework}.}


\subsection{Proof of 
\texorpdfstring{\Cref{th:0alphak,th:kalpha}}{main-th-full}} 
\label{sec:proof-mainTH}
{This section contains the full scheme needed to prove the rotational smoothing 
phenomenon.}

Consider the sequence $\{ q_{\alpha, N}^{-1}: N\in\N\}$ with $q_{\alpha, N}^{-1}$ 
defined in \eqref{id:qalphaN-1}, and define for $N \in \N$ the symbol
\[
 p_{\alpha, N}^{-1} (\xi, Q) = [q_{\alpha, N}^{-1} \otimes \mathbf{1}_{\R^{n-k}}](Q^\T \xi), \qquad  (\xi, Q) \in \R^n \times \SO(n),
\]
and the operator
\begin{equation}
\label{id:Salpha}
S_{\alpha, N} f (x, Q) = \frac{1}{(2\pi)^{n/2}} \int_{\R^n} e^{\i x \cdot \xi} p_{\alpha, N}^{-1} (\xi, Q) \widehat{f}(\xi) \, \dd \xi,\qquad (x, Q) \in \R^n \times \SO(n),
\end{equation} 
with $f \in \mathcal{S}(\R^n)$. One can 
check that $S_{\alpha, N} f \in \mathcal{A}(\R^n \times \SO(n))$ for all 
$f \in \mathcal{S}(\R^n)$ and all $N \in \N$. By 
\eqref{id:norm_different-look} we have that for 
every $N \in \N$ and $f \in \mathcal{S}(\R^n)$, the following identity holds
\begin{equation}
\label{id:plancherel_SalphaN}
\begin{aligned}
\| S_{\alpha, N} f & \|_{H^{-\delta}(\SO(n); \dot{H}^\alpha(\R^n))}^2  = \int_{\R^n} |\xi|^{2 \alpha} \| p_{\alpha, N}^{-1}(\xi, \centerdot) \|_{H^{-\delta}(\SO(n))}^2 |\widehat{f}(\xi)|^2 \, \dd \xi.
\end{aligned}
\end{equation}
For $v= {\xi}/|\xi|$, we can write 
\[
p_{\alpha, N}^{-1} (\xi, Q) = [q_{\alpha, N}^{-1} \otimes \mathbf{1}_{\R^{n-k}}](|\xi|  Q^\T v),  \qquad  (\xi, Q) \in \R^n \times \SO(n),
\]
and by 
\Cref{prop:SOn_Sn-1_Sobolev} 
we have that  
\begin{equation}
\label{id:killing_freedom_degrees}
\begin{split}
\| p_{\alpha, N}^{-1} (\xi, \centerdot)  \|_{H^{- \delta}(\SO(n))}& = \| [q_{\alpha, N}^{-1} \otimes \mathbf{1}_{\R^{n-k}}](|\xi| F_v) \|_{H^{- \delta}(\SO(n))} 
\\
& = S(\SS^{n-1})^{-1/2} \| [q_{\alpha, N}^{-1} \otimes \mathbf{1}_{\R^{n-k}}](|\xi| \centerdot ) \|_{H^{- \delta}(\SS^{n-1})}.
\end{split}
\end{equation}
Our first step towards deriving \Cref{th:0alphak,th:kalpha}
will be to show that, for all $f \in \mathcal{S}(\R^n)$, the 
sequence $\{S_{\alpha, N} f: N\in\N\}$ is a Cauchy sequence with respect to the 
norm $\| \centerdot \|_{H^{-\delta}(\SO(n); \dot{H}^\alpha(\R^n))}$. Start by 
considering an arbitrary $f \in \mathcal{S}(\R^n)$ and noting that for $L, N \in \N$ we have that
\[
\begin{split}
S(\SS^{n-1}) & \| S_{\alpha, N+L} f - S_{\alpha, N} f \|_{H^{-\delta}(\SO(n); \dot{H}^\alpha(\R^n))}^2
\\
&= \int_{\R^n} |\xi|^{2 \alpha} \| [(q_{\alpha, N+L}^{-1} - q_{\alpha, N}^{-1}) \otimes \mathbf{1}_{\R^{n-k}}](|\xi| \centerdot ) \|_{H^{- \delta}(\SS^{n-1})}^2 |\widehat{f}(\xi)|^2 \, \dd \xi.
\end{split}
\]

{Recall that $\delta = \lceil \alpha - k \rceil$ if  $\alpha - k/ 2 < \lceil \alpha - k \rceil$, and $\delta > \alpha - k/2$  if $\alpha - k/ 2 \geq \lceil \alpha - k \rceil$. We make the following claims:} 
\begin{enumerate}[label=(\Roman*),ref=\Roman*]
\item \label{lim:vanishing_Cauchy} For every $\xi \in \R^n \setminus \{ 0 \}$, we have that
\[
\lim_{N \to \infty} |\xi|^\alpha \| [(q_{\alpha, N+L}^{-1} - q_{\alpha, N}^{-1}) \otimes \mathbf{1}_{\R^{n-k}}](|\xi| \centerdot ) \|_{H^{- \delta}(\SS^{n-1})} = 0
\]
uniformly for all $L \in \N$.
\item \label{in:domination_sequence} There exists a constant $C > 0$ such that
\[
|\xi|^\alpha \| [q_{\alpha, N}^{-1} \otimes \mathbf{1}_{\R^{n-k}}](|\xi| \centerdot ) \|_{H^{- \delta}(\SS^{n-1})} \leq C
\]
for all $\xi \in \R^n \setminus \{ 0 \}$ and all $N \in \N$.
\item\label{in:r_alpha_convergence} {There is a distribution $r_\alpha \in H^{-\delta}(\SS^{n-1})$ such that
\[
\lim_{N \to \infty} \|\lambda^\alpha [q_{\alpha, N}^{-1}   \otimes \mathbf{1}_{\R^{n-k}}](\lambda \centerdot ) - r_\alpha \|_{H^{- \delta} (\SS^{n-1})} = 0,
\]
for every $\lambda>0$, where $r_\alpha$ is independent of $\lambda$.}
\end{enumerate}
{We postpone the proofs of the claims \eqref{lim:vanishing_Cauchy}, 
\eqref{in:domination_sequence} and \eqref{in:r_alpha_convergence} until 
\Cref{sec:proof_I_II_III} below. Assuming \eqref{lim:vanishing_Cauchy} and 
\eqref{in:domination_sequence}, the dominated convergence theorem implies that}
\[
\lim_{N \to \infty} \| S_{\alpha, N+L} f - S_{\alpha, N} f \|_{H^{-\delta}(\SO(n); \dot{H}^\alpha(\R^n))} = 0 
 \]
uniformly in  $L \in \N$, which means that $\{S_{\alpha, N} f: N\in\N\}$ is Cauchy with respect to the norm  $\| \centerdot \|_{H^{-\delta}(\SO(n); \dot{H}^\alpha(\R^n))}$.


We proceed to proving the identity in the statement of the theorem. Let us denote by  $u_\alpha^f \in H^{-\delta}(\SO(n); \dot{H}^\alpha(\R^n))$ the equivalence 
class of the Cauchy sequence $\{S_{\alpha, N} f: N\in\N\}$; we have by 
\eqref{id:plancherel_SalphaN} and \eqref{id:killing_freedom_degrees} that
\[
\begin{split}
S(\SS^{n-1}) \| & u_\alpha^f \|_{H^{-\delta}(\SO(n); \dot{H}^\alpha(\R^n))}^2
\\
& = \lim_{N \to \infty} \int_{\R^n} |\xi|^{2 \alpha} 
\| [q_{\alpha, N}^{-1} \otimes \mathbf{1}_{\R^{n-k}}](|\xi| \centerdot ) \|_{H^{- \delta}(\SS^{n-1})}^2 |\widehat{f}(\xi)|^2 \, \dd \xi.
\end{split}
\]
On the one hand, as mentioned in \Cref{sec:funct_subsec},   $u_\alpha^f$ determines a distribution in $\R^n \times \SO(n)$. On the other hand, we have that
\[ 
S_\alpha f (x, Q) - S_{\alpha, N} f (x, Q) = \frac{1}{(2\pi)^{n/2}} \langle (q_\alpha^{-1} - q_{\alpha, N}^{-1})\otimes \mathbf{1}_{\R^{n - k}}, \widehat{ f(Q (\centerdot + x))} \rangle. 
\]
By \eqref{lim:tempered_distributions} we can ensure that for every  compact $K \subset \R^n$ we have
\[
\lim_{N \to \infty} \|S_\alpha f - S_{\alpha, N} f\|_{L^\infty(K \times \SO(n))} = 0.
\]
Hence $S_\alpha f \in C(\R^n \times \SO(n))$ can be identified with a  
distribution in $\R^n \times \SO(n)$ and
\[
S_{\alpha, N} f \xrightarrow[N \to \infty]{} S_\alpha f \quad \textnormal{in}\quad \mathcal{D}^\prime (\R^n \times \SO(n)).
\]
This implies that $u_\alpha^f = S_\alpha f $ {which yields that}
\begin{equation}
\label{id:lim-int_plancherel}
\begin{aligned}
S(\SS^{n-1}) & \| S_\alpha f \|_{H^{-\delta}(\SO(n); \dot{H}^\alpha(\R^n))}^2 \\
& = \lim_{N \to \infty} \int_{\R^n} |\xi|^{2 \alpha} 
\| [q_{\alpha, N}^{-1} \otimes \mathbf{1}_{\R^{n-k}}](|\xi| \centerdot ) \|_{H^{- \delta}(\SS^{n-1})}^2 |\widehat{f}(\xi)|^2 \, \dd \xi.
\end{aligned}
\end{equation}
%
%
{Now \eqref{in:r_alpha_convergence} implies that
\[
\| r_\alpha  \|_{H^{-\delta} (\SS^{n-1})} = \lim_{N \to \infty} |\xi|^\alpha \| [q_{\alpha, N}^{-1} \otimes \mathbf{1}_{\R^{n-k}}](|\xi| \centerdot ) \|_{H^{- \delta}(\SS^{n-1})}.
\]
This identity together with \eqref{in:domination_sequence} turn, by the 
dominated convergence theorem, the identity \eqref{id:lim-int_plancherel} into
\[
\| S_\alpha f \|_{H^{-\delta}(\SO(n); \dot{H}^\alpha(\R^n))}^2 = S(\SS^{n-1})^{-1}  \| r_\alpha   \|_{H^{-\delta} (\SS^{n-1})}^2  \int_{\R^n}|\widehat{f}(\xi)|^2 \, \dd \xi.
\]
{Consequently, Plancherel's identity in $\R^n$ yields the identities stated in 
\Cref{th:0alphak,th:kalpha}. This completes their proofs, up to showing the 
claims \eqref{lim:vanishing_Cauchy}, 
\eqref{in:domination_sequence} and \eqref{in:r_alpha_convergence}.}

\subsection{Proofs of claims \texorpdfstring{\eqref{lim:vanishing_Cauchy}, 
\eqref{in:domination_sequence} and \eqref{in:r_alpha_convergence}}{(I)-(III)} }
\label{sec:proof_I_II_III}



Before addressing these claims, we make a simple observation
about the set of distributions 
$\{ q_{\alpha, t}^{-1} \otimes \mathbf{1}_{\R^{n-k}}: t \in\R_+\}$ with 
$q_{\alpha, t}^{-1}$ defined in \eqref{id:qalphaN-1}. This is, for 
$\lambda \in \R_+$ and $\theta\in\SS^{n-1}$ we have that
\begin{equation}
\label{id:smooth_homogeneity}
\lambda^\alpha [q_{\alpha, t}^{-1} \otimes \mathbf{1}_{\R^{n-k}}](\lambda \theta ) = [q_{\alpha, \lambda t}^{-1} \otimes \mathbf{1}_{\R^{n-k}}](\theta ).
\end{equation}
As we see in the following identities this is a consequence of the homogeneity 
property of $q_\alpha^{-1}$ (see \Cref{lem:non_pos_integers,lem:pos_integers}): 
\[
\begin{split}
\lambda^\alpha q_{\alpha, t}^{-1} (\lambda \eta) &= \lambda^\alpha \langle q_\alpha^{-1}, \psi_t ( \lambda \eta - \centerdot) \rangle = \lambda^\alpha \lambda^{-\alpha} \langle q_\alpha^{-1}, \lambda^k \psi_t ( \lambda \eta - \lambda \centerdot) \rangle 
\\
& = \langle q_\alpha^{-1}, \psi_{\lambda t} ( \eta - \centerdot) \rangle = q_{\alpha, \lambda t}^{-1} (\eta),
\end{split}
\]
where $\psi_t (\eta) = t^k \psi (t \eta)$.

We start by proving \eqref{lim:vanishing_Cauchy}. In order to deal with the norm 
of 
$H^{- \delta}(\SS^{n-1})$ we will use \Cref{prop:norm_charts}. For this reason 
we introduce an atlas 
$\{(U_1, \varphi_1), \dots, (U_{2n}, \varphi_{2n})\}$ for $\SS^{n - 1}$. 
Consider the open subset of $\R^{n - 1}$ given by
\begin{equation}  \label{id:setV}
V = \{ y \in \R^{n - 1} : |y|^2 < (n - 1/2)/n \},
\end{equation}
and let $\phi$ denote the function  $\phi(y) = (1 - |y|^2)^{1/2}$ for all 
$y \in V$. The set $U = \{ (y, \phi(y)) \in \R^n : y \in V \}$ is an open patch 
in $\SS^{n - 1}$, and 
\[
\omega \in U \Rightarrow |e_n \cdot \omega|^2 > 1/(2n).
\label{sta:boundedness_patch}
\]
From this patch, we obtain an atlas of $\SS^{n - 1}$. To do so, consider the reflections $R_{j n}$ and $P_j$ with $j \in \{ 1, \dots, n \}$ that satisfy: 
\begin{enumerate}[label=(\roman*)]
\item $R_{j n} e_j = e_n$, $R_{j n} e_n = e_j$ and $R_{j n} e_l = e_l$ for $l \notin \{ j, n \}$,
\item $P_j e_j = - e_j$ and $P_j e_l = e_l$ for $l \neq j$.
\end{enumerate}
With this choice $R_{n n} = \id$.
The atlas that we consider consists of $2n$ charts denoted by $(U_1, \varphi_1), \dots, (U_{2n}, \varphi_{2n})$. The odd charts are defined by
\begin{equation}\label{id:odd_charts} 
\begin{split}
U_{2j - 1} &= \{ R_{jn} \omega : \omega \in U \}, 
\\ 
\varphi_{2j - 1} (\theta) &= (e_1 \cdot R_{jn}^\T \theta, \dots, e_{n-1} \cdot R_{jn}^\T \theta),\quad  \theta \in U_{2j - 1},
\end{split}
\end{equation}
with $j \in \{ 1, \dots, n \}$.
The even charts are given by
\begin{equation}\label{id:even_charts} 
\begin{split}
 U_{2j} &= \{ P_j R_{jn} \omega : \omega \in U \}, 
 \\ 
 \varphi_{2j} (\theta) &= (e_1 \cdot R_{jn}^\T P_j^\T \theta, \dots, e_{n-1} \cdot R_{jn}^\T P_j^\T \theta),  \quad   \theta \in U_{2j},
\end{split}
\end{equation}
with $j \in \{ 1, \dots, n \}$. {This choice ensures that 
$\varphi_m(U_m) =V$ for every $m \in \{1, \dots, 2n  \}$.}

These charts satisfy some simple properties that we will refer to in a later analysis. For the odd charts we have that if $j \in \{ 1, \dots, n - 1 \}$ then
\begin{equation}
\zeta \in V \Rightarrow
\begin{cases}
e_l \cdot \varphi_{2j-1}^{-1} (\zeta) = \zeta_l,&  \quad\forall l\in \{ 1, \dots, n - 1 \}\setminus\{j\}, 
\vspace{.7em}\\
e_n \cdot \varphi_{2j-1}^{-1} (\zeta) = \zeta_j.&  
\end{cases}
\label{id:odd_j}
\end{equation}
Additionally,
\begin{equation}
\zeta \in V \Rightarrow
\begin{aligned}
& e_l \cdot \varphi_{2n-1}^{-1} (\zeta) = \zeta_l\qquad\forall l\in \{ 1, \dots, n - 1 \}.
\end{aligned}
\label{id:odd_n}
\end{equation}
Moreover,
\begin{equation}
\theta \in U_{2j - 1} \Rightarrow |e_j \cdot \theta|^2 > 1/(2n),
\label{sta:boundedness_patch_odd_theta}
\end{equation}
or equivalently, 
\begin{equation}
\zeta \in {V}
\Rightarrow 
|e_j \cdot \varphi_{2j - 1}^{-1}(\zeta)|^2 > 1/(2n).
\label{sta:boundedness_patch_odd}
\end{equation}
For the even charts we similarly have that if 
$j \in \{ 1, \dots, n - 1 \}$ then
\begin{equation}
\zeta \in V \Rightarrow
\begin{cases}
 e_l \cdot \varphi_{2j}^{-1} (\zeta) = \zeta_l, & \quad \forall l\in \{ 1, \dots, n - 1 \}\setminus\{j\}, 
 \vspace{.7em}\\
 e_n \cdot \varphi_{2j}^{-1} (\zeta) = \zeta_j. &
\end{cases}
\label{id:even_j}
\end{equation}
In addition,
\begin{equation}
\zeta \in V \Rightarrow
\begin{aligned}
& e_l \cdot \varphi_{2n}^{-1} (\zeta) = \zeta_l \qquad \forall l\in \{ 1, \dots, n - 1 \}.
\end{aligned}
\label{id:even_n}
\end{equation}
Moreover,
\begin{equation}
\theta \in U_{2j} \Rightarrow |e_j \cdot \theta|^2 > 1/(2n),
\label{sta:boundedness_patch_even_theta}
\end{equation}
or equivalently, 
\begin{equation}
\zeta \in {V} 
\Rightarrow |e_j \cdot \varphi_{2j}^{-1}(\zeta)|^2 > 1/(2n).
\label{sta:boundedness_patch_even}
\end{equation}

After having introduced the atlas $\{(U_1, \varphi_1), \dots, (U_{2n}, \varphi_{2n})\}$ for $\SS^{n - 1}$, let us focus on the task of proving the limit in \eqref{lim:vanishing_Cauchy}. By \eqref{id:smooth_homogeneity} and  \Cref{prop:norm_charts} we have for all $ \lambda,t \in \R_+ $ and $s \ge 0$, 
that
\[
\begin{split}
\lambda^\alpha \| [(q_{\alpha, t+s}^{-1} & - q_{\alpha, t}^{-1}) \otimes \mathbf{1}_{\R^{n-k}}](\lambda \centerdot ) \|_{H^{- \delta}(\SS^{n-1})} \\
& = \| (q_{\alpha, \lambda(t+s)}^{-1} - q_{\alpha, \lambda t}^{-1}) \otimes \mathbf{1}_{\R^{n-k}} \|_{H^{- \delta}(\SS^{n-1})} \\
& \eqsim \sum_{i = 1}^{2n} \| [\chi_i (q_{\alpha, \lambda(t+s)}^{-1} - q_{\alpha, \lambda t}^{-1}) \otimes \mathbf{1}_{\R^{n-k}}]\circ \varphi_i^{-1} \|_{H^{-\delta}(\R^{n - 1})}.
\end{split}
\]
By \eqref{sta:boundedness_patch_odd_theta} and \eqref{sta:boundedness_patch_even_theta}, the function defined in \eqref{map:qalpha-1_function} is bounded in $U_1 \cup \dots \cup U_{2k}$, so we bound the first $2k$ terms of the previous sum as follows
\[
\begin{split}
\sum_{i = 1}^{2k} \| [\chi_i (q_{\alpha, \lambda(t+s)}^{-1} &- q_{\alpha, \lambda t}^{-1}) \otimes \mathbf{1}_{\R^{n-k}}]\circ \varphi_i^{-1} \|_{H^{-\delta}(\R^{n - 1})} \\
& \lesssim \sum_{i = 1}^{2k} \| [\chi_i (q_{\alpha, \lambda(t+s)}^{-1} - q_{\alpha, \lambda t}^{-1}) \otimes \mathbf{1}_{\R^{n-k}}]\circ \varphi_i^{-1} \|_{L^2(\R^{n - 1})} \\
& \lesssim \sum_{i = 1}^{2k} \| (q_{\alpha, \lambda(t+s)}^{-1} - q_{\alpha, \lambda t}^{-1}) \otimes \mathbf{1}_{\R^{n-k}} \|_{L^\infty (U_i)}.
\end{split}
\]
Now, by  \eqref{lim:uniform_convergence_compacts}, \eqref{sta:boundedness_patch_odd_theta} and \eqref{sta:boundedness_patch_even_theta} we have that
\[
\lim_{t \to \infty} \sum_{i = 1}^{2k} \| (q_{\alpha, \lambda(t+s)}^{-1} - q_{\alpha, \lambda t}^{-1}) \otimes \mathbf{1}_{\R^{n-k}} \|_{L^\infty (U_i)} = 0 
\]
uniformly for in $s \ge 0$. This corresponds to the first $2k$ charts. {We turn} 
now our attention to the last $2(n - k)$ charts. From the identities 
\eqref{id:odd_j}, \eqref{id:odd_n}, \eqref{id:even_j} and \eqref{id:even_n} we 
know that if $ i  \in \{ 2k + 1, \dots, 2n \}$ then
\[ (q_{\alpha, t}^{-1} \otimes \mathbf{1}_{\R^{n - k}})(\varphi_i^{-1}(\zeta)) =
(q_{\alpha, t}^{-1} \otimes \mathbf{1}_{\R^{n - k - 1}})(\zeta) \qquad \forall \, \zeta \in \varphi_i (U_i), \]
and consequently that
\[
\begin{split}
\sum_{i = 2k + 1}^{2n} \| [\chi_i (q_{\alpha, \lambda(t+s)}^{-1} &- q_{\alpha, \lambda t}^{-1}) \otimes \mathbf{1}_{\R^{n-k}}]\circ \varphi_i^{-1} \|_{H^{-\delta}(\R^{n - 1})} \\
= & \sum_{i = 2k + 1}^{2n} \|\tilde{\chi}_i (q_{\alpha, \lambda(t+s)}^{-1} - q_{\alpha, \lambda t}^{-1}) \otimes \mathbf{1}_{\R^{n-k-1}} \|_{H^{-\delta}(\R^{n - 1})},
\end{split}
\]
where $\tilde{\chi}_i(\zeta) = \chi_i \circ \varphi_i^{-1} (\zeta)$ for $\zeta \in \R^{n - 1}$. Additionally,
\[
\begin{split}
\big\|\tilde{\chi}_i (q_{\alpha, \lambda(t+s)}^{-1} &- q_{\alpha, \lambda t}^{-1}) \otimes \mathbf{1}_{\R^{n-k-1}} \big\|_{H^{-\delta}(\R^{n - 1})}
\\
& \leq \bigg( \int_{\R^{n-k-1}} \big\| \tilde{\chi}_i (\centerdot, \kappa) (q_{\alpha, \lambda(t+s)}^{-1} - q_{\alpha, \lambda t}^{-1}) \big\|_{H^{-\delta}(\R^k)}^2 \, \dd \kappa \bigg)^{1/2}.
\end{split}
\]
By \eqref{lim:H-deltaLOC}, we know that
\[ 
\lim_{t \to \infty} \sum_{i = 2k + 1}^{2n} \bigg( \int_{\R^{n-k-1}} \big\| \tilde{\chi}_i (\centerdot, \kappa) (q_{\alpha, \lambda(t+s)}^{-1} - q_{\alpha, \lambda t}^{-1})\big \|_{H^{-\delta}(\R^k)}^2 \, \dd \kappa\bigg )^{1/2} = 0 
\]
uniformly in $s \ge 0$. This corresponds to the last $2(n - k)$ charts.
Summing up, we have proved that for every $\lambda > 0$
\begin{equation} \label{lim:stronger_I}
\lim_{t \to \infty} \lambda^\alpha \big\| [(q_{\alpha, t+s}^{-1}  - q_{\alpha, t}^{-1}) \otimes \mathbf{1}_{\R^{n-k}}](\lambda \centerdot ) \big\|_{H^{- \delta}(\SS^{n-1})} = 0
\end{equation}
uniformly in $s \ge 0$, which proves \eqref{lim:vanishing_Cauchy}. 

{As a consequence of \eqref{lim:stronger_I} in the case $\lambda =1$, we can 
derive that \eqref{in:r_alpha_convergence} holds. Indeed, thanks to 
\eqref{lim:stronger_I} we have that, for every sequence 
$\{ a_N \in \R_+ : N \in \N \}$ such that $\lim_{N\to \infty} a_N = \infty$,
the sequence 
$\{ [ q_{\alpha, a_N}^{-1} \otimes \mathbf{1}_{\R^{n-k} }]( \centerdot): N \in \N \}$ 
is Cauchy in $H^{-\delta}(\SS^{n-1})$. Moreover, given two sequences 
$\{ a_N \in \R_+ : N \in \N \}$ and $\{ b_N \in \R_+ : N \in \N \}$ such that 
$\lim_{N\to \infty} a_N = \lim_{N\to \infty} b_N = \infty$, we have that the 
sequence $\{ c_N \in \R_+ : N \in \N \}$ with 
$c_{2L - 1} = a_L$ and $c_{2L} = b_L$ for 
$L \in \N $ also satisfies that $\lim_{N\to \infty} c_N = \infty$. Therefore, 
the sequence 
$\{ [ q_{\alpha, c_N}^{-1} \otimes \mathbf{1}_{\R^{n-k} }]( \centerdot): N \in \N \}$ 
is also Cauchy in $H^{-\delta}(\SS^{n-1})$. Hence, the Cauchy sequences
$\{ [ q_{\alpha, a_N}^{-1} \otimes \mathbf{1}_{\R^{n-k} }]( \centerdot): N \in \N \}$
and
$\{ [ q_{\alpha, b_N}^{-1} \otimes \mathbf{1}_{\R^{n-k} }]( \centerdot): N \in \N \}$
belong to the same equivalence class. Consequently, all these Cauchy sequences in $H^{-\delta}(\SS^{n-1})$ belong to the same equivalence class.
Let $r_\alpha$ be the equivalence class determined by  these Cauchy sequences. Then, for every sequence 
$\{ a_N \in \R_+ : N \in \N \}$ such that $\lim_{N\to \infty} a_N = \infty$, we have that
\begin{equation} \label{lim:no_lambda}
\lim_{N \to \infty}  \big\|  [ q_{\alpha, a_N}^{-1} \otimes \mathbf{1}_{\R^{n-k}}]( \centerdot)  - r_\alpha \big\|_{H^{-\delta}(\SS^{n-1})} =0.
\end{equation}
Eventually, by \eqref{id:smooth_homogeneity} we have for all $ \lambda \in \R_+$  that
\[
 \lim_{N \to \infty}  \big \|  [  \lambda^\alpha [ q_{\alpha, N}^{-1} \otimes \mathbf{1}_{\R^{n-k}}](\lambda \centerdot) - r_\alpha \big\|_{H^{-\delta}(\SS^{n-1})} = \lim_{N \to \infty} \big\|  [ q_{\alpha, \lambda N}^{-1} \otimes \mathbf{1}_{\R^{n-k}}](  \centerdot) - r_\alpha \big\|_{H^{-\delta}(\SS^{n-1})} =0 .
\]
This proves \eqref{in:r_alpha_convergence}.

Finally, we address \eqref{in:domination_sequence}. In order to show that it holds, note that the function
\[
t \in (0 , \infty) \longmapsto   \| [q_{\alpha, t}^{-1} \otimes \mathbf{1}_{\R^{n-k}}]( \centerdot ) \|_{H^{- \delta}(\SS^{n-1})}
\]
is continuous and by \eqref{lim:conv_t_small} it follows  that
\[\lim_{t \to 0} \| [q_{\alpha, t}^{-1} \otimes \mathbf{1}_{\R^{n-k}}]( \centerdot ) \|_{H^{- \delta}(\SS^{n-1})} = 0. \]
This together with \eqref{lim:no_lambda} imply that there is a constant $C$ that only depends on $k$, $n$ and $\delta$ such that 
\[
\sup_{t>0}\big \| [q_{\alpha, t}^{-1} \otimes \mathbf{1}_{\R^{n-k}}]( \centerdot )\big \|_{H^{- \delta}(\SS^{n-1})} \le C.
\]
Again the identity \eqref{id:smooth_homogeneity} implies that 
\[
\lambda^\alpha \big\| [q_{\alpha, N}^{-1} \otimes \mathbf{1}_{\R^{n-k}}](\lambda \centerdot ) \big\|_{H^{- \delta}(\SS^{n-1})} \leq C
\]
for all $\lambda \in \R_+$ and $N \in \N$.
This concludes the proof of \eqref{in:domination_sequence}.}

\subsection{The distribution \texorpdfstring{$r_\alpha$}{ralpha} on \texorpdfstring{$\SS^{n-1}$}{Sn-1}}\label{sec:ralph_S}
On the one hand, $r_\alpha$ determines a unique distribution 
on $\SS^{n - 1}$ (see \Cref{cor:distribution}). On the other hand, distributions 
on $\SS^{n - 1}$ are a family of distributions in open subsets of $\R^{n-1}$ 
associated to local charts (see \Cref{sec:distributions}). Therefore, in order to 
describe the distribution on $\SS^{n - 1}$ determined by $r_\alpha$, we only have 
to compute the distributional limit of
\[ \lim_{N \to \infty} [q_{\alpha, N}^{-1} \otimes \mathbf{1}_{\R^{n-k}}] \circ \varphi_m^{-1} \]
in $\varphi_m(U_m)$, where $\{ (U_1, \varphi_1), \dots, (U_{2n}, \varphi_{2n})\}$ 
is the atlas introduced in \Cref{sec:proof_I_II_III}.
For convenience, recall that $ \varphi_m(U_m) = V$ for all 
$m \in \{1, \dots, 2n \}$, where $V$ was defined in \eqref{id:setV}.

\begin{proposition} \sl \label{prop:r_formula}
Let $r_{\alpha, m} \in \mathcal{D}^\prime (V)$ with $m \in \{1, \dots, 2n \}$ be defined by
\[
\langle r_{\alpha, m}, \phi \rangle = 
\begin{cases}\displaystyle
 \int_{V} \frac{\phi(\zeta)}{q_\alpha \otimes \mathbf{1}_{\R^{n-k}} (\varphi_m^{-1}(\zeta))}  \, \dd  \zeta,&\quad   1 \leq m \leq 2k, 
\vspace{.7em}\\
 \langle q_\alpha^{-1} \otimes \mathbf{1}_{\R^{n-k-1}}, \phi \rangle,&\quad  	 2k < m \leq 2n,
\end{cases}
\]
for all $\phi \in \mathcal{D} (V) $, where 
$q_\alpha \otimes \mathbf{1}_{\R^{n-k}} (\xi) = q_\alpha (e_1 \cdot \xi,
\dots, e_k \cdot \xi)$
for all $\xi \in \R^n$. Then, the sequence 
$ \{ r_\alpha^{N} \in C^\infty (\SS^{n-1}) : N \in \N\}$ given by
\[
 r_\alpha^N (\theta) = (q_{\alpha, N}^{-1} \otimes \mathbf{1}_{\R^{n - k}}) (\theta), \qquad  \theta \in \SS^{n - 1} ,
 \]
satisfies that
\begin{equation}
\lim_{N \to \infty} \int_{V} r_\alpha^{N}(\varphi_m^{-1}(\zeta)) \phi(\zeta) \, \dd \zeta = \langle r_{\alpha, m}, \phi \rangle,
\label{lim:rNalpha}
\end{equation}
for all $\phi \in \mathcal{D} (V) $ and $m \in \{1, \dots, 2n  \}$.
\end{proposition}
Note that the integral in $V$ for 
$1 \leq m \leq 2k$ is finite for all $\phi \in \mathcal{D} (V) $ since 
\eqref{sta:boundedness_patch_odd} and \eqref{sta:boundedness_patch_even} imply 
$|q_\alpha \otimes \mathbf{1}_{\R^{n-k}} (\varphi_m^{-1}(\zeta))| > 0$ for all
$\zeta \in V$.

\begin{proof}
{By definition,
\[ \int_V r_\alpha^{ N}(\varphi_m^{-1}(\zeta)) \phi(\zeta) \, \dd \zeta = \int_V (q_{\alpha,   N}^{-1} \otimes \mathbf{1}_{\R^{n - k}})(\varphi_m^{-1}(\zeta)) \phi(\zeta) \, \dd \zeta. \]
From \eqref{lim:uniform_convergence_compacts}, 
\eqref{sta:boundedness_patch_odd_theta} and 
\eqref{sta:boundedness_patch_even_theta} we know that whenever $1 \leq m \leq 2k$
\[\lim_{N \to \infty} \sup_{\theta \in \overline{U_m}} \Big| \frac{1}{q_\alpha \otimes \mathbf{1}_{\R^{n-k}} (\theta)} - (q_{\alpha,   N}^{-1} \otimes \mathbf{1}_{\R^{n - k}}) (\theta)\Big| = 0,\]
and consequently
\[
\lim_{N \to \infty} \int_V r_\alpha^{  N}(\varphi_m^{-1}(\zeta)) \phi(\zeta) \, \dd \zeta = \int_V \frac{\phi(\zeta)}{q_\alpha \otimes \mathbf{1}_{\R^{n-k}} (\varphi_m^{-1}(\zeta))} \, \dd \zeta.
\]
This proves \eqref{lim:rNalpha} whenever $1 \leq m \leq 2k$.
We focus now on the case $2k < m \leq 2n$. From the identities \eqref{id:odd_j}, 
\eqref{id:odd_n}, \eqref{id:even_j} and \eqref{id:even_n} we know that
\[ (q_{\alpha,  N}^{-1} \otimes \mathbf{1}_{\R^{n - k}})(\varphi_m^{-1}(\zeta)) =
(q_{\alpha, N}^{-1} \otimes \mathbf{1}_{\R^{n - k - 1}})(\zeta) \quad \forall \, \zeta \in V, \]
and consequently that
\[
\begin{split}
\int_V r_\alpha^{  N}(\varphi_m^{-1}(\zeta)) \phi(\zeta) \, \dd \zeta &= \int_V (q_{\alpha,   N}^{-1} \otimes \mathbf{1}_{\R^{n - k - 1}})(\zeta) \phi(\zeta) \, \dd \zeta
\\
& = \int_{\R^k} q_{\alpha,   N}^{-1}(\eta) \Big( \int_{\R^{n-k-1}} \phi(\eta, \kappa) \, \dd \kappa \Big) \, \dd \eta.
\end{split}
\]
By \eqref{lim:tempered_distributions}, we know that
\[
\begin{split}
\lim_{n \to \infty} \int_{\R^k} q_{\alpha,   N}^{-1}(\eta) & \Big( \int_{\R^{n-k-1}} \phi(\eta, \kappa) \, \dd \kappa \Big) \, \dd \eta = \langle q_\alpha^{-1}, \int_{\R^{n-k-1}} \phi(\centerdot, \kappa) \, \dd \kappa \rangle
\\
& = \langle q_\alpha^{-1} \otimes \mathbf{1}_{\R^{n-k-1}}, \phi \rangle.
\end{split}
\]
This means that
\[
\lim_{n \to \infty} \int_V r_\alpha^{  N}(\varphi_m^{-1}(\zeta)) \phi(\zeta) \, \dd \zeta = \langle q_\alpha^{-1} \otimes \mathbf{1}_{\R^{n-k-1}}, \phi \rangle, 
\]
which proves \eqref{lim:rNalpha} when $2k < m \leq 2n$.
This completes the proof of \Cref{prop:r_formula} for every $m \in \{1, \dots, 2n  \}$.}
\end{proof}

\section{Proofs of Theorems \texorpdfstring{\ref{th:loss_of_derivatives}}{loss} and \texorpdfstring{\ref{th:cancellation_counterexample}}{counter-ex}}
\label{sec:counterexamples}
This section contains two subsections devoted to the proofs of Theorem
\ref{th:loss_of_derivatives} and \ref{th:cancellation_counterexample}, 
respectively.

\subsection{Loss of derivatives when averaging in \texorpdfstring{$L^2(\SO(n))$}{L2(SO(n))}}
\label{sec:loss_of_derivatives}

In this section we prove {\Cref{th:loss_of_derivatives}}, which asserts 
a loss of  at least $s$ derivatives for the $k$-plane Riesz potential of order 
$\alpha $ with $s = \alpha - k/2 $ if $\alpha > k/2$ and $s > 0 $ if
$\alpha = k/2$. Before providing the explicit expression of the sequence
$\{ f_N : N \in \N \}$ stated in that theorem, note that by 
identities \eqref{id:solution_map_multiplier} and \eqref{id:mAlphaK} with $p_\alpha (\xi, Q) = (|\xi \cdot Q e_1|^2 + \dots + |\xi \cdot Q e_k|^2)^{\alpha/2}$ we have
\[
\begin{split}
\int_{\SO(n)} & \| (-\Delta)^{\alpha/2} I_\alpha  f_N(\centerdot,Q) \|^2_{L^2 (B_R)} \, \dd \mu(Q) 
\\
& = \int_{\SO(n)} \int_{B_R}  \Big| \frac{1}{(2\pi)^{n/2}} \int_{\R^n} e^{\i x \cdot \xi} m_\alpha (\xi, Q) \widehat{f_N}(\xi) \, \dd \xi \Big|^2 \, \dd x \, \dd \mu(Q)
\\
& = \int_{B_R} \int_{\R^n \times \R^n} e^{\i x \cdot (\xi - \eta)} K_\alpha (\xi, \eta) \widehat{f_N}(\xi) \overline{\widehat{f_N}(\eta)} \, \dd (\xi, \eta) \, \dd x,
\end{split}
\]
where
\[
K_\alpha (\xi, \eta) = \frac{1}{(2\pi)^n} \int_{\SO(n)} m_\alpha (\xi, Q) m_\alpha (\eta, Q) \, \dd \mu(Q),\qquad (\xi,\eta)\in\R^n\times \R^n.
\]
Since $x$ is localized in $B_R$, we will localize $\xi - \eta$ in a ball 
of radius of the order $1/R$ so that $e^{ix \cdot (\xi - \eta)}$
does not oscillate. Additionally, we want $\widehat{f_N}$ to be suitably supported on the high frequencies of the order $N$. Thus we choose
\[
\widehat{f_N}(\xi) = R^{n/2} \mathbf{1}_{B(Ne_1; \pi/(8R))}(\xi), \qquad \xi \in \R^n. 
\]
Here $\mathbf{1}_{B(Ne_1; \pi/(8R))}(\xi)$ denotes the characteristic function of the ball $B(Ne_1; \pi/(8R))$
centred at $Ne_1$ and {having} radius $\pi/(8R)$.
{The normalization factor  $R^{n/2}$  {easily implies the estimate}}
$\| f_N \|_{\dot{H}^s(\R^n)} \eqsim N^s $ whenever $N > \pi/(4R)$.
With this choice for $f_N$ we have that
\[
\begin{split}
\int_{\SO(n)} & \| (-\Delta)^{\alpha/2} I_\alpha  f_N(\centerdot,Q) \|^2_{L^2 (B_R)} \, \dd \mu(Q)
 \\
& = R^n \int_{B_R} \int_{\Sigma_N^R} \cos (x \cdot (\xi - \eta)) K_\alpha (\xi, \eta) \, \dd (\xi, \eta) \, \dd x,
\end{split}
\]
where $\Sigma_N^R = B(Ne_1; \pi/(8R)) \times B(Ne_1; \pi/(8R)) $. Because 
of the localization $x \in B_R$ and $(\xi, \eta) \in \Sigma_N^R$, we have 
that $\cos (x \cdot (\xi - \eta)) \geq 1/\sqrt{2}$ and consequently
\[
\int_{\SO(n)} \| (-\Delta)^{\alpha/2} I_\alpha  f_N(\centerdot,Q) \|^2_{L^2 (B_R)} \, \dd \mu(Q) \gtrsim R^{2n} \int_{\Sigma_N^R} K_\alpha (\xi, \eta) \, \dd (\xi, \eta).
\]
By the translations $\zeta = \xi - Ne_1 $ and $\kappa = \eta - N e_1$, 
the right-hand side of the previous inequality {is equal to}
\[ R^{2n} \int_{\Sigma_0^R} K_\alpha (Ne_1 + \zeta, Ne_1 + \kappa) \, \dd (\zeta, \kappa). \]
The fact that $\zeta \in B(0; \pi/(8R)) $ {implies} that
\[m_\alpha (Ne_1 + \zeta, Q) \gtrsim \frac{1}{(|e_1 \cdot Q e_1|^2 + \dots + |e_1 \cdot Q e_k|^2 + (RN)^{-2})^{\alpha/2}} \]
whenever $N > \pi / (2R)$. Thus,
\[
\begin{split}
R^{2n} \int_{\Sigma_0^R} & K_\alpha (Ne_1 + \zeta, Ne_1 + \kappa) \, \dd (\zeta, \kappa) \\
& \gtrsim \int_{\SO(n)} \frac{1}{(|e_1 \cdot Q e_1|^2 + \dots + |e_1 \cdot Q e_k|^2 + (RN)^{-2})^\alpha} \, \dd \mu(Q) \\
& \eqsim \int_{\SS^{n - 1}} \frac{1}{(|\theta \cdot e_1|^2 + \dots + |\theta \cdot e_k|^2 + (RN)^{-2})^\alpha} \, \dd S(\theta).
\end{split}
\]
{In the last step we have used the well known identity for the Haar measure stated  in \Cref{prop:SOn_Sn-1}.} Finally, in order to {complete} the proof of {\Cref{th:loss_of_derivatives}} we {show} that the integral on the sphere has a lower bound of the order $\mathfrak{g} (RN)$ whenever $N > \pi/(2 R)$.

{In the statement of the lemma below we remember  that the growth function $\mathfrak{g}$ is defined as 
$\mathfrak{g} (t) = \log t$ if $\alpha = k/2$ and 
$\mathfrak{g} (t) = t^{2 \alpha - k}$ if $\alpha > k/2$.}

\begin{lemma}\label{lem:growth}\sl We have that
\[\int_{\SS^{n - 1}} \frac{1}{(|\theta \cdot e_1|^2 + \dots + |\theta \cdot e_k|^2 + \varepsilon^2)^\alpha} \, \dd S(\theta) \gtrsim \mathfrak{g} (1/\varepsilon) \]
whenever $\varepsilon < 1$.
\end{lemma}


\begin{proof}
{To prove this lemma we could use two charts consisting on half of the sphere
each one. However, for simplicity we use again  the same atlas introduced in \eqref{id:setV},\eqref{id:odd_charts} and \eqref{id:even_charts}.}
Let $\{ \chi_1, \dots, \chi_{2n} \} $  be
a partition of unity subordinated to this atlas such that
$\chi_j (\varphi_j^{-1}(y)) \geq c > 0$ whenever $|y|^2 \leq (n - 1)/n$. {By \eqref{id:setV} this condition can be satisfied for all $y\in V$}.
We start by neglecting the charts where the {integrand} is bounded
\[
\begin{split}
\int_{\SS^{n - 1}} & \frac{1}{(|\theta \cdot e_1|^2 + \dots + |\theta \cdot e_k|^2 + \varepsilon^2)^\alpha} \, \dd S(\theta) \\
& \geq \sum_{j = 2k + 1}^{2n} \int_{\SS^{n - 1}} \frac{\chi_j(\theta)}{(|\theta \cdot e_1|^2 + \dots + |\theta \cdot e_k|^2 + \varepsilon^2)^\alpha} \, \dd S(\theta).
\end{split}
\]
{As mentioned, we have} neglected the first $2k$ charts because 
$|\theta \cdot e_j|^2 > 1/(2n)$ for all $\theta \in U_{2j - 1} \cup U_{2j}$ 
and $j \in \{ 1, \dots, k \}$.
{Expressing the  remaining integrals explicitly in coordinates} the right-hand 
side of the previous inequality is bounded below by
\[2(n - k)c \int_W \frac{(1 + |\nabla \phi(y)|^2)^{1/2}}{(|y^\prime|^2 + \varepsilon^2)^\alpha} \, \dd y \gtrsim \int_{W^\prime} \frac{1}{(|y^\prime|^2 + \varepsilon^2)^\alpha} \, \dd y^\prime \]
where $W = \{ y \in \R^{n - 1} : |y|^2 \leq (n - 1)/n \} \subset V$, 
$y = (y^\prime, y^{\prime\prime}) \in \R^k \times \R^{n - 1 - k}$ and
$W^\prime = \{ y^\prime \in \R^k : |y^\prime|^2 \leq (n - 1)/(2n) \}$. 
Finally, changing to spherical coordinates in $\R^k$, we see that right-hand
side of the last inequality {is of the order
$\mathfrak{g} (1/\varepsilon)$ when  $\varepsilon\to 0$}.
\end{proof}

\subsection{The cancellation property \texorpdfstring{\eqref{it:cancellation}}{cancellation} is necessary}
In this section we prove \Cref{th:cancellation_counterexample}, which implies the 
need of the cancellation property \eqref{it:cancellation} to ensure the maximum 
gain in the rotational smoothing.

Let $f \in \mathcal S(\R^n)$ be any function with $\| f \|_{L^2(\R^n)} =1$ such that $\widehat{f}$ has compact support in the set $\{\xi \in \R^n : 1< |\xi| < 2\}$. We then define $f_N$ so that
\[\widehat{f}_N(\xi) = N^{-n/2} \widehat{f}  (\xi/N) \]
holds. Thus, the support of $\widehat{f}_N $ is contained in the set $\{\xi \in \R^n : N < |\xi| < 2N\}$.

Consider the operator $S_{\alpha,t}$ defined in \eqref{id:Salpha}.  Then by a 
straightforward change of variables, and by the discussion in 
\Cref{sec:proof-mainTH} one has that
\begin{multline}  \label{id:S_alpha_t}
\| S_{\alpha,t} f_N \|_{H^{-\delta} (\SO(n); \dot{H}^\alpha(\R^n))}^2 =
\int_{\R^n} |\xi|^{2\alpha} \| [q_{\alpha, t}^{-1} \otimes \mathbf{1}_{\R^{n-k}}](|\xi| \centerdot ) \|_{H^{- \delta}(\SS^{n-1})}^2 |\widehat{f}_N(\xi)|^2 \, \dd \xi \\
=
\int_{ 1<|\xi|<2} N^{2\alpha} |\xi|^{2\alpha}   \| [q_{\alpha, t}^{-1} \otimes \mathbf{1}_{\R^{n-k}}](N|\xi| \centerdot ) \|_{H^{- \delta}(\SS^{n-1})}^2 |\widehat{f}(\xi)|^2 \, \dd \xi .
\end{multline}
From \eqref{eq:last_L} one gets, for all $\phi \in \mathcal S(\R^k)$, that
\begin{equation} \label{id:nonhomogeneous}
\langle q_\alpha^{-1},  \phi\rangle  = \lambda^{-\alpha} \langle q_\alpha^{-1},    \lambda^k \phi(\lambda\centerdot) \rangle  + \lambda^{-\alpha} \log \lambda \langle d_\alpha,    \lambda^k \phi(\lambda\centerdot) \rangle,
\end{equation}
where
\[
\langle d_\alpha,   \phi \rangle =  \int_{\R^k} \frac{1}{q_\alpha(\zeta)} (\zeta \cdot \nabla)^{\alpha - k + 1}   \phi  (\zeta) \, \dd \zeta.
\]  
Consider $q_{\alpha,t}^{-1}$ as defined in \eqref{id:qalphaN-1}, and let $d_{\alpha,t} (\eta) = \langle d_\alpha,  \psi_t (\eta - \centerdot) \rangle$ with 
$\psi_t (\eta) = t^k \psi (t \eta)$ $\eta \in \R^k$ and $t >0$.
For all $\eta \in \R^k$, $t >0$ and 
$\lambda >0$ we have by   \eqref{id:nonhomogeneous}  that
\[
\begin{split}
\lambda^\alpha q_{\alpha, t}^{-1} (\lambda \eta) &= \lambda^\alpha \langle q_\alpha^{-1}, \psi_t ( \lambda \eta - \centerdot) \rangle =
  \langle q_\alpha^{-1},    \lambda^k \psi_t ( \lambda \eta - \lambda \centerdot) \rangle  
  +   \log \lambda \langle d_\alpha,    \lambda^k \psi_t ( \lambda \eta - \lambda \centerdot) \rangle
\\
& =
  \langle q_\alpha^{-1},     \psi_{\lambda t} (  \eta -   \centerdot) \rangle  
  +   \log \lambda \langle d_\alpha,     \psi_{\lambda t} (  \eta -   \centerdot) \rangle = q_{\alpha, \lambda t}^{-1} (\eta) + \log \lambda \, d_{\alpha,\lambda t}(\eta).
\end{split}
\]
Thus, for every $\lambda ,t \in \R_+$ and $\theta\in\SS^{n-1}$ we have that
\begin{equation} \label{id:qasi_homogeneous}
 \lambda^\alpha [q_{\alpha, t}^{-1} \otimes \mathbf{1}_{\R^{n-k}}](\lambda \theta ) = [q_{\alpha, \lambda t}^{-1} \otimes \mathbf{1}_{\R^{n-k}}](\theta ) + \log \lambda [d_{\alpha, \lambda t}  \otimes \mathbf{1}_{\R^{n-k}}](\theta ).
\end{equation}

We now claim that 
\begin{equation} \label{eq:q_cauchy}
\lim_{t \to \infty} \big\| (q_{\alpha, t+s}^{-1}  - q_{\alpha, t}^{-1}) \otimes \mathbf{1}_{\R^{n-k}} \big\|_{H^{- \delta}(\SS^{n-1})} = 0,
 \end{equation}
 and
 \begin{equation} \label{eq:d_cauchy}
\lim_{t \to \infty}  \big\| (d_{\alpha, t+s}  - d_{\alpha, t}) \otimes \mathbf{1}_{\R^{n-k}} \big\|_{H^{- \delta}(\SS^{n-1})} = 0,
 \end{equation}
 uniformly for $s\ge 0$. The proof of this claim is, in both cases,  exactly the same as the proof of \eqref{lim:stronger_I} with $\lambda =1$, so we omit it.  
As a consequence of this, there exist some constants  $A>0$ and $B>0$ independent of $t$ and $\lambda$ such that
\begin{equation} \label{eq:claim_AB}
\| q_{\alpha, \lambda t}^{-1} \otimes \mathbf{1}_{\R^{n-k}}   \|_{H^{- \delta}(\SS^{n-1})} \le A \quad \text{and} \quad \| d_{\alpha, \lambda t} \otimes \mathbf{1}_{\R^{n-k}}  \|_{H^{- \delta}(\SS^{n-1})} \ge B
\end{equation}
for all $t>1$ and $\lambda\ge \lambda_0$ for a large enough $\lambda_0>0$. Notice that one needs \eqref{it:cancellation}  to be false in order to get the second estimate, since this guarantees that $d_\alpha$  and its restriction  to the sphere defined by \eqref{eq:d_cauchy} are non-vanishing (see the proof of \Cref{lem:pos_integers}).

From the previous discussion with   $\lambda = |\xi|N$  we get that  
\begin{align*} 
N^{2\alpha} |\xi|^{2\alpha}   \| [q_{\alpha, t}^{-1} &\otimes \mathbf{1}_{\R^{n-k}}](N|\xi| \centerdot ) \|_{H^{- \delta}(\SS^{n-1})}^2 \\
&=    \| q_{\alpha, N|\xi| t}^{-1} \otimes \mathbf{1}_{\R^{n-k}} + \log (|\xi|N) \, d_{\alpha, N|\xi| t}  \otimes \mathbf{1}_{\R^{n-k}}   \|_{H^{- \delta}(\SS^{n-1})}^2 \\
&\ge   \left( B \log (|\xi|N)  -  A \right)^2 \ge \left( B\log N  -  A \right)^2 \gtrsim (\log N)^2
\end{align*}
for all $1<|\xi|<2$, $t>1$, and for all $N \ge N_0$ for a large enough $N_0 \in \N$ independent of $\xi$ and $t$.  Combining this with \eqref{id:S_alpha_t}  yields
\[ 
\| S_{\alpha,t} f_N \|_{H^{-\delta} (\SO(n); \dot{H}^\alpha(\R^n))} \gtrsim \log N  ,
\]
 for all $t>1$ and $N\ge N_0$.  Since the previous estimate is uniform in $t$, to conclude the proof of the theorem it is enough to justify that we can pass to the limit $t$ tends to $\infty$. We now do this. First observe that \eqref{id:qasi_homogeneous},\eqref{eq:q_cauchy} and \eqref{eq:d_cauchy} imply that for any fixed $\lambda>0$
 \begin{equation*} 
\lambda^\alpha \lim_{t \to \infty} \big\| (q_{\alpha, t+s}^{-1}  - q_{\alpha, t}^{-1}) \otimes \mathbf{1}_{\R^{n-k}}(\lambda \centerdot) \big\|_{H^{- \delta}(\SS^{n-1})} = 0
 \end{equation*}
uniformly in $s\ge 0$. Hence, if $M,L \in \N$,
 \[
\begin{split}
&\lim_{M \to 0} S(\SS^{n-1})  \| S_{\alpha, M+L} f_N - S_{\alpha, M} f_N \|_{H^{-\delta}(\SO(n); \dot{H}^\alpha(\R^n))}^2
\\
&= \int_{N<|\xi|<2N} \lim_{M\to 0} |\xi|^{2 \alpha} \| [(q_{\alpha, |\xi|(M+L)}^{-1} - q_{\alpha, |\xi|M}^{-1}) \otimes \mathbf{1}_{\R^{n-k}}](|\xi| \centerdot ) \|_{H^{- \delta}(\SS^{n-1})}^2 |\widehat{f}_N(\xi)|^2 \, \dd \xi =0,
\end{split}
\]
for $N \ge N_0$. Since $|\xi| \simeq N$ in the domain of the integral, here the 
use of the dominated convergence theorem is allowed for $N$ large enough by the 
first estimate of \eqref{eq:claim_AB}. 
Therefore $\{S_{\alpha, M} f_N : M\in \N \}$ is a Cauchy sequence for each fixed 
$N\ge N_0$ so,  by the same  uniqueness on limit arguments explained at 
\Cref{sec:proof-mainTH}, we conclude that
\[ 
\| S_{\alpha} f_N \|_{H^{-\delta} (\SO(n); \dot{H}^\alpha(\R^n))}  = \lim_{M\to \infty}\| S_{\alpha,M} f_N \|_{H^{-\delta} (\SO(n); \dot{H}^\alpha(\R^n))} \gtrsim \log N  ,
\]
for all $N \ge N_0$. This concludes the proof of 
\Cref{th:cancellation_counterexample}.

\section{The functional analytical framework}\label{sec:abstract_framework} 
{In this section we adopt a rather abstract point of view and establish the 
analytical framework  that we have used in order to state and prove
\Cref{th:0alphak,th:kalpha}. The section is divided in two, 
\Cref{sec:distributions,sec:sobolev}. In 
\Cref{sec:distributions} we recall the notion of distributions on smooth 
manifolds and prove a technical lemma that will be very useful in 
\Cref{sec:sobolev}. In \Cref{sec:sobolev} we give an intrinsic description of 
fractional Sobolev spaces on compact  boundaryless Riemannian manifolds. Then, we 
introduce the mixed-norm Sobolev spaces that we have used in 
\Cref{th:0alphak,th:kalpha} in order to quantify the effect of rotational 
smoothing.}

\subsection{Distributions on a smooth manifolds}\label{sec:distributions} We begin by recalling the definition of a distribution on a smooth manifold $M$.
For convenience, if $U$ and $V$ are open subsets  of $\R^m$ with $m \in \N$, $U \subset V$ and $\phi \in \mathcal{D}(U)$, then we will still write $\phi$ to denote the extension of this function to $V$, defined to be zero in $V \setminus U$. In the same way, if $u \in \mathcal{D}^\prime (V)$ we still write $u$ to refer to the distribution $\phi \in \mathcal{D}(U) \mapsto \langle u, \phi \rangle $ that belongs to $\mathcal{D}^\prime (U)$. Let $(M, \mathcal{A})$ be a smooth manifold where $M$ is a topological manifold
of dimension $m \in \N$ and $\mathcal{A}$ is a maximal smooth atlas.
Assume that for every chart 
$(U_\alpha, \varphi_\alpha)\in \mathcal{A}$ there is a distribution $u_\alpha $ 
in $\varphi_\alpha (U_\alpha) \subset \R^m $. If for every 
$(U_\alpha, \varphi_\alpha)$ and $(U_\beta, \varphi_\beta)$ in $ \mathcal{A}$ 
with $U_\alpha \cap U_\beta \neq \emptyset$ we have that
\begin{equation}
\langle u_\alpha, \phi \rangle = \langle (\varphi_\beta \circ 
\varphi_\alpha^{-1})^\ast u_\beta, \phi \rangle\qquad \forall\, \phi \in 
\mathcal{D}(\varphi_\alpha (U_\alpha \cap U_\beta)),
\label{id:transition}
\end{equation}
then the family of all distributions $u_\alpha$ is said to be 
a distribution $u \in \mathcal{D}^\prime (M)$.
Given a distribution $u\in \mathcal{D}^\prime (M)$ it is convenient to denote the 
corresponding $u_\alpha$ by $u \circ \varphi_\alpha^{-1}$.
According to \cite[Theorem 6.3.4]{zbMATH01950198} a distribution
$u \in \mathcal{D}^\prime (M)$ 
is determined only by the family of distributions $u_\alpha$ corresponding to the 
charts $(U_\alpha, \varphi_\alpha)$ 
in an atlas $\mathcal{B}$ that is not necessarily maximal. More precisely, assume
that $\mathcal{B} \subset \mathcal{A}$ is an atlas and that for every chart 
$(U_\alpha, \varphi_\alpha)\in \mathcal{B}$ there is a distribution $u_\alpha $ 
in $\varphi_\alpha (U_\alpha) \subset \R^m $. If for every 
$(U_\alpha, \varphi_\alpha)$ and $(U_\beta, \varphi_\beta)$ in $ \mathcal{B}$ 
with $U_\alpha \cap U_\beta \neq \emptyset$ the identity \eqref{id:transition} 
holds, then there exists a unique $u \in \mathcal{D}^\prime(M)$ such that
$ u \circ \varphi_\alpha^{-1} = u_\alpha $ for all 
$(U_\alpha, \varphi_\alpha)\in \mathcal{B}$.

In the remaining of \Cref{sec:distributions} we will show that elements in the Banach completion of a normed 
subspace of the smooth functions on a manifold determine distributions in 
the manifold.

For convenience, we first recall briefly the definition of the Banach completion of a normed space $(V, \| \centerdot \|_V)$. Let $B$ denote the 
set of all equivalence classes of Cauchy sequences in the normed space obtained by the following relation: Given two Cauchy sequences $\{ x_N : N \in \N \}$ and 
$\{ y_N : N \in \N \}$ in $V$ we say that they are related if and only if
\[ 
\lim_{N \to \infty} \| x_N - y_N \|_V = 0.
\]
The vector space structure of $V$ naturally induces a vector structure on $B$. Furthermore, we can endow $B$ with a norm defined as follows: if $x \in B$ 
denotes the equivalence class of the Cauchy sequence $\{ x_N : N \in \N \} \subset V$, we define the norm
\[
\| x \|_B = \lim_{N \to \infty} \| x_N \|_V.
\]
Then, $(B, \| \centerdot \|_B)$ is a Banach space and contains a dense subspace isometric to $V$. The space $(B, \| \centerdot \|_B)$ is called the Banach completion of $(V, \| \centerdot \|_V)$.

In order to address the last endeavour of \Cref{sec:distributions}, let $(M, \mathcal{A})$ be a smooth manifold of dimension $m \in \N$. Assume that there exists an atlas $\{ (U_1, \varphi_1), \dots, (U_l, \varphi_l) \}$ with a finite number $l \in \N$ of charts and choose $\{ \chi_1, \dots, \chi_l \}$ a  smooth partition of unity subordinated to this atlas. For convenience, we agree from now on  that whenever $u \in C^\infty(M)$, the extension by zero of the  functions $(\chi_j u) \circ \varphi_j^{-1}$ to $\R^m \setminus \varphi_j(U_j)$ with $j \in \{ 1, \dots, l \}$ will be still denoted by $(\chi_j u) \circ\varphi_j^{-1}$. Let $\mathscr{V}(M)$ denote a vector subspace of $C^\infty (M)$ such that if $u \in \mathscr{V}(M)$, then $(\chi_j u) \circ \varphi_j^{-1} \in \mathcal{S}(\R^m)$ for all $j \in \{ 1, \dots, l \}$. Given a norm $\| \centerdot \|$ on $\mathcal{S}(\R^m)$ such that convergent sequences with respect to this norm also converge in $\mathcal{S}^\prime (\R^m)$, we endow $\mathscr{V}(M)$ with the norm defined by
\[
\| u \|_{\mathscr{V}(M)} =  \sum_{j = 1}^l\big \| (\chi_j u) \circ \varphi_j^{-1} \big\| ,\qquad u \in \mathscr{V}(M).
\]
Let $(\mathscr{B}(M), \| \centerdot \|_{\mathscr{B}(M)})$ denote the Banach completion of $(\mathscr{V}(M), \| \centerdot \|_{\mathscr{V}(M)})$. Under these conditions we can see that the elements of $\mathscr{B}(M)$ determine distributions in $M$.

\begin{lemma}\label{lem:distributions} \sl 
Assume $M$ and $\mathscr{B}(M)$ to be as above. Then every 
element in $\mathscr{B}(M)$ determines a unique distribution in $M$.
\end{lemma}

\begin{proof}
Let $u \in \mathscr{B}(M)$ be given and consider a Cauchy 
sequence $\{ u_N : N \in \N \} \subset \mathscr{V}(M)$ in the equivalence class 
denoted by $u$. We will show that $u$ determines a unique 
distribution in $M$. For every $\phi \in \mathcal{D} (\varphi_j (U_j))$ with 
$j \in \{1, \dots , l \}$ the following limit exists
\begin{equation}
\label{lim:distribution_chart}
\lim_{N \to \infty} \int_{\varphi_j (U_j)} u_N \circ \varphi_j^{-1}(x)  \phi(x) \, \dd x.
\end{equation}
Indeed, 
\[
\int_{\varphi_j (U_j)} u_N \circ \varphi_j^{-1}(x)  \phi(x) \, \dd x = 
\sum_{i = 1}^l \int_{\varphi_j (U_j)} (\chi_i u_N) \circ \varphi_j^{-1}(x)  \phi(x) \, \dd x.
\]
On the one hand if $U_i \cap U_j = \emptyset$, then
\[ \int_{\varphi_j (U_j)} (\chi_i u_N) \circ \varphi_j^{-1}(x)  \phi(x) \, \dd x = 0. \]
On the other hand, $U_i \cap U_j \neq \emptyset$, then
\[\int_{\varphi_j (U_j)} (\chi_i u_N) \circ \varphi_j^{-1}(x)  \phi(x) \, \dd x = 
\int_{\varphi_j (U_i \cap U_j)} (\chi_i u_N) \circ \varphi_j^{-1}(x)  \phi(x) \, \dd x.\]
Changing variable according to $x = \varphi_j \circ \varphi_i^{-1} (y)$ we have that the right-hand side equals
\begin{equation}
\label{term:integral_changed}
\int_{\varphi_i (U_i \cap U_j)} (\chi_i u_N) \circ \varphi_i^{-1}(y)  \phi \big( \varphi_j \circ \varphi_i^{-1} (y) \big) \big|\det \big(\nabla \varphi_j \circ \varphi_i^{-1} (y) \big) \big| \, \dd y.
\end{equation}
Recall that we have assumed that convergent sequences with respect to the norm 
$\| \centerdot \|$ also converge in $\mathcal{S}^\prime (\R^m)$. Then, since
$\{ (\chi_i u_N) \circ \varphi_i^{-1} : N \in \N \}$ with $i \in \{ 1, \dots, l \}$ so that $U_i \cap U_j \neq \emptyset$ are Cauchy sequences with 
respect to $\| \centerdot \|$, we know that they converge in $\mathcal{S}^\prime (\R^m)$. Therefore the limits of the integrals \eqref{term:integral_changed} with $i \in \{ 1, \dots, l \}$ so that $U_i \cap U_j \neq \emptyset$ exist as $N\to\infty$. This implies that the limit in \eqref{lim:distribution_chart} also exists.

Furthermore, the limit 
\eqref{lim:distribution_chart} is independent of the Cauchy sequence chosen in the equivalence class denoted by $u$. Namely, if 
$\{ v_N : N \in \N \}$ is another Cauchy sequence in the equivalence class of $u$, then
\[\lim_{N \to \infty} \int_{\varphi_j (U_j)} u_N \circ \varphi_j^{-1}(x)  \phi(x) \, \dd x = \lim_{N \to \infty} \int_{\varphi_j (U_j)} v_N \circ \varphi_j^{-1}(x)  \phi(x) \, \dd x.\]
Indeed,
\[
\begin{split}
& \bigg| \int_{\varphi_j (U_j)} u_N \circ \varphi_j^{-1}(x)  \phi(x) \, \dd x - \int_{\varphi_j (U_j)} v_N \circ \varphi_j^{-1}(x)  \phi(x) \, \dd x \bigg| 
\\
&\qquad \leq \sum_{i = 1}^l \bigg| \int_{\varphi_j (U_j)} (\chi_i u_N - \chi_i v_N) \circ \varphi_j^{-1}(x)  \phi(x) \, \dd x \bigg|.
\end{split}
\]
As we argued earlier, if $U_i \cap U_j = \emptyset$, then
\[ \int_{\varphi_j (U_j)} (\chi_i u_N - \chi_i v_N) \circ \varphi_j^{-1}(x)  \phi(x) \, \dd x = 0. \]
Otherwise, if $U_i \cap U_j \neq \emptyset$, then
\[
\begin{split}
\int_{\varphi_j (U_j)} & (\chi_i u_N - \chi_i v_N) \circ \varphi_j^{-1}(x)  \phi(x) \, \dd x \\
& = \int_{\varphi_j (U_i \cap U_j)} (\chi_i u_N - \chi_i v_N) \circ \varphi_j^{-1}(x)  \phi(x) \, \dd x.
\end{split}
\]
Changing variable according to $x = \varphi_j \circ \varphi_i^{-1} (y)$ we have that the right-hand side equals
\begin{equation}
\label{term:integral_changed_2}
\int_{\varphi_i (U_i \cap U_j)} (\chi_i u_N - \chi_i v_N) \circ \varphi_i^{-1}(y)   \psi(y) \, \dd y
\end{equation}
with $\psi(y) = \phi \big( \varphi_j \circ \varphi_i^{-1} (y) \big) \big|\det \big(\nabla \varphi_j \circ \varphi_i^{-1} (y) \big) \big|$.
Since $\{ (\chi_i u_N - \chi_i v_N) \circ \varphi_i^{-1} : N \in \N \}$ 
with $i \in \{ 1, \dots, l \}$ so that $U_i \cap U_j \neq \emptyset$ 
converges to zero with respect to $\| \centerdot \|$, we know that it also does 
so in $\mathcal{S}^\prime (\R^m)$, and therefore 
the limits of the integrals \eqref{term:integral_changed_2} with 
$i \in \{ 1, \dots, l \}$ so that $U_i \cap U_j \neq \emptyset$ converge to zero.

Thus, given $u \in \mathscr{B}(M)$, the limit \eqref{lim:distribution_chart} is 
well defined. Then we introduce $u^j \in \mathcal{D}^\prime (\varphi_j (U_j)) $ 
for $j \in \{ 1, \dots, l \}$ as
\[ \langle u^j, \phi \rangle = \lim_{N \to \infty} \int_{\varphi_j (U_j)} u_N \circ \varphi_j^{-1}(x)  \phi(x) \, \dd x, \qquad \forall \phi \in \mathcal{D} (\varphi_j (U_j)).\]
Note that if 
$(U_j, \varphi_j)$ and $(U_i, \varphi_i)$ satisfy 
$U_j \cap U_i \neq \emptyset$ then
\[ \int_{\varphi_j (U_j \cap U_i)} u_N \circ \varphi_j^{-1}(x)  \phi(x) \, \dd x = \int_{\varphi_j (U_j \cap U_i)} (\varphi_i \circ \varphi_j^{-1})^\ast (u_N \circ \varphi_i^{-1})(x)  \phi(x) \, \dd x \]
for all $\phi \in \mathcal{D}(\varphi_j (U_j \cap U_i))$. Then, by
\cite[Theorem 6.1.2]{zbMATH01950198} we have
\[\langle u^j, \phi \rangle = \langle (\varphi_i \circ 
\varphi_j^{-1})^\ast u^i, \phi \rangle, \qquad \forall\, \phi \in 
\mathcal{D}(\varphi_j (U_j \cap U_i)).\]
Consequently, \cite[Theorem 6.3.4]{zbMATH01950198} states that 
$\{ u^1, \dots, u^l \}$ determines a unique distribution in $M$.
This concludes the proof of this lemma.
\end{proof}
\subsection{A convenient description of Sobolev spaces}\label{sec:sobolev}
We start this section by describing an intrinsic notion of Sobolev spaces on a compact boundaryless Riemannian manifold $(M, g)$. Then we introduce the mixed-norm Sobolev spaces on $\R^n \times M$ used to capture the rotational smoothing phenomenon. 

\subsubsection{Sobolev spaces on compact boundaryless Riemannian manifolds}
{Let $(M,g)$ as above and denote the Laplace--Beltrami by  $\Delta_g$ as in 
\Cref{sec:funct_subsec}.} By Rellich's theorem, the mapping 
$ L^2(M)\ni f \mapsto u \in L^2(M)$, satisfying $-\Delta_g u = f$ in $M$,  is 
compact. Additionally, $-\Delta_g $ is self-adjoint and non-negative. By the 
spectral theorem there exists a sequence $\{ \lambda_j : j \in \N_0 \}$ of non-
negative real eigenvalues and an  orthonormal basis of $L^2(M)$ formed by 
eigenfunctions
$\{ \psi_j^i : j \in \N_0,\, i\in \N,\, i \leq d_j \}$ such that  
$\lambda_0 = 0$, $ \lambda_j < \lambda_{j+1} $, and  
$-\Delta_g \psi_j^i = \lambda_j \psi_j^i$ with $\psi_j^i \in C^\infty(M)$  for 
$i \in \{ 1, \dots, d_j \}$ and all $j\in \N_0$. Set 
$H_j = \vspan \{ \psi_j^i : i\in \N,\, i \leq d_j \}$ and define, for  
$u \in L^1(M)$, the projector $\Pi_j$ on $H_j$ as
\[
\Pi_j u (p) = \sum_{i = 1}^{d_j} \Big(\int_M \overline{\psi_j^i} u \, \dd \mu_g\Big) \psi_j^i (p), \qquad p \in M. 
\]
Here $\mu_g$ denotes the Riemannian density. Note that for $u \in C^\infty(M)$ we 
have that $\Pi_j[(\Id - \Delta_g) u] = (1 + \lambda_j) \Pi_j u $. Consequently,
\begin{equation}
\Pi_j [ (\Id - \Delta_g) \circ \overset{k}{\cdots} \circ (\Id - \Delta_g) u ] = (1 + \lambda_j)^{k} \Pi_j u  \qquad \forall\, u \in C^\infty(M).
\label{id:projector_kBessel}
\end{equation}
The fact that $u \in C^\infty(M)$ implies  $(\Id - \Delta_g) \circ \overset{k}{\cdots} \circ (\Id - \Delta_g) u \in L^2(M)$  for every $k \in \N_0$. This conclusion together with the identity \eqref{id:projector_kBessel} ensures that
\begin{equation}
(\Id - \Delta_g) \circ \overset{k}{\cdots} \circ (\Id - \Delta_g) u = \sum_{j \in \N_0} (1 + \lambda_j)^k \Pi_j u
\label{id:composition_Bessel}
\end{equation}
and for any $k\in \N_0$
\[
\sum_{j \in \N_0} (1 + \lambda_j)^{2k} \| \Pi_j u \|_{L^2(M)}^2 < \infty.
\]
Thus, we know that 
$ \{ \sum_{j = 0}^N (1 + \lambda_j)^{s/2} \Pi_j u : N \in \N_0 \}$  with 
$u \in C^\infty(M)$ 
is a Cauchy sequence in $L^2(M)$ for every $s \in \R$. The limit of this sequence 
defines the  Bessel potential of order $s$. {This is the definition given in 
\eqref{eq:bessel_pot} for $u \in C^\infty(M)$ and $s \in \R$.} Note that for 
$s \in \R$ and $u \in C^\infty(M)$, we have that
%
%
%
%
%
\begin{equation}
\label{id:Hs_projectors}
\|(\Id - \Delta_g)^{s/2}  u \|_{L^2(M)}
= \Big( \sum_{j \in \N_0} (1 + \lambda_j)^s \| \Pi_j u \|_{L^2(M)}^2 \Big)^{1/2}.
\end{equation}
Furthermore, the non-negative functional
\[
\| \centerdot \|_{H^s(M)} : u \in C^\infty(M) \longmapsto \|(\Id - \Delta_g)^{s/2}  u \|_{L^2(M)}
\]
defines a norm. Identity \eqref{id:Hs_projectors}  gives a simple way of computing $ \| u \|_{H^s(M)}$ whenever $u $ belongs to $ C^\infty (M)$.
Note that the norm for $s = 0$ coincides with the norm of ${L^2(M)}$. With this norm we define the Sobolev space of order $s$ in an intrinsic way.
\begin{definition}\sl For $s \in \R$ define $H^s(M)$  as the Banach  completion of the normed space
$(C^\infty(M), \| \centerdot \|_{H^s(M)})$. 
\end{definition}
By continuity, the Bessel potential of order $s$ can be defined for every 
$u \in H^s(M)$ and
\begin{equation}
\| u \|_{H^s(M)} = \| (\Id - \Delta_g)^{s/2} u \|_{L^2(M)},
\label{id:norm_Bessel}
\end{equation}
holds for all $u \in H^s(M)$.

{The intrinsic definition of the Sobolev spaces given above has an equivalent 
description in terms of local charts, as we have stated in 
\Cref{prop:norm_charts}.}
%
%
This proposition is a consequence of the results in
\cite{zbMATH03256081,zbMATH03814630}. Let us mention the main idea of the 
proof. Consider first the case where $s=2k$ with $k \in \N$. As a consequence of the identities \eqref{id:norm_Bessel} and 
\eqref{id:composition_Bessel} and the fact that, in every chart $(U_j, \varphi_j)$, the operator
$$(\Id - \Delta_g)^k = (\Id - \Delta_g) \circ \overset{k}{\cdots} \circ (\Id - \Delta_g)$$
is identified with a differential operator of order $2k$ in $\R^m$, one sees that for every $k \in \N$
\begin{equation}
\| u \|_{H^{2k}(M)} \lesssim \sum_{j = 1}^N \| (\chi_j u)\circ \varphi_j^{-1} \|_{H^{2k}(\R^m)}.
\label{in:upperbound}
\end{equation}
The reverse inequality follows from the well-posedness and the regularity theory for the equation $u - \Delta_g u = f$ in $M$. This is fully  detailed in the book \cite{zbMATH06271961}. This is a sketch of the proof in the cases where $s = 2 k$ with  $k \in \N$. For general $s \in \R$ one first needs to show that in every chart $(U_j, \varphi_j)$, the operator $(\Id - \Delta_g)^{s/2}$ is identified with an elliptic  pseudodifferential operator of order $s$ with a principal symbol given by $(\sum_{i,k = 1}^m g^{ik}(x) \xi_i \xi_k)^{s/2}$, for $x \in \varphi_j (U_j)$ and $|\xi| \gtrsim 1$ (see the details in
\cite{zbMATH03256081,zbMATH03814630}). Thus, the generalization of 
\eqref{in:upperbound} to $s \in \R$ follows from the boundedness of 
{pseudodifferential} operators of order $s$. The reverse inequality is {a}
consequence of the ellipticity and the possibility of constructing the 
corresponding parametrix for $(\Id - \Delta_g)^{s/2}$ in each chart.
This finishes {the sketch of} the
proof in the general case $s \in \R$.

We finish the discussion of Sobolev spaces on compact boundaryless Riemannian 
manifolds by showing that the elements of $H^s(M)$ are distributions in $M$.

\begin{corollary}\label{cor:distribution} \sl Let $(M, g)$ be a compact 
boundaryless Riemannian manifold of dimension $m \in \N$.
Every element in $ H^s(M)$ with $s \in \R$ determines a unique distribution in 
$M$.
\end{corollary}

\begin{proof}
This follows from \Cref{lem:distributions} and \Cref{prop:norm_charts}.
\end{proof}
\subsubsection{Mixed-norm Sobolev spaces}\label{sec:Plancherel}
As previously agreed, $(M, g)$ stands for a compact boundaryless 
Riemannian manifold of dimension $m \in \N$. {Recall the definition of $\mathcal{A}(\R^n \times M)$ given in \Cref{def:A}.}
%
%
As mentioned then, a function  
$u \in C^\infty (\R^n \times M) $ belongs to $ \mathcal{A}(\R^n \times M)$ if and only if its Fourier transform, denoted by $\widehat{u}$ or $\mathcal{F}u$ and defined by \eqref{eq:Fourier},
%
%
belongs to $\mathcal{A}(\R^n \times M)$. Consequently, for all $t \in \R_+$ and all $k \in \N_0$ there exists a constant $C>0$ such that
\begin{equation}
\label{in:bound_t-k}
\int_{\R^n} |\xi|^{2t} \| (\Id - \Delta_g)^k \widehat{u}(\xi, \centerdot) \|_{L^2(M)}^2 \, \dd \xi \leq C.
\end{equation}
Note that for $k=0$ we have
\[\int_M \| u(\centerdot, p) \|_{\dot{H}^t(\R^n)}^2 \, \dd \mu_g(p) = \int_{\R^n} |\xi|^{2t} \| \widehat{u}(\xi, \centerdot) \|_{L^2(M)}^2 \, \dd \xi. \]
This implies that for all $t \in \R_+$ we have that the non-negative functional
\[
 u \in \mathcal{A}(\R^n \times M)\longmapsto \bigg( \int_M \| u(\centerdot, p) \|_{\dot{H}^t(\R^n)}^2 \, \dd \mu_g(p) \bigg)^{1/2} 
 \]
is well defined. Furthermore, it defines norm on $\mathcal{A}(\R^n \times M)$ that 
will be denoted from now on by $\| \centerdot \|_{L^2(M; \dot{H}^t(\R^n))}$. 

\begin{definition} \sl For $t \in \R_+$, we define 
$L^2(M; \dot{H}^t(\R^n))$ as the Banach completion 
of the normed space $(\mathcal{A}(\R^n \times M), \| \centerdot \|_{L^2(M; \dot{H}^t(\R^n))})$.
\end{definition}

\begin{proposition}\label{prop:distributionL2Ht}\sl For $t \in \R_+$, we have that every 
$u \in L^2(M; \dot{H}^t(\R^n))$ determines a unique distribution in 
$\R^n \times M$.
\end{proposition}

\begin{proof} This statement is a straightforward conclusion of 
\Cref{lem:distributions}. This becomes clear when noting that $\R^n \times M$ is 
a smooth manifold that can be covered by an atlas 
$\{ (\R^n \times U_1, \id \otimes \varphi_1), \dots, (\R^n \times U_l, \id \otimes \varphi_l) \}$  
with a finite number $l \in \N$ of charts. Here the charts 
$(U_1, \varphi_1), \dots, (U_l, \varphi_l) $ form an atlas for the compact 
Riemannian manifold $M$. Then, we choose $\{\chi_1, \dots, \chi_l \}$ a smooth  
partition of unity subordinated to the atlas for $M$ and set 
$\mathscr{V}(\R^n \times M) = \mathcal{A}(\R^n \times M)$ with the norm 
\[
\| u \|_{\mathscr{V}(\R^n \times M)} =  \sum_{j = 1}^l \Big( \int_{\R^m} \| \chi_j\circ \varphi_j^{-1}(y) u(\centerdot, \varphi_j^{-1}(y)) \|_{\dot{H}^t(\R^n)}^2 |g_j(y)|^{1/2} \dd y \Big)^{1/2}, 
\]
where $|g_j(y)|$ denotes the determinant of the matrix representing the metric $g$ in the chart $(U_j, \varphi_j)$. The only point left to conclude this  proposition from \Cref{lem:distributions} is to observe that a sequence in $\mathscr{V}(\R^n \times M) = \mathcal{A}(\R^n \times M)$ is Cauchy with respect  to the norm $\| \centerdot \|_{\mathscr{V}(\R^n \times M)}$ if and only if is Cauchy with respect to $\| \centerdot \|_{L^2(M; \dot{H}^t(\R^n))}$.
\end{proof}

From the discussion on \Cref{sec:sobolev}, we know that 
$(\Id - \Delta_g)^s u(x, \centerdot) \in L^2(M) $ for all 
$u \in \mathcal{A}(\R^n \times M)$ and $s\in \R$. Furthermore,
whenever $k \in \N_0$ with $k > s$ we have
\[
\begin{split}
\| (\Id - \Delta_g)^s u(x, \centerdot) \|_{L^2(M)} &\leq \| (\Id - \Delta_g)^k u(x, \centerdot) \|_{L^2(M)} \\
& \leq \mu_g (M)^{1/2} \| (\Id - \Delta_g)^k u(x, \centerdot) \|_{L^\infty(M)}
\end{split}
\]
for all $x\in \R^n$. The fact that $u \in \mathcal{A}(\R^n \times M)$ implies 
that the function 
$x \mapsto \| (\Id - \Delta_g)^k u(x, \centerdot) \|_{L^\infty(M)}$ 
belongs to $L^2(\R^n)$. Hence, $(\Id - \Delta_g)^s u \in L^2(\R^n \times M) $ for 
all $u \in \mathcal{A}(\R^n \times M)$ and $s\in \R$.
Below, we will see that in fact
$(\Id - \Delta_g)^s u \in L^2(M; \dot{H}^t(\R^n))$ for all $t \in \R_+$, and we 
will provide a convenient way of computing its norm. Start by noticing that from the identity 
\eqref{id:composition_Bessel} we have that
\begin{equation}
\label{id:k-bessel_mixed}
\| (\Id - \Delta_g)^k \widehat{u}(\xi, \centerdot) \|_{L^2(M)}^2 = \sum_{j \in \N_0} (1 + \lambda_j)^{2k} \| \Pi_j \widehat{u}(\xi, \centerdot)\|_{L^2(M)}^2
\end{equation}
for all $\xi \in \R^n$ and all $k \in \N_0$. Then, by the inequality \eqref{in:bound_t-k}, we have that for all $t \in \R_+$ and $k \in \N_0$ there exists a constant $C>0$ such that
\[ \sum_{j \in \N_0} (1 + \lambda_j)^{2k} \int_{\R^n} |\xi|^{2t} \| \Pi_j \widehat{u}(\xi, \centerdot)\|_{L^2(M)}^2 \, \dd \xi \leq C. \]
In particular, from \eqref{id:Hs_projectors} we know that the non-negative functional
\[ u \in \mathcal{A}(\R^n \times M)
\longmapsto \Big( \int_{\R^n} |\xi|^{2t} \| \widehat{u}(\xi, \centerdot) \|_{H^s(M)}^2 \, \dd \xi \Big)^{1/2} \]
is well-defined for all $s \in \R$ and $t \in \R_+$.

\begin{lemma} \label{lem:new_norm} \sl For all $u \in \mathcal{A}(\R^n \times M)$ we have that $(\Id - \Delta_g)^s u \in L^2(M; \dot{H}^t(\R^n))$ for all $s \in \R$ and $t \in \R_+$. 
Moreover,
\[ \| (\Id - \Delta_g)^{s/2} u \|_{L^2(M; \dot{H}^t(\R^n))} = \Big( \int_{\R^n} |\xi|^{2t} \| \widehat{u}(\xi, \centerdot) \|_{H^s(M)}^2 \, \dd \xi \Big)^{1/2} \]
for all $u \in \mathcal{A}(\R^n \times M)$, $s \in \R$ and $t \in \R_+$.
\end{lemma}

\begin{proof}
Observe that $ \sum_{j = 0}^N (1 + \lambda_j)^{s/2} \Pi_j u $ belongs to 
$ \mathcal{A}(\R^n \times M)$ for an arbitrary 
$u \in \mathcal{A}(\R^n \times M)$, 
$N \in \N_0$ and $s \in \R$. Moreover, 
\begin{equation}
\label{term:sequence_partial_sum}
\Big\{ \sum_{j = 0}^N (1 + \lambda_j)^{s/2} \Pi_j u : N \in \N \Big\}
\end{equation}
is a Cauchy sequence in $L^2(M; \dot{H}^t(\R^n))$ for every $s \in \R$ and 
$t \in \R_+$. Indeed, since $\mathcal{F}(\Pi_j u) = \Pi_j \widehat{u}$, we have for every $L, N \in \N$ that
\begin{equation}
\label{id:norm_partial_sum}
\begin{split}
\Big\| \sum_{j = N}^{N+L} (1 + \lambda_j)^{s/2} \Pi_j  u \Big\|_{L^2(M; \dot{H}^t(\R^n))}^2 
& = \int_{\R^n} |\xi|^{2t} \Big\| \sum_{j = N}^{N+L} (1 + \lambda_j)^{s/2} \Pi_j \widehat{u} (\xi, \centerdot) \Big\|_{L^2(M)}^2 \, \dd \xi
\\
& = \int_{\R^n} |\xi|^{2t}  \sum_{j = N}^{N+L} (1 + \lambda_j)^s \|\Pi_j \widehat{u} (\xi, \centerdot)\|_{L^2(M)}^2 \, \dd \xi.
\end{split}
\end{equation}
From \eqref{id:k-bessel_mixed} we have for $k \in \N_0$ such that $k > s/2$ that
\[ \sum_{j = N}^{N+L} (1 + \lambda_j)^s \|\Pi_l \widehat{u} (\xi, \centerdot)\|_{L^2(M)}^2 \leq \| (\Id - \Delta_g)^k \widehat{u}(\xi, \centerdot) \|_{L^2(M)}^2. \]
By the fact that $\widehat{u} \in \mathcal{A}(\R^n \times M)$ we have that the 
function $\xi \mapsto |\xi|^t \| (\Id - \Delta_g)^k \widehat{u}(\xi, \centerdot) \|_{L^2(M)}$
belongs to $L^2(\R^n)$. Then, by the dominated convergence theorem we have that
\[ 
\lim_{N \to \infty} \int_{\R^n} |\xi|^{2t}  \sum_{j = N}^{N+L} (1 + \lambda_j)^s \|\Pi_j \widehat{u} (\xi, \centerdot)\|_{L^2(M)}^2 \, \dd \xi = 0  
 \]
uniformly in $L\in \N_0$, which implies that
\[ 
\lim_{N \to \infty} \Big\| \sum_{j = N}^{N+L} (1 + \lambda_j)^{s/2} \Pi_j u \Big\|_{L^2(M; \dot{H}^t(\R^n))} = 0 , 
\]
uniformly in $L\in \N_0$. Consequently the sequence \eqref{term:sequence_partial_sum} is  Cauchy in $L^2(M; \dot{H}^t(\R^n))$. Therefore, there is $v_s \in L^2(M; \dot{H}^t(\R^n))$ defined by the Cauchy sequence \eqref{term:sequence_partial_sum}. By a standard limiting argument applied to  \eqref{id:norm_partial_sum} with $N = 0$ and $L$ tending to infinity, we have that
\[
\| v_s \|_{L^2(M; \dot{H}^t(\R^n))} = \Big( \int_{\R^n} |\xi|^{2t} \| \widehat{u}(\xi, \centerdot) \|_{H^s(M)}^2 \, \dd \xi \Big)^{1/2}. 
\]
It remains to prove that $v_s = (\Id - \Delta_g)^{s/2} u$.

Right before this lemma we saw that 
$(\Id - \Delta_g)^{s/2} u \in L^2(\R^n \times M)$ since $u \in \mathcal{A}(\R^n \times M)$.
Furthermore, recall that for every $x \in \R^n$, we defined $(\Id - \Delta_g)^{s/2} u(x, \centerdot)$ as the equivalence class of
\[
\Big\{ \sum_{j = 0}^N (1 + \lambda_j)^{s/2} \Pi_j u(x, \centerdot) : N \in \N \Big\}
\]
which is a Cauchy sequence in $L^2(M)$ for every $x \in \R^n$. Hence, whenever $k \in \N_0$ with $k > s$ and $N \in \N_0$ we have that
\[
\begin{split}
\Big\| (\Id - \Delta_g)^{s/2} u(x, \centerdot) &- \sum_{j = 0}^N (1 + \lambda_j)^{s/2} \Pi_j u(x, \centerdot) \Big\|_{L^2(M)}
 \\
&\leq  \Big( \sum_{j = N + 1}^\infty (1 + \lambda_j)^{k} \|\Pi_j u(x, \centerdot)\|_{L^2(M)}^2 \Big)^{1/2} 
\\
&\leq \mu_g (M)^{1/2} \| (\Id - \Delta_g)^k u(x, \centerdot) \|_{L^\infty(M)}
\end{split}
\]
for all $x\in \R^n$. Since $u \in \mathcal{A}(\R^n \times M)$, 
the function $x \mapsto \| (\Id - \Delta_g)^k u(x, \centerdot) \|_{L^\infty(M)}$ 
belongs to $L^2(\R^n)$. From here one can derive that
\[ \lim_{N \to \infty} \int_{\R^n} \Big\| (\Id - \Delta_g)^{s/2} u(x, \centerdot) - \sum_{j = 0}^N (1 + \lambda_j)^{s/2} \Pi_j u(x, \centerdot) \Big\|_{L^2(M)} \, \dd x = 0. \]

At this point we are in the position to ensure that $v_s = (\Id - \Delta_g)^{s/2} u$. Indeed, on the one hand we know from \Cref{prop:distributionL2Ht} that $v_s$ determines a unique distribution in $\R^n \times M$ defined from the Cauchy sequence in \eqref{term:sequence_partial_sum}. On the other hand, this same Cauchy sequence converges to $(\Id - \Delta_g)^{s/2} u$ in $L^2(\R^n \times M)$. Since $(\Id - \Delta_g)^{s/2} u$ is in particular a distribution in $\R^n \times M$, we can conclude that $(\Id - \Delta_g)^{s/2} u$ is the distribution determined by $v_s$. This concludes the proof of this lemma.
\end{proof}

After \Cref{lem:new_norm} we know that the non-negative functional defined in \eqref{eq:functional}
is well defined for all $s \in \R$, $t \in \R_+$ and $u\in \mathcal \mathcal{A}(\R^n \times M) $.  

\begin{definition} \sl For $s \in \R$ and $t \in \R_+$, we define 
$H^s(M; \dot{H}^t(\R^n))$ as the Banach completion 
of the set $\mathcal{A}(\R^n \times M) $ endowed with the norm 
$ \| \centerdot \|_{H^s(M; \dot{H}^t(\R^n))}$.
\end{definition} 

We now can prove the following proposition.

%

\begin{proposition}\label{prop:distributionHsHt}\sl For $s \in \R$ and $t \in \R_+$ we have that every  $u \in H^s(M; \dot{H}^t(\R^n))$ determines a unique distribution in $\R^n \times M$.
\end{proposition}

\begin{proof}
Again this statement follows from 
\Cref{lem:distributions}. Recall that $\R^n \times M$ is 
smooth manifold that can be covered by an atlas 
$\{ (\R^n \times U_1, \id \otimes \varphi_1), \dots, (\R^n \times U_l, \id \otimes \varphi_l) \}$ 
with a finite number $l \in \N$ of charts, where 
$\{ (U_1, \varphi_1), \dots, (U_l, \varphi_l) \}$ was an atlas for the compact 
Riemannian manifold $M$. Then, with the choice of a smooth partition of unity 
$\{ \chi_1, \dots, \chi_l \}$ subordinated to the atlas for $M$, we set 
$\mathscr{V}(\R^n \times M) = \mathcal{A}(\R^n \times M)$ with the norm 
\[
 \| u \|_{\mathscr{V}(\R^n \times M)} = \sum_{j = 1}^l \bigg( \int_{\R^n} |\xi|^{2t} \| \chi_j\circ \varphi_j^{-1}(y) \widehat{u}(\xi, \varphi_j^{-1}(\centerdot)) \|_{H^s(\R^m)}^2 \, \dd \xi \bigg)^{1/2}. 
 \]
The only point left in order to conclude this  proposition from \Cref{lem:distributions} is to check that a sequence in $\mathscr{V}(\R^n \times M) = \mathcal{A}(\R^n \times M)$ is Cauchy with respect to the norm $\| \centerdot \|_{\mathscr{V}(\R^n \times M)}$ if and only if is  Cauchy with respect to $\| \centerdot \|_{H^s(M; \dot{H}^t(\R^n))}$. The fact that this is the case follows from \Cref{prop:norm_charts}.
\end{proof}

\appendix

\section{The special orthogonal group}\label{app:orthogonal_group}

In this appendix we collect some facts about the special orthogonal group 
$\SO(n)$ {in order to be able to prove \Cref{prop:SOn_Sn-1_Sobolev}}. Some of the properties stated here are well-known, other not so much.
For the sake of completeness, we have included the proofs of all the results 
collected.

Let $\M(n)$ denote the vector space of real $n \times n$ matrices. The 
general linear group
\[ \GL(n) = \{ Q \in \M(n) : \det\, Q \neq 0 \}\]
is an open subset of $\M(n)$ and consequently a smooth manifold of dimension
$n^2$. Endowed with the matrix multiplication, it becomes a Lie group. The 
standard coordinates on $\GL(n)$ are denoted by
$x^\alpha (Q) = e_{\alpha_1} \cdot Q e_{\alpha_2}$ with
$\alpha = (\alpha_1, \alpha_2)$ and $\alpha_j \in \{ 1, \dots, n \}$.
Its tangent space at a point $Q \in\GL(n)$ is defined as follows
\[T_Q \GL(n) = \vspan \{\partial_{x^\alpha}|_Q : \alpha=(\alpha_1,\alpha_2)\in\{1,\ldots,n\}^2 \}.\]
Every $X = \sum_\alpha X^\alpha \partial_{x^\alpha}$ in $T_Q GL(n)$ can be 
identified with a matrix whose components are $X^\alpha$ with 
$\alpha = (\alpha_1, \alpha_2)$,
$\alpha_j \in \{ 1, \dots, n \}$ and $\alpha_1$ enumerating the rows and $\alpha_2$ the columns.\footnote{The
summatory $\sum_{\alpha}$ corresponds to the summation of all
indices $\alpha = (\alpha_1, \alpha_2)$ with
$\alpha_j \in \{ 1, \dots, n \}$.} Thus, we identify $T_Q \GL(n)$ with
$\M(n)$, and we will use the same notation for the corresponding elements of 
$T_Q \GL(n)$ and $\M(n)$. The euclidean 
metric on $T_Q \GL(n)$ is defined by
\[
e (X, Y) = \tr (X^\T Y), \qquad X,Y \in T_Q \GL(n),
\]
where $\tr (X^\T Y)$ denotes the trace of the matrix $X^\T Y$. The 
orthogonal group
\[\O(n) = \{ Q \in \GL(n) : Q^\T Q = \id \}\]
is the level set $\Phi^{-1}(\id)$ of the function
$\Phi: \GL(n) \rightarrow \M(n)$ defined by $\Phi (Q)= Q^\T Q$. Thus,
$\O(n)$ is a compact Lie subgroup of $\GL(n)$ of dimension $n(n - 1)/2$. The 
special orthogonal group $\SO(n)$ is the connected component of $\O(n)$, 
consisting of matrices with positive determinant. Thus, it is also a compact 
manifold of dimension $n(n - 1)/2$ and becomes a Lie group under matrix 
multiplication. The tangent space of $\SO(n)$ at $Q$ is the vector space 
\[
T_Q \SO(n) = \{ X \in T_Q \GL(n) : X^T Q + Q^T X = 0 \}.
\]
On each tangent space $T_Q \SO(n)$ we consider the metric
\[ g (X, Y) = \frac{1}{2} \tr (X^\T Y), \qquad X,Y \in T_Q \SO(n). \]
Thus, $(\SO(n), g)$ becomes a compact Riemannian manifold.
The metric $g$ has some simple invariances that can be described by the
following three bijective mappings: For $P \in \SO(n)$, we define the 
left and right translations, as $L_P (Q) = PQ$ and 
$R_P (Q) = QP$ for all $Q \in \SO(n)$. We also define the transposition 
as $\T(Q) = Q^\T$ for all $Q \in \SO(n)$. 
Note that if $Q \in \SO(n)$ and $X \in T_Q \SO(n)$, then
the corresponding push-forwards satisfy:
\[
\begin{aligned}
& (L_P)_\ast X \in T_{PQ} \SO(n), & & (L_P)_\ast X = PX,
\\
& (R_P)_\ast X \in T_{QP} \SO(n), & & (R_P)_\ast X = XP,
\\
& \T_\ast X \in T_{Q^\T} \SO(n), & & T_\ast X = X^\T.
\end{aligned}
\]
Thus, given $Q \in \SO(n)$ and $X, Y \in T_Q \SO(n)$ we have the following 
\begin{align}
\label{id:g_left}
& [(L_P)^\ast g]_Q (X, Y) =  {\frac{1}{2}}\tr(X^\T P^\T P Y) = {\frac{1}{2}} \tr(X^\T Y) = g_Q(X, Y), \\
\label{id:g_right}
& [(R_P)^\ast g]_Q (X, Y) = {\frac{1}{2}} \tr(P^\T X^\T Y P) = {\frac{1}{2}} \tr(X^\T Y) = g_Q(X, Y), \\
\label{id:g_transposition}
& [\T^\ast g]_Q (X, Y) = {\frac{1}{2}} \tr(X Y^\T) = {\frac{1}{2}} \tr(X^\T Y) = g_Q(X, Y),
\end{align}
since $P \in \SO(n)$. This means that $g$ is left and right invariant, that 
is bi-invariant, and transposition invariant.

In the following lines we will construct the normalized 
Haar measure $\mu$ in $\SO(n)$ using the language of Riemannian geometry.
Since we are considering a very particular case, it is worth presenting a 
concrete construction. Consider 
\begin{equation}
\label{id:Ealpha}
E_{j k} =  \partial_{x^{j k}} - \partial_{x^{k j}}
\end{equation}
with
$j, k \in \{1, \dots, n\}$ for $j < k$.\footnote{The Greek letters are 
used for multi-indices, while the Latin ones are reserved for indices. Thus, 
$x^\alpha = x^{j k}$ if $\alpha = (j, k)$.}
Note that every $E_\alpha$ belongs to
$T_{\id} \SO(n)$ and $\{ E_\alpha : 1 \leq \alpha_1 < \alpha_2 \leq n \}$ is an
orthonormal basis of $T_{\id} \SO(n)$.
Choose an orientation for $T_{\id} \SO(n)$ so that the ordered basis
\[ \{ E_{1\, 2}, \dots, E_{1\, n}, E_{2\, 3}, \dots, E_{2\, n}, \dots, E_{n - 1\, n} \} \]
is positively oriented.
Consider the covector { $\varepsilon^{j k} = 2^{-1} (dx^{j k} - dx^{k j})$ }
in the cotangent space $T^\ast_\id \SO(n)$. The set 
$\{ \varepsilon^\alpha : 1 \leq \alpha_1 < \alpha_2 \leq n \}$ forms an 
orthonormal basis of $T^\ast_{\id} \SO(n)$. With these covectors we will 
define the Riemannian or Haar volume form $\nu$ at $Q$ as
\[
\nu_Q (X_1, \dots, X_{n(n - 1)/2}) = \big[(L_{Q^\T})^\ast \bigwedge_{\alpha_1 < \alpha_2} \varepsilon^\alpha \big] (X_1, \dots, X_{n(n - 1)/2}) 
\]
for $X_1, \dots, X_{n(n - 1)/2} \in T_Q \SO(n)$.
Here $(L_{Q^\T})^\ast \bigwedge_{\alpha_1 < \alpha_2} \varepsilon^\alpha$ 
denotes the pull-back of the $n(n-1)/2$-covector 
\[ \bigwedge_{\alpha_1 < \alpha_2} \varepsilon^\alpha = \varepsilon^{1\, 2} \wedge \cdots \wedge \varepsilon^{1\, n} \wedge \varepsilon^{2\, 3} \wedge \cdots \wedge \varepsilon^{2\, n} \wedge \cdots \wedge \varepsilon^{n - 1\, n} \]
through the mapping $L_{Q^\T}$. The quantity
$\varpi_n = \int_{\SO(n)} \nu $ is positive and, since $\SO(n)$ is compact, it is also finite. Thus, we define the normalized 
Haar measure $\mu$ as the representative of the functional
\[ f \in C(\SO(n)) \longmapsto \frac{1}{\varpi_n} \int_{\SO(n)} f \nu \in \C,\]
where $C(\SO(n))$ stands for the continuous functions on $\SO(n)$. The 
measure $\mu$ has several useful invariances with respect to the mappings 
$L_P$, $R_P$ and $\T$.

\begin{lemma}\label{lem:translation_transoposed_invariances}\sl We have that
the Haar volume form is left and right invariant ---bi-invariant--- and transposition invariant: 
\begin{align}
\label{id:left-translation}
& (L_P)^\ast \nu = \nu \qquad \forall\, P \in \SO(n),\\
\label{id:right-translation}
& (R_P)^\ast \nu = \nu \qquad \forall\, P \in \SO(n),\\
\label{id:transposition_invariance}
& \T^\ast \nu = \nu.
\end{align}
\end{lemma}

\begin{proof}
These identities follow from the invariances of the metric on $\SO(n)$. 
Recall that if $M$ and $N$ are oriented manifolds, $g$ is a Riemannian  
metric on $N$, and $F : M \rightarrow N$ is a smooth immersion, then the 
pull-back $F^\ast g$ is a Riemannian metric on $M$.
Furthermore, by the definition of Riemannian volume form
we have that
\[ \dd V_{F^\ast g} = F^\ast \dd V_g. \]
Since $L_P$, $R_P$ and $\T$ are smooth immersions, the volume forms generated 
by $(L_P)^\ast g $, $(R_P)^\ast g $ and $T^\ast g $ are $(L_P)^\ast \nu $, 
$(R_P)^\ast \nu $ and $T^\ast \nu $ respectively. Thus, by the 
invariances \eqref{id:g_left}, \eqref{id:g_right} and 
\eqref{id:g_transposition} for the metric $g$, we derive the invariances \eqref{id:left-translation}, \eqref{id:right-translation} and \eqref{id:transposition_invariance} for $\nu$.
\end{proof}

\begin{proposition}\label{prop:translation_transposed_invariance}\sl If $f \in L^1(\SO(n))$, then
\begin{equation}\label{id:translation}
\begin{split}
 \int_{\SO(n)} f(PQ) \, \dd \mu(Q) &= \int_{\SO(n)} f(Q) \, \dd \mu(Q) \quad \forall\, P \in \SO(n), 
\\
 \int_{\SO(n)} f(QP) \, \dd \mu(Q) &= \int_{\SO(n)} f(Q) \, \dd \mu(Q) \quad \forall\, P \in \SO(n), 
\\
 \int_{\SO(n)} f(Q^\T) \, \dd \mu(Q) &= \int_{\SO(n)} f(Q) \, \dd \mu(Q). 
\end{split}
\end{equation}
\end{proposition}

\begin{proof}
These identities {are} immediately {seen to hold} from 
\Cref{lem:translation_transoposed_invariances} and the definition of $\mu$ 
for functions $f \in C^\infty (\SO(n))$. The general case {follows} by density.
\end{proof}

The next invariance shows a relation between $\mu$ and the hypersurface 
measure on $\SS^{n - 1}$.

\begin{proposition}\label{prop:SOn_Sn-1}\sl {Let $F:\SO(n) \to \SS^{n-1}$ be given either by $F (Q) = Q e_1 $ or by $F (Q) = Q^\T e_1 $.} If $f \in L^1(\SS^{n-1})$ then
\[ \int_{\SO(n)} f(F(Q)) \, \dd \mu(Q) = \frac{1}{S(\SS^{n - 1})}  \int_{\SS^{n - 1}} f(\theta) \, \dd S(\theta), \]
where $S(\SS^{n - 1}) = \int_{\SS^{n - 1}} \, \dd S(\theta) $.
\end{proposition}

\begin{proof}
{By \Cref{prop:translation_transposed_invariance} it is enough to verify the case of $F (Q) = Q e_1$.}
Consider the smooth map $F(Q) = Q e_1$. Let $\mu_F$ denote the image measure 
of $\mu$ on $\SS^{n - 1}$, that is, given a Borel
$\Sigma \subset \SS^{n - 1}$, $ \mu_F (\Sigma) = \mu (F^{-1}(\Sigma)) $. 
The identity
\[ \int_{\SO(n)} f(Qe_1) \, \dd \mu(Q) = \int_{\SS^{n - 1}} f(\theta) \, \dd \mu_F(\theta) \]
holds trivially for simple functions, and can be extended by density. It 
remains to prove that $ \mu_F (\Sigma) = S(\Sigma)/S(\SS^{n - 1}) $ for 
every Borel $\Sigma \subset \SS^{n - 1}$. From the identity 
\eqref{id:translation} in
\Cref{prop:translation_transposed_invariance}, one can derive that
if $\Sigma$ is a Borel set of $\SS^{n - 1}$ then
$\mu_F (P \Sigma) = \mu_F (\Sigma)$ for every $P \in \SO(n)$. The same holds 
for $S (P \Sigma) = S (\Sigma)$. This common invariance means 
that $\mu_F$ and 
$S$ are {both} uniformly distributed on $\SS^{n - 1}$. Therefore, there exists a 
$c > 0$ such that $ \mu_F (\Sigma) = c S(\Sigma) $ for all Borel {sets} 
$\Sigma$  (see for example \cite[Theorem 3.4]{zbMATH00739280}). 
By testing the identity in $\SS^{n - 1}$ one gets 
$S(\SS^{n - 1}) = \mu_F(\SS^{n - 1})/c = \mu(\SO(n))/c = 1/c$ and this completes
the proof of this proposition.
\end{proof}

The next step is to transfer these invariances to Sobolev spaces. To do so
we first collect some invariances of the Laplace--Beltrami operator 
$\Delta_{\SO(n)}$.

\begin{lemma}\label{lem:L-B_invariances}\sl
Assume $f \in C^\infty (\SO(n))$ and for $P,Q \in \SO(n)$ set
\[(L_P)^\ast f (Q) = f (PQ), \quad  (R_P)^\ast f (Q) = f (QP)\quad \text{and} \quad
\T^\ast f(Q) = f(Q^\T).
\]
Then,
\begin{align}
\label{id:LB_left}
\Delta_{\SO(n)} (L_P)^\ast f (Q) &= \Delta_{\SO(n)}f (PQ) \qquad \forall\, P,Q \in \SO(n), \\
\label{id:LB_right}
\Delta_{\SO(n)} (R_P)^\ast f (Q) &= \Delta_{\SO(n)}f (QP) \qquad \forall\, P,Q \in \SO(n), \\
\label{id:LB_transposition}
\Delta_{\SO(n)} \T^\ast f (Q) &= \Delta_{\SO(n)}f (Q^\T) \qquad \forall\, Q \in \SO(n).
\end{align}
\end{lemma}

\begin{proof}
These identities follow again from the invariances of the metric on 
$\SO(n)$. As in the proof of \Cref{lem:translation_transoposed_invariances}, if $M$ and $N$ are manifolds, 
$g$ is a Riemannian metric on $N$, and $F : M \rightarrow N$ is a smooth 
immersion, then the pull-back $F^\ast g$ is a Riemannian metric on $M$, and 
the canonical Laplace--Beltrami operators $\Delta_{F^\ast g}$ on $M$ and 
$\Delta_g$ on $N$ satisfy the following relation
\[ \Delta_{F^\ast g} f \circ F (p) = \Delta_g f (F(p)) \qquad \forall\, p \in M,\, f\in C^\infty(N).
\]
Since $L_P$, $R_P$ and $\T$ are smooth immersions, 
the invariances \eqref{id:g_left}, \eqref{id:g_right} and 
\eqref{id:g_transposition} for the metric $g$ imply the invariances 
\eqref{id:LB_left}, \eqref{id:LB_right} and \eqref{id:LB_transposition} 
respectively.
\end{proof}

\begin{proposition}\label{prop:Sobolev_translation_transposed_invariance}\sl 
Let $f $ belong to $  H^s (\SO(n))$ with $s \in \R$. 
Then,
\[
\begin{split}
\| (L_P)^\ast f \|_{H^s(\SO(n))} &= \| f \|_{H^s(\SO(n))} \qquad \forall\, P \in \SO(n),
\\
\| (R_P)^\ast f \|_{H^s(\SO(n))} &= \| f \|_{H^s(\SO(n))}  \qquad \forall\, P \in \SO(n),
\\
\| \T^\ast f \|_{H^s(\SO(n))} &= \| f \|_{H^s(\SO(n))}. 
\end{split}
\]
\end{proposition}

\begin{proof}
{Assume that $f \in C^\infty (\SO(n))$. The general case follows by density}.
Since the three identities follow in the same way, let $F$ denote any of the 
mappings $L_P$, $R_P$ and $\T$. Recall that
\[\| F^\ast f \|_{H^s(\SO(n))} = \sum_{l \in \N_0} (1 + \lambda_l)^s \| \Pi_l F^\ast f \|_{L^2(\SO(n))}^2\]
with
\[\| \Pi_l F^\ast f \|_{L^2(\SO(n))}^2 = \sum_{m = 1}^{d_l} \Big| \int_{\SO(n)} \overline{\psi_l^m} F^\ast f \, \dd \mu \Big|^2. \]
By \Cref{prop:translation_transposed_invariance} we have
\begin{equation}
\label{id:projection}
\| \Pi_l F^\ast f \|_{L^2(\SO(n))}^2 = \sum_{m = 1}^{d_l} \Big| \int_{\SO(n)} \overline{(F^{-1})^\ast \psi_l^m} f \, \dd \mu \Big|^2.
\end{equation}
Then, \Cref{lem:L-B_invariances} implies that $(F^{-1})^\ast \psi_l^m$ is in
an eigenfunction for $\lambda_l$ and belongs to 
the eigenspace $H_l$, while
\Cref{prop:translation_transposed_invariance} ensures that 
$\{ (F^{-1})^\ast \psi_l^m : m = 1, \dots, d_l \}$ is an orthonormal basis 
of $H_l$. Thus, there exists a unitary matrix $A \in \U(d_l)$ with entries 
$A_{j k}$ such that
\[ (F^{-1})^\ast \psi_l^m = \sum_{j = 1}^{d_l} A_{m j} \psi_l^j. \]
Therefore, since $A$ is unitary
\[
\begin{split}
\sum_{m = 1}^{d_l} \Big| \int_{\SO(n)} \overline{(F^{-1})^\ast \psi_l^m} f \, \dd \mu \Big|^2 &= \sum_{m = 1}^{d_l} \Big| \sum_{j = 1}^{d_l} \overline{A_{m j}} \int_{\SO(n)} \overline{\psi_l^j} f \, \dd \mu \Big|^2\\
&= \sum_{j = 1}^{d_l} \Big| \int_{\SO(n)} \overline{\psi_l^j} f \, \dd \mu \Big|^2 = \| \Pi_l f \|_{L^2(\SO(n))}^2
\end{split}
\]
By the last chain of equalities together with \eqref{id:projection} we have 
that 
$$\| \Pi_l F^\ast f \|_{L^2(\SO(n))} = \| \Pi_l f \|_{L^2(\SO(n))},$$
and consequently that 
$$\| F^\ast f \|_{H^s(\SO(n))} = \| f \|_{H^s(\SO(n))}.$$
This proves the proposition.
\end{proof}

Finally, we will extend the invariance of \Cref{prop:SOn_Sn-1}, first to 
Laplace--Beltrami operators and then to the Sobolev norms.

{To state} the lemma below we {recall} that $\Delta_{\SS^{n - 1}}$ denotes the 
Laplace--Beltrami operator on $\SS^{n - 1}$.

\begin{lemma}\label{lem:SOn_Sn-1_LB}\sl
Let $f$ be in $C^\infty (\SS^{n - 1})$ and set $F (Q) = Q e_1 $ for every $Q \in \SO(n)$. Then,
\[ 
\Delta_{\SO(n)}  [f \circ F] (Q) = \Delta_{\SS^{n - 1}} f(Q e_1) \qquad \forall\,  Q\in \SO(n).
\]
\end{lemma}

This lemma can be deduced from {the rather abstract} 
\cite[Theorem 2.7]{zbMATH05960760}.
However, for the sake of completeness, we provide here a full proof {in the 
specific case of $\SO(n)$ and $\SS^{n-1}$.}

\begin{proof}
Start by observing that it is enough to prove the identity in the statement for 
$Q= \id$ since then \eqref{id:LB_left} and the invariance under rotations of 
$\Delta_{\SS^{n - 1}} $ imply the general result.
Then, we only need to show that
\[ 
\Delta_{\SO(n)}  [f \circ F] (\id) = \Delta_{\SS^{n - 1}} f(e_1),
\]
In order to prove this identity we will use a convenient way of 
writing the Laplace--Beltrami operator $\Delta_g$ on a Riemannian manifold $M$ of 
dimension $m$. Assume $f \in C^\infty (M)$ and $p \in M$; if 
$\{ X_1, \dots, X_m \}$ is a local frame around the point $p$ and 
$\{ X_1|_p, \dots, X_m|_p \}$ is an orthonormal basis of $T_p M$, then
\begin{equation}
\label{id:LB}
\Delta_g f (p) = \sum_{j = 1}^m \big( X_j (X_j f) - (\nabla_{X_j} X_j) f \big) (p),
\end{equation}
where $\nabla$ denotes the covariant derivative on $M$.

In order to write the Laplace--Beltrami on $\SO(n)$, we 
define the vector field $X_\alpha$ given, at every point $Q \in \SO(n)$, by 
$X_\alpha|_Q = Q E_\alpha \in T_Q \SO(n) $ with $E_\alpha$ as in 
\eqref{id:Ealpha} and $\alpha_1 < \alpha_2$. Since 
$\{ E_\alpha : 1 \leq \alpha_1 < \alpha_2 \leq n \}$ is an orthonormal basis 
of $T_{\id} \SO(n)$, we have that 
$\{ X_\alpha : 1 \leq \alpha_1 < \alpha_2 \leq n \}$ is 
an orthonormal global frame. Hence, for every $Q \in \SO(n)$, we have that
\[ \Delta_{\SO(n)} [f \circ F] (Q) = \sum_{\alpha_1 < \alpha_2} \big( X_\alpha (X_\alpha [f \circ F]) - (\nabla_{X_\alpha} X_\alpha) [f \circ F] \big) (Q). \]
{Since the $X_\alpha$ are left invariant vector fields and the metric is bi-
invariant it follows that  
$ \sum_{\alpha_1 < \alpha_2} \nabla_{X_\alpha} X_\alpha = 0$ ---see for example 
\cite[Proposition 4.2]{geis} for a proof. Then,
\begin{equation} \label{id:SOnLB_0}
\Delta_{\SO(n)} [f \circ F] (Q) = \sum_{\alpha_1 < \alpha_2} X_\alpha (X_\alpha [f \circ F])(Q). 
\end{equation}
We will use this formula only for $Q = \id$. In order to compute the right-hand 
side {of \eqref{id:SOnLB_0},
we introduce} curves $\gamma_{\alpha}:(-\pi/2,\pi/2) \to \SO(n)$
with {$\alpha=(\alpha_1,\alpha_2)$ and }$1 \le \alpha_1<\alpha_2\le n$ {such that}  $ \gamma_{\alpha}(0) = \id $ {and}  
$ \frac{d \phantom{t}}{dt} \gamma_{\alpha}|_{t=0} = E_\alpha = X_\alpha|_{\id}$.
We define the curves $\gamma_\alpha$ by giving the explicit matrix components:
\[ 
\{\gamma_{\alpha}(t)\}_{\beta_1,\beta_2} =  
\begin{cases} 
 \cos(t), \; \phantom{-0i} \text{if} \quad  \beta_1= \beta_2 = \alpha_1 \text{ or } \beta_1= \beta_2 = \alpha_2, \\
 \sin(t),  \;  \phantom{-0o} \text{if} \quad  \beta_1=\alpha_1, \beta_2 = \alpha_2, \\
 - \sin(t), \; \phantom{o0} \text{if} \quad \beta_1=\alpha_2, \beta_2 = \alpha_1, \\
   1, \; \phantom{- o\sin(t)} \text{if} \quad   \beta_1 = \beta_2  \text{ and } \beta_1 \neq \alpha_1,\alpha_2, \\
  0, \; \phantom{- o\sin(t)} \text{otherwise}. \\
\end{cases} 
\]
Notice that the element  $\gamma_\alpha(t) \in \SO(n)$  is just a rotation of $t$ radians in the $2$ dimensional plane  determined by $e_{\alpha_1} $ and  $e_{\alpha_2}$ in $\R^n$ that fixes the rest of the basis vectors: $\gamma_\alpha(t) e_k = e_k $ if $k \neq \alpha_1,\alpha_2$.
One can also verify that $X_\alpha|_{\gamma_\alpha(t)} = \frac{d \phantom{t}}{dt}\gamma_\alpha(t)$, {so that $X_\alpha|_{\gamma(t)} \, g = \frac{d \phantom{t}}{dt} g(\gamma_{\alpha}(t))$ 
for any smooth $g$ defined around $\gamma(t)$}.  Also, since  $f \circ F  (\gamma_\alpha(t)) = f(\gamma_\alpha(t)e_1)$ we have that  $\gamma_\alpha(t)e_1 = e_1$ for all $t \in (-\pi/2,\pi/2)$ and {$\alpha$ with} $\alpha_1>1$.  Using this in  \eqref{id:SOnLB_0} yields  
\begin{equation}
\begin{split}
\label{id:SOnLB}
\Delta_{\SO(n)} [f \circ F] (\id) &=\sum_{\alpha_1<\alpha_2}  \left. \frac{d^2 \phantom{t}}{dt^2} \big[ f \circ F  (\gamma_\alpha(t) )\big] \right|_{t=0} 
\\
&= \sum_{ \alpha_2 = 2}^n  \left. \frac{d^2 \phantom{t}}{dt^2}f\left (\gamma_{(1,\alpha_2)}(t)e_1 \right) \right|_{t=0} . 
\end{split}
\end{equation}

Now, we use \eqref{id:LB} for the Laplace--Beltrami operator on $\SS^{n - 1}$.
We start by noting that every point in a neighbourhood of $e_1$ can be written
as $\gamma_{(1,2)} (t_2) \dots \gamma_{(1, n)} (t_n) e_1$ with
$(t_2, \dots, t_n) \in (-\pi/2,\pi/2)^{n - 1}$. Then, in that neighbourhood, we 
define the local frame $\{ Y_2, \dots, Y_n \}$ given by the vector fields
\[Y_j|_{\gamma_{(1,2)} (t_2) \dots \gamma_{(1, n)} (t_n) e_1} = \gamma_{(1,2)} (t_2) \dots \gamma_{(1, n)} (t_n) e_j.\]
Note that $\{ Y_2|_{e_1}, \dots, Y_n|_{e_1} \}$ is an orthonormal basis of $T_{e_1} \SS^{n - 1}$. For $2 \le j \le n$, introduce the curve $\widetilde{\gamma}_j : (-\pi/2,\pi/2) \to \SS^{n-1}  $ given by
\[ \widetilde{\gamma}_j(t) = \gamma_{(1,j)}(t)e_1 = (\cos t) e_1 {-} (\sin t) e_{j} .\]
Thus,
\[
Y_j|_{\widetilde{\gamma}_j(t)} = Y_j|_{\gamma_{(1,j)} (t) e_1} = \gamma_{(1, j)} ({t}) e_j = (\sin t) e_1 + (\cos t) e_{j} , 
\] 
and consequently,
\begin{equation}
\label{id:Yjcurve}
Y_j|_{\widetilde{\gamma}_j(t)} = {-} \frac{d \phantom{t}}{dt} \widetilde{\gamma}_{j}(t), \quad 2 \le j \le n.
\end{equation}
By the identity \eqref{id:LB} we have that
\[\Delta_{\SS^{n - 1}} f(e_1) = \sum_{j = 2}^n \big( Y_j (Y_j f) - (\nabla_{Y_j} Y_j) f \big) (e_1). \]
Then, the identity \eqref{id:Yjcurve}, together with the fact {that the  curves
$ \widetilde{\gamma}_j $} are maximal circles (i.e. geodesics), implies
\[\Delta_{\SS^{n - 1}} f(e_1) {= \sum_{j=2} ^n Y_j(0)[Y_j(t)f(\tilde \gamma(t))]\Big|_{t=0}} = \sum_{ j = 2}^n   \frac{d^2 \phantom{t}}{dt^2}f\left (\widetilde{\gamma}_j(t) \right) \Big |_{t=0}.\]
%
Then, returning to \eqref{id:SOnLB} we can directly conclude that
\[
\begin{split}
\Delta_{\SO(n)} [f \circ F] (\id) &= \sum_{ \alpha_2 = 2}^n  \frac{d^2 \phantom{t}}{dt^2}f\left (\gamma_{(1,\alpha_2)}(t)e_1 \right) \Big |_{t=0}
\\ 
& = \sum_{ j = 2}^n  \frac{d^2 \phantom{t}}{dt^2}f\left (\widetilde{\gamma}_j(t) \right) \Big |_{t=0} = \Delta_{\SS^{n - 1}} f(e_1). 
\end{split}
\]
This yields the desired result.}
\end{proof} 

{We conclude this appendix by proving  \Cref{prop:SOn_Sn-1_Sobolev}, a further invariance for the Sobolev spaces {on $\SO(n)$}.} 
 
\begin{proof}[Proof of \Cref{prop:SOn_Sn-1_Sobolev}]
{By the invariances stated in 
\Cref{prop:Sobolev_translation_transposed_invariance} it is enough to prove the 
case $v=e_1$ and $F(Q)= F_{e_1}(Q)= Q(e_1)$.}  On the one hand, let 
$\{ Y_k^m : k \in \N_0,\, m=1,\dots,d^\prime_k \}$ 
be an orthonormal basis of eigenfunctions of $L^2(\SS^{n - 1})$ with the 
corresponding eigenvalues $\{ \lambda^\prime_k : k \in \N_0 \}$ satisfying
$\lambda^\prime_0 = 0$, $ \lambda^\prime_k < \lambda^\prime_{k+1} $, and 
$-\Delta_{\SS^{n - 1}} Y_k^m = \lambda^\prime_k Y_k^m$. Set 
$H^\prime_k = \vspan \{ Y_k^1,\dots, Y_k^m \}$ and consider the projector 
$\Pi^\prime_k$ on $H^\prime_k$ defined in \Cref{sec:sobolev}.
On the other hand, let $\{ \psi_l^m : l \in \N_0,\, m=1,\dots,d_l \}$
be an orthonormal basis of $L^2(\SO(n))$ formed by the eigenfunctions
$\{ \psi_l^m : l \in \N_0,\, m=1,\dots,d_l \}$ whose eigenvalues are
$\{ \lambda_l : l \in \N_0 \}$ and satisfy
$\lambda_0 = 0$, $ \lambda_l < \lambda_{l+1} $, and 
$-\Delta_{\SO(n)} \psi_l^m = \lambda_l \psi_l^m$. Set 
$H_l = \vspan \{ \psi_l^1,\dots, \psi_l^m \}$ and the projector $\Pi_l$ on $H_l$ defined as in \Cref{sec:sobolev}.

By \Cref{prop:SOn_Sn-1}, the sequence
$\{ \phi_k^m : k \in \N_0,\, m=1,\dots,{ d^\prime _k} \}$ with $ \phi_k^m(Q) = S(\SS^{n - 1})^{1/2} Y_k^m (Q e_1)$ is an orthonormal 
set of $L^2(\SO(n))$.  Moreover, \Cref{lem:SOn_Sn-1_LB} implies that
\[ \Delta_{\SO(n)} \phi_k^m = {\lambda^\prime_k} \phi_k^m, \]
which tells us that
\[
\begin{split}
& \{ \lambda^\prime_k : k \in \N_0 \} \subset \{ \lambda_l : l \in \N_0 \},\\  & \lambda^\prime_k = \lambda_l \Longrightarrow \vspan \{ \phi_k^m : m=1,\dots,{d^\prime _k} \} \subset H_l.
\end{split}
\]
We will prove the proposition for $f\in C^\infty(\SS^{n - 1})$; the general case follows by density. By \Cref{prop:SOn_Sn-1} we have
\[ 
f (\theta) = S(\SS^{n - 1})^{1/2} \sum_{k \in \N_0} \sum_{{ m= 1}}^{{d_k ^\prime}} \Big( \int_{\SO(n)} \overline{\phi_k^m} f\circ F \, \dd \mu \Big)  Y_k^m(\theta), 
\]
which implies that $f\circ F$ is totally described by the orthonormal set
$\{ \phi_k^m : k \in \N_0,\, m=1,\dots,{d^\prime _k}  \}$. Thus, if given 
$\lambda_l$ there exists a $\lambda^\prime_k$ such that $\lambda^\prime_k =  \lambda_l$ , then
\[
\Pi_l [f \circ F] (Q) = \Pi^\prime_k f (Qe_1) \qquad \forall Q\in \SO(n),
\]
and \Cref{prop:SOn_Sn-1} implies that
\[ \| \Pi^\prime_k f \|^2_{L^2(\SS^{n - 1})} = S(\SS^{n - 1}) \| F^\ast \Pi^\prime_k f \|^2_{L^2(\SO(n))} = S(\SS^{n - 1}) \| \Pi_l f \circ F \|^2_{L^2(\SO(n))} \]
with $F^\ast \Pi^\prime_k f (Q) = \Pi^\prime_k f (Qe_1) $. 
Furthermore, if $\lambda_l \notin \{ \lambda^\prime_k : k \in \N_0 \}$ then $\Pi_l f \circ F (Q) = 0$ for all $Q \in \SO(n)$.
Thus, we can conclude that
\[ \sum_{l \in \N_0} (1 + \lambda_l)^s \| \Pi_l f \circ F \|^2_{L^2(\SO(n))} = S(\SS^{n - 1})^{-1}  \sum_{k \in \N_0} (1 + \lambda^\prime_k)^s \| \Pi^\prime_k f \|^2_{L^2(\SS^{n - 1})}. \]
This completes the proof of the proposition.
\end{proof}

\bibliographystyle{plain}
\bibliography{references}{}

\section*{Affiliation of the authors}
\end{document}